\documentclass[11pt]{article}

\usepackage{mathptmx}       
\usepackage{helvet}         
\usepackage{courier}        
\usepackage{type1cm}        
\usepackage{color}                            

\usepackage{graphicx}        
\usepackage{multicol}        
\usepackage[bottom]{footmisc}

\usepackage{epstopdf}
\usepackage{verbatim}
\usepackage{algorithm,algorithmic}
\usepackage{url}

\usepackage[]{amsthm, amsmath, amssymb, amsfonts, amsbsy, latexsym}

\newcommand{\bitem}{\begin{itemize}}
\newcommand{\eitem}{\end{itemize}}
\newcommand{\beq}{\begin{equation}}
\newcommand{\eeq}{\end{equation}}

\newcommand{\cC}{{\cal C}}

\newcommand{\cS}{{\cal S}}
\newcommand{\cR}{{\cal R}}

\newcommand{\cL}{{\cal L}}

\newcommand{\bZ}{{\mathbb Z}}

\newcommand{\bR}{{\mathbb R}}
\newcommand{\bC}{{\mathbb C}}
\newcommand{\bN}{{\mathbb N}}

\newcommand{\ox}{\omega_1}
\newcommand{\oy}{\omega_2}

\newcommand{\CC}{\mathbb{C}}

\newcommand{\Sp}{\mbox{supp }}
\newcommand{\Id}{\mbox{Id}}

  \newcommand{\R}{\mathbb{R}}
 
 \newcommand{\C}{\mathbb{C}}
 \newcommand{\Z}{\mathbb{Z}}

\newcommand{\ip}[2]{\left\langle#1,#2\right\rangle}

\newtheorem{theorem}{Theorem}[section]

\newtheorem{proposition}{Proposition}[section]
\newtheorem{lemma}{Lemma}[section]
\newtheorem{definition}{Definition}[section]

\newtheorem{measure}{Measure} 

\newcommand{\wl}[1]{{\color{blue}[{#1}]}}

\numberwithin{equation}{section}

\title{Digital Shearlet Transforms}
\author{Gitta Kutyniok, Wang-Q Lim, and Xiaosheng Zhuang}
\date{}
\begin{document}
\maketitle
\abstract{Over the past years, various representation systems which sparsely approximate functions governed by
anisotropic features such as edges in images have been proposed. We exemplarily mention the systems of contourlets,
curvelets, and shearlets. Alongside the theoretical development of these systems, algorithmic realizations of the
associated transforms were provided. However, one of the most common shortcomings of these frameworks is the lack
of providing a unified treatment of the continuum and digital world, i.e., allowing a digital theory to be a
natural digitization of the continuum theory. In fact, shearlet systems are the only systems so far which satisfy
this property, yet still deliver optimally sparse approximations of cartoon-like images. In this chapter, we
provide an introduction to digital shearlet theory with a particular focus on a unified treatment of the
continuum and digital realm. In our survey we will present the implementations of two shearlet transforms, one
based on band-limited shearlets and the other based on compactly supported shearlets. We will moreover
discuss various quantitative measures, which allow an objective comparison with other directional transforms and
an objective tuning of parameters. The codes for both presented transforms as well as the framework for quantifying
performance are provided in the Matlab toolbox \url{ShearLab}.
}

\section{Introduction}

One key property of wavelets, which enabled their success as a universal methodology for signal processing, is the
unified treatment of the continuum and digital world. In fact, the wavelet transform can be implemented by a natural
digitization of the continuum theory, thus providing a theoretical foundation for the digital transform. Lately,
it was observed that wavelets are however suboptimal when sparse approximations of 2D functions are {sought}. The reason
is that these functions are typically governed by anisotropic features such as edges in images or evolving shock
fronts in solutions of transport equations. However, Besov models -- which wavelets optimally encode -- are clearly
deficient to capture these features.
Within the model of cartoon-like images, introduced by Donoho in \cite{Don99}
in 1999, the suboptimal behavior of wavelets for such 2D functions was made mathematically precise; see also
Chapter \cite{SparseApproximation}.

Among the various directional representation systems which have since then been proposed such as contourlets \cite{DV05},
curvelets \cite{CD04}, and shearlets, the shearlet system is in fact the only one which delivers optimally sparse
approximations of cartoon-like images and still also allows for a unified treatment of the continuum and digital world.
One main reason in comparison to the other two mentioned systems is the fact that shearlets are affine systems,
thereby enabling an extensive theoretical framework, but parameterize directions by slope (in contrast to angles)
which greatly supports treating the digital setting. As a thought experiment just note that a shear matrix leaves
the digital grid $\Z^2$ invariant, which is in general not true for rotation.

This raises the following questions, which we will answer in this chapter:
\begin{enumerate}
\item[(P1)] What are the main desiderata for a digital shearlet theory?
\item[(P2)] Which approaches do exist to derive a natural digitization of the continuum shearlet theory?
\item[(P3)] How can we measure the accuracy to which the desiderata from (P1) are matched?
\item[(P4)] Can we even introduce a framework within which different directional transforms can be objectively
compared?
\end{enumerate}

Before delving into a detailed discussion, let us first contemplate about these questions on a more intuitive level.

\subsection{A Unified Framework for the Continuum and Digital World}

Several desiderata come to one's mind, which guarantee a unified framework for both the continuum and digital world,
and provide an answer to (P1). The following are the choices of desiderata which were considered in \cite{KSZ11,DKSZ11}:

\begin{itemize}
\item {\em Parseval Frame Property.} The transform shall ideally have the tight frame property, which enables taking the
adjoint as inverse transform. This property can be broken into the following two parts, which most, but not
all, transforms admit:\\[-0.6cm]
\bitem
\item[$\diamond$] {\em Algebraic Exactness.} The transform should be based on a theory for digital data in the sense that the
analyzing functions should be an exact digitization of the continuum domain analyzing elements.
\item[$\diamond$] {\em Isometry of Pseudo-Polar Fourier Transform.} If the image is first mapped into a different domain -- here
the pseudo-polar domain --, then this map should be an isometry.\\[-0.6cm]
\eitem
\item {\em Space-Frequency-Localization.} The analyzing elements of the associated transform should ideally be
highly localized in space and frequency -- to the extent to which uncertainty principles allow this.
\item {\em Shear Invariance.} Shearing naturally occurs in digital imaging, and it can -- in contrast to rotation --
be precisely realized in the digital domain. Thus the transform should be shear invariant, i.e., a shearing of the input
image should be mirrored in a simple shift of the transform coefficients.
\item {\em Speed.} The transform should admit an al\-go\-ri\-thm of or\-der $O(N^2 \log N)$ flops, where $N^2$ is the number
of digital points of the input image.
\item {\em Geometric Exactness.} The transform should preserve geometric properties parallel to those of the continuum theory,
for example, edges should be mapped to edges in transform domain.
\item {\em {Stability.}} The transform should be resilient against impacts such as (hard) thresholding.
\end{itemize}

\subsection{Band-Limited Versus Compactly Supported Shearlet Transforms}

In general, two different types of shearlet systems are utilized today: Band-limited shearlet systems and compactly supported
shearlet systems (see also Chapters \cite{Introduction} and \cite{SparseApproximation}). Regarding those from an algorithmic
viewpoint, both have their particular advantages and disadvantages:

Algorithmic realizations of the {\em band-limited shearlet transform} have on the one hand typically a
higher computational complexity due to the fact that the windowing takes place in frequency domain. However, on the other hand,
they do allow a high localization in frequency domain which is important, for instance, for handling seismic data. Even more, band-limited
shearlets do admit a precise digitization of the continuum theory.

In contrast to this, algorithmic realizations of the {\em compactly
supported shearlet transform} are much faster and have the advantage of achieving a high accuracy in spatial domain. But for
a precise digitization one has to lower one's sights slightly. A more comprehensive answer to (P2) will be provided in the sequel
of this chapter, where we will present the digital transform based on band-limited shearlets introduced in \cite{KSZ11} and the
digital transform based on compactly supported shearlets from \cite{Lim2010}.

\subsection{Related Work}

Since the introduction of directional representation systems by many pioneer researchers (\cite{CD99,CD04,CD05a,CD05b,DV05}),
various numerical implementations of their directional representation systems have been proposed. Let us next briefly survey
the main features of the two closest to shearlets, which are the contourlet and curvelet algorithms.

\bitem
\item {\em Curvelets} \cite{CDDY06}. The discrete curvelet transform is implemented in the software package {\em CurveLab}, which
comprises two different approaches. One is based on unequispaced FFTs, which are used to interpolate the function in the frequency
domain on different tiles with respect to different orientations of curvelets. The other is based on frequency wrapping, which wraps
each subband indexed by scale and angle into a fixed rectangle around the origin. Both approaches can be realized efficiently in
$O(N^2\log N)$ flops with $N$ being the image size. The disadvantage of this approach is the lack of an associated continuum domain
theory.
\item {\em Contourlets} \cite{DV05}. The implementation of contourlets is based on a directional filter bank, which produces a
directional frequency partitioning similar to the one generated by curvelets. The main advantage of this approach is that it allows
a tree-structured filter bank implementation, in which aliasing due to subsampling is allowed to exist. Consequently, one can
achieve great efficiency in terms of redundancy and good spatial localization. A drawback of this approach is that various
artifacts are introduced and that an associated continuum domain theory is missing.
\eitem
Summarizing, all the above implementations of directional representation systems have their own advantages and disadvantages; one
of the most common shortcomings is the lack of providing a unified treatment of the continuum and digital world.

Besides the shearlet implementations we will present in this chapter, we would like to refer to Chapter
\cite{Applications} for a discussion of the algorithm in \cite{ELL08} based on the Laplacian pyramid scheme and directional
filtering. It should be though noted that this implementation is not focussed on a natural digitization of the continuum theory,
{which is a crucial aspect} of the work presented in the sequel. We further
would like to draw the reader's attention to Chapter \cite{ShearletMRA} which is based on \cite{KS08} aiming at introducing
a shearlet MRA from a subdivision perspective. Finally, we should mention that a different approach to a shearlet MRA was
recently undertaken in \cite{HKZ10}.

\subsection{Framework for Quantifying Performance}

A major problem with many computation-based results in applied mathematics is the non-availability of an accompanying code,
and the lack of a fair and objective comparison with other approaches. The first problem can be overcome by following the
philosophy of `reproducible research' \cite{DMSSU08} and making the code publicly available with sufficient documentation.
In this spirit, the shearlet transforms presented in this chapter are all downloadable from \url{http://www.shearlab.org}.
One approach to overcome the second obstacle is the provision of a carefully selected set of prescribed performance measures
aiming to prohibit a biased comparison on isolated tasks such as denoising and compression of specific standard images like
`Lena', `Barbara', etc. It seems far better from an intellectual viewpoint to carefully decompose performance according to
a more insightful array of tests, each one motivated by a particular well-understood property we are trying to obtain.
In this chapter we will present such a framework for quantifying performance specifically of implementations
of directional transforms, which was originally introduced in \cite{KSZ11,DKSZ11}. We would like to emphasize that such a
framework does not only provide the possibility of a fair and thorough comparison, but also enables the tuning of the parameters
of an algorithm in a rational way, thereby providing an answer to both (P3) and (P4).

\subsection{ShearLab}

Following the philosophy of the previously detailed thoughts, \url{ShearLab}\footnote{ShearLab (Version 1.1) is available from
\url{http://www.shearlab.org}.} was introduced by Donoho, Shahram, and the authors. This software package contains
\bitem
\item An algorithm based on band-limited shearlets introduced in \cite{KSZ11}.
\item An algorithm based on compactly supported {separable} shearlets introduced in \cite{Lim2010}.
\item {An algorithm based on compactly supported nonseparable shearlets introduced in \cite{Lim2011}.}
\item A comprehensive framework for quantifying performance of directional representations in general.
\eitem
This chapter is also devoted to provide an introduction to and discuss the mathematical foundation of these components.

\subsection{Outline}

In Section~\ref{sec:dsh-ppft}, we introduce and analyze the fast digital shearlet transform FDST, which is based on
band-limited shearlets. Section \ref{sec:dst} is then devoted to the presentation and discussion of the digital separable
shearlet transform DSST and the digital non-separable shearlet transform DNST. The framework of performance measures for
parabolic scaling based transforms is provided in Section \ref{sec:framework}. In the same section, we further discuss these measures
 for the special cases of the three previously introduced transforms.

\section{Digital Shearlet Transform using Band-Limited Shearlets}
\label{sec:dsh-ppft}

The first algorithmic realization of a digital shearlet transform we will present, coined {\em Fast
Digital Shearlet Transform (FDST)}, is based on band-limited shearlets.
Let us start by defining the class of shearlet systems we are interested in. Referring to Chapter \cite{Introduction},
we will consider the cone-adapted discrete shearlet system $SH(\phi,\psi,\tilde{\psi};\Delta,\Lambda,\tilde{\Lambda}) =
\Phi(\phi;\Delta) \cup \Psi(\psi;\Lambda) \cup \tilde{\Psi}(\tilde{\psi};\tilde{\Lambda})$ with $\Delta = \Z^2$ and
\[
\Lambda = \tilde{\Lambda} = \{(j,k,m) : j \ge 0,  |k| \leq 2^j, m \in \Z^2\}.
\]
We wish to emphasize that this choice relates to a scaling by $4^j$ yielding an integer valued parabolic scaling matrix,
which is better adapted to the digital setting than a scaling by $2^j$. We further let $\psi$ be a classical shearlet ($\tilde{\psi}$
likewise with $\tilde{\psi}(\xi_1,\xi_2) = \psi(\xi_2,\xi_1)$), i.e.,
\beq \label{eq:psidef}
\hat{\psi}(\xi) = \hat{\psi}(\xi_1,\xi_2) = \hat{\psi}_1(\xi_1) \, \hat{\psi}_2(\tfrac{\xi_2}{\xi_1}),
\eeq
where $\psi_1 \in L^2(\bR)$ is a wavelet with $\hat{\psi}_1 \in C^\infty(\mathbb{R})$ and supp $\hat{\psi}_1 \subseteq [-4,-\frac14]
\cup [\frac14,4]$, and $\psi_2 \in L^2(\mathbb{R})$ a `bump' function satisfying $\hat{\psi}_2 \in C^\infty(\mathbb{R})$ and
supp $\hat{\psi}_2 \subseteq [-1,1]$. We remark that the chosen support deviates slightly from the choice in the introduction,
which is however just a minor adaption again to prepare for the digitization. Further, recall the definition of the cones
$\cC_{11}$ -- $\cC_{22}$ from Chapter \cite{Introduction}.

The digitization of the associated discrete shearlet transform will be performed in the frequency domain. Focussing, on the cone
$\cC_{21}$, say, the discrete shearlet transform is of the form
\begin{equation}\label{def:csht}
f \mapsto \langle f,\psi_\eta\rangle =\langle\hat{f},\hat{\psi}_\eta\rangle
=\Big\langle \hat{f},2^{-j\tfrac32}\hat\psi(S_k^TA_{4^{-j}}\cdot)e^{2\pi i \ip{A_{4^{-j}}S_km}{\cdot}} \Big\rangle,
\end{equation}
where $\eta=(j,k,m,\iota)$ indexes {\em scale} $j$, {\em orientation} $k$, {\em position} $m$, and {\em cone} $\iota$. Considering
this shearlet transform for continuum domain data (taking all cones into account) implicitly induces a trapezoidal tiling of
frequency space which is evidently not cartesian. A digital grid perfectly adapted to this situation is the so-called `pseudo-polar
grid', which we will introduce and discuss subsequently in detail. Let us for now mention that this viewpoint enables representation of the
discrete shearlet transform as a cascade of three steps:

\bitem
\item[1)] Classical Fourier transformation  and change of variables to pseudo-polar coordinates.
\item[2)] Weighting by a radial `density com\-pen\-sa\-tion' factor.
\item[3)] Decomposition into rectangular tiles and inverse Fourier transform of each tiles.
\eitem

Before discussing these steps in detail, let us give an overview of how these steps will be faithfully digitized. First, it will be shown in
Subsection \ref{subsec:dsh-ppft}, that the two operations in Step 1) can be combined to the so-called pseudo-polar Fourier transform.
An oversampling in radial direction of the pseudo-polar grid, on which the pseudo-polar Fourier transform is computed, will then
enable the design of `density-compensation-style' weights on those grid points leading to Steps 1) \& 2) being an isometry. This will
be discussed in Subsection \ref{subsec:weights}. Subsection \ref{subsec:DSHOnPPGrid} is then concerned with the digitization of the
discrete shearlets to subband windows. Notice that a digital analog of \eqref{def:csht} moreover requires an additional 2D-iFFT.
Thus, concluding the digitization of the discrete shearlet transform will cascade the following steps, which is the
exact analogy of the continuum domain shearlet transform \eqref{def:csht}:
\begin{enumerate}
\item[(S1)] PPFT: Pseudo-polar Fourier transform with oversampling factor in the radial direction.
\item[(S2)] Weighting: Multiplication by `density-compensation-style' weights.
\item[(S3)] Windowing: Decomposing the pseudo-polar grid into rectangular subband windows with additional 2D-iFFT.
\end{enumerate}
With a careful choice of the weights and subband windows, this transform is an isometry. Then the inverse transform can be computed
by merely taking the adjoint in each step. A final discussion on the FDST will be presented in Subsection \ref{subsec:FDST}.


\subsection{Pseudo-Polar Fourier Transform}
\label{subsec:dsh-ppft}

We start by discussing Step (S1).

\subsubsection{Pseudo-Polar Grids with Oversampling}
\label{subsubsec:grid}

In \cite{ACDIS08}, a fast pseudo-polar Fourier transform (PPFT) which evaluates the discrete
Fourier transform at points on a trapezoidal grid in frequency space, the so-called pseudo-polar grid, was already developed. However, the
direct use of the PPFT is problematic, since it is -- as defined in \cite{ACDIS08} -- not an isometry. The main obstacle is the
highly nonuniform arrangement of the points on the pseudo-polar grid. This intuitively suggests to downweight points in regions
of very high density by using weights which correspond roughly to the density compensation weights underlying the continuous
change of variables. This will be enabled by a sufficient radial oversampling of the pseudo-polar grid.

This new pseudo-polar grid, which we will denote in the sequel by $\Omega_R$ to indicate the oversampling rate $R$, is defined by
\beq \label{eq:OmegaR}
\Omega_R = \Omega_R^1 \cup \Omega_R^2,
\eeq
where
\begin{eqnarray} \label{eq:OmegaR1}
\Omega_R^1 & = &  \{(-\tfrac{2n}{R}\cdot\tfrac{2\ell}{N},\tfrac{2n}{R}) : -\tfrac{N}{2} \le \ell \le \tfrac{N}{2}, \, -\tfrac{RN}{2} \le n \le \tfrac{RN}{2}\},\\
\label{eq:OmegaR2}
\Omega_R^2 & = &  \{(\tfrac{2n}{R},-\tfrac{2n}{R}\cdot\tfrac{2\ell}{N}) : -\tfrac{N}{2} \le \ell \le \tfrac{N}{2}, \, -\tfrac{RN}{2} \le n \le \tfrac{RN}{2}\}.
\end{eqnarray}
This grid is illustrated in Fig.~\ref{fig:PPgridR}.
\begin{figure}[h]
\begin{center}
\includegraphics[height=1.2in]{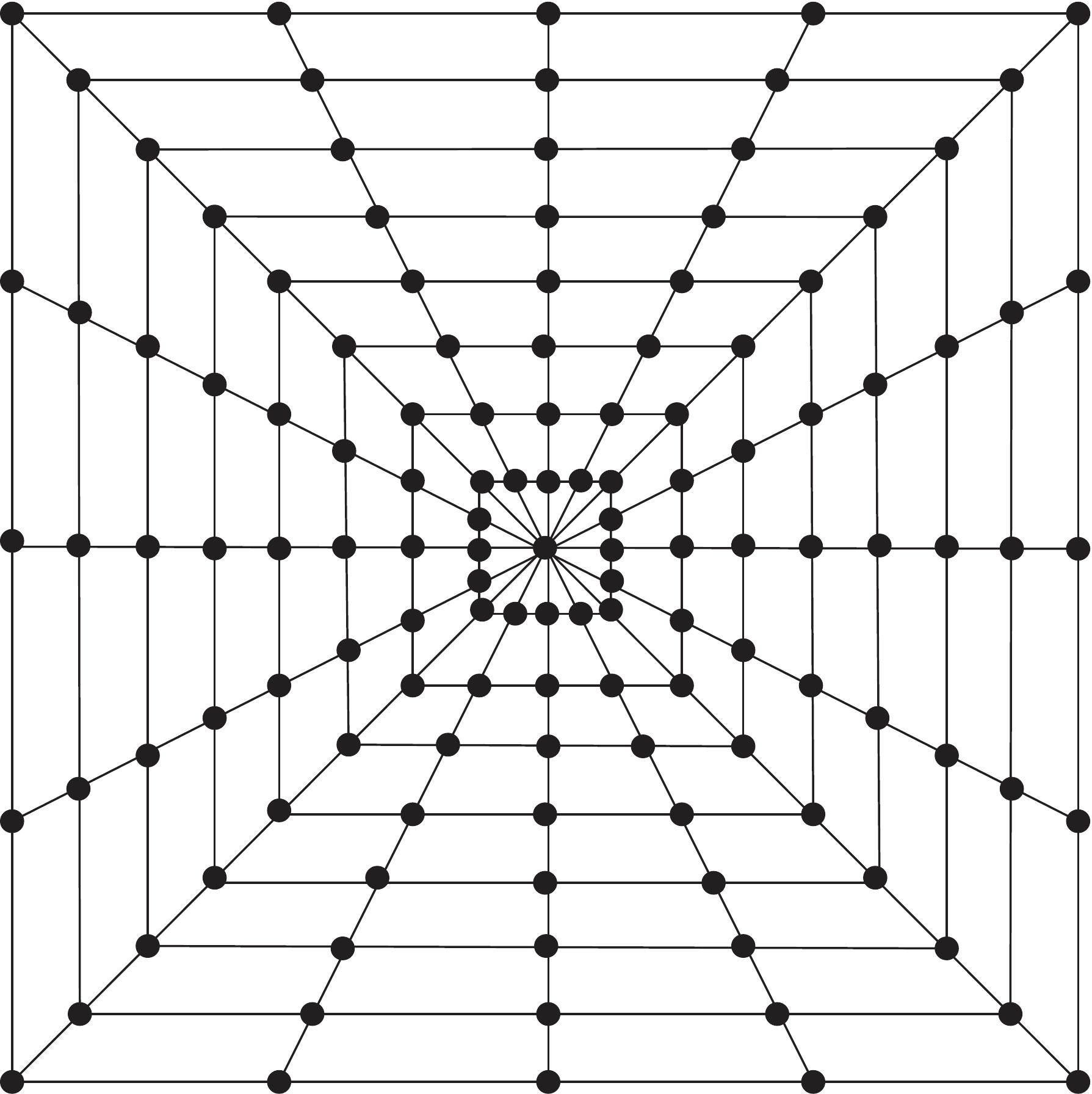}
\put(-48,-17){$\Omega_R$}
\hspace*{1cm}
\includegraphics[height=1.2in]{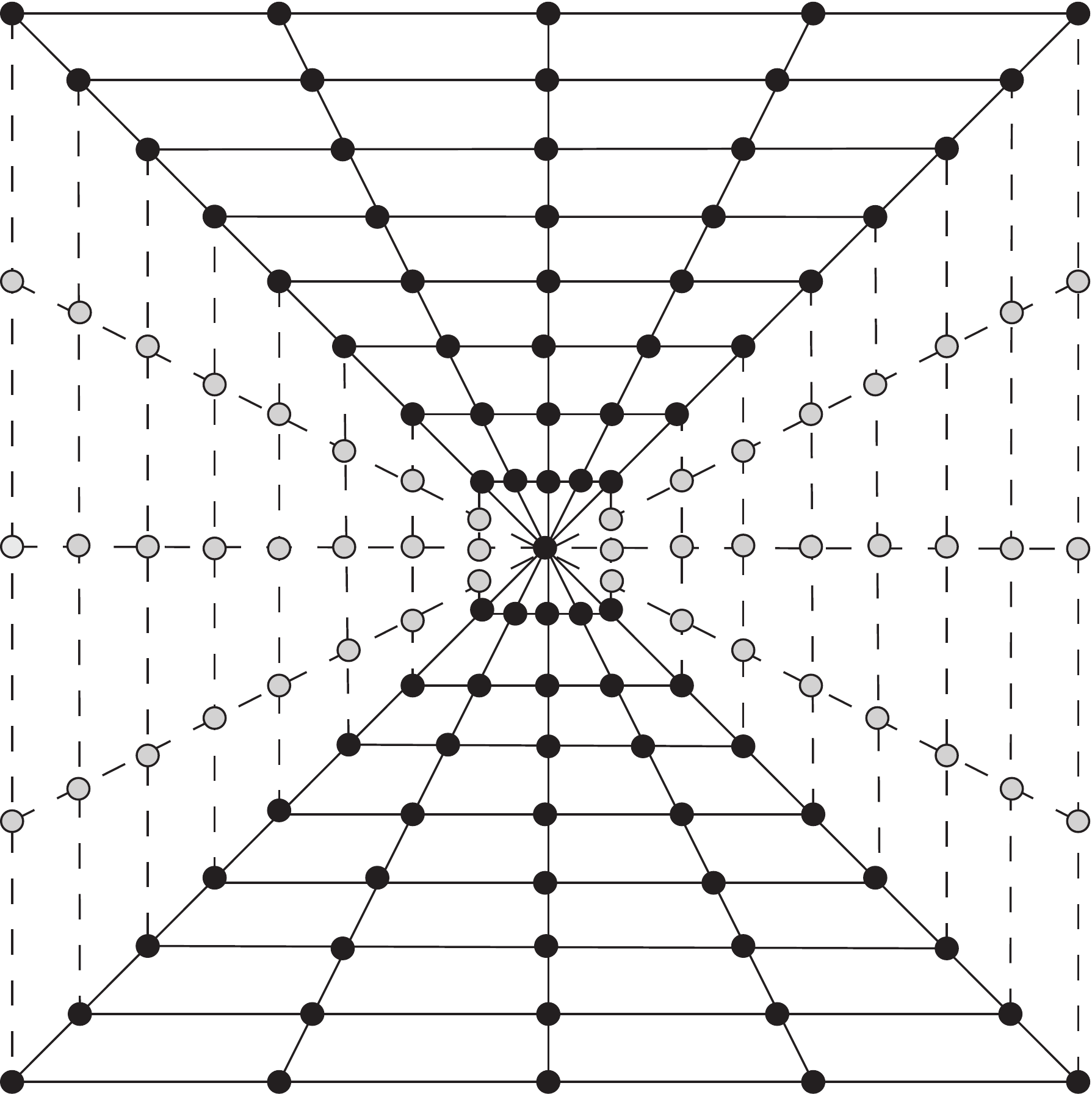}
\put(-48,-17){$\Omega^1_R$}
\hspace*{1cm}
\includegraphics[height=1.2in]{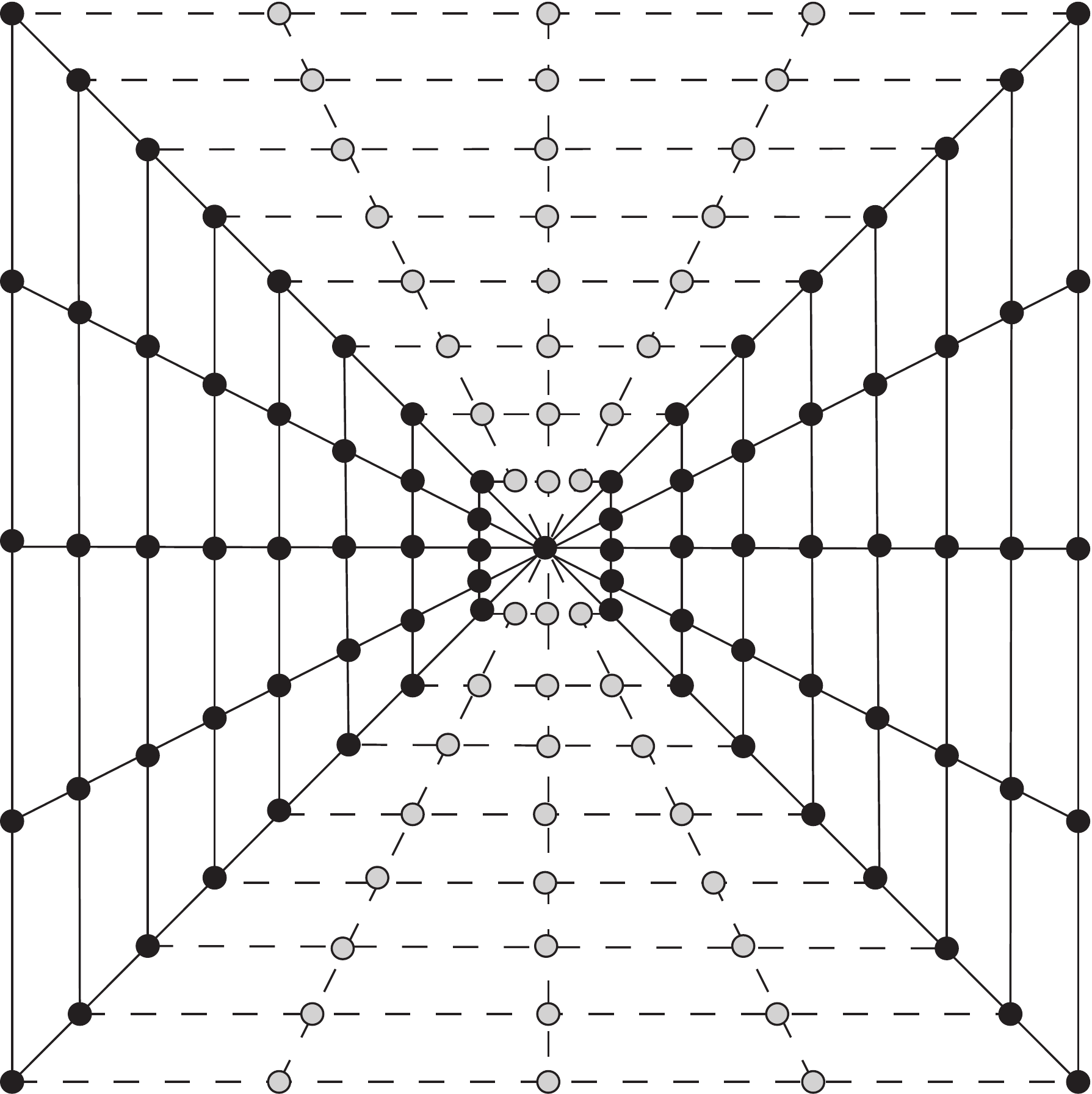}
\put(-48,-17){$\Omega^2_R$}
\end{center}
\caption{The pseudo-polar grid $\Omega_R=\Omega_R^1\cup\Omega_R^2$ for $N=4$ and $R=4$.}
\label{fig:PPgridR}
\end{figure}
We remark that the pseudo-polar grid introduced in \cite{ACDIS08} coincides with $\Omega_R$ for the particular choice $R=2$. It should be
emphasized that $\Omega_R = \Omega_R^1 \cup \Omega_R^2$ is not a disjoint partitioning, nor is the mapping $(n,\ell) \mapsto
(-\tfrac{2n}{R}\cdot\tfrac{2\ell}{N},\tfrac{2n}{R})$ or $(\tfrac{2n}{R},-\tfrac{2n}{R}\cdot\tfrac{2\ell}{N})$ injective.
In fact,  the center
\beq \label{eq:00}
\cC = \{(0,0)\}
\eeq
appears $N+1$ times in $\Omega_R^1$ as well as $\Omega_R^2$, and the points on the seam lines
\begin{eqnarray*}
\cS_R^1 & = &  \{(-\tfrac{2n}{R},\tfrac{2n}{R}) : -\tfrac{RN}{2} \le n \le \tfrac{RN}{2}, \, n \neq 0\},\\
\cS_R^2 & = &  \{(\tfrac{2n}{R},-\tfrac{2n}{R}) : -\tfrac{RN}{2} \le n \le \tfrac{RN}{2}, \, n \neq 0\}.
\end{eqnarray*}
appear in both $\Omega_R^1$ and $\Omega_R^2$.

\begin{definition}
\label{defi:ppft}
Let $N, R$ be positive integer, and let $\Omega_R$ be the pseudo-polar grid given by \eqref{eq:OmegaR}. For an $N\times N$ image
$I:=\{I(u,v) : -\tfrac{N}{2}\le u,v \le \tfrac{N}{2}-1\}$, the {\em pseudo-polar Fourier transform (PFFT) $\hat{I}$} of $I$ evaluated
on $\Omega_R$ is then defined to be
\[
\hat{I}(\ox,\oy) = \sum_{u,v=-N/2}^{N/2-1}I(u,v)e^{-\frac{2\pi i}{m_0}(u\ox+v\oy)}, \quad (\ox,\oy)\in\Omega_R,
\]
where $m_0 \ge N$ is an integer.
\end{definition}

We wish to mention that $m_0 \ge N$ is typically set to be $m_0 = \tfrac{2}{R}(RN+1)$ for computational reasons (see also \cite{ACDIS08}),
but we for now allow more generality.


\subsubsection{Fast PPFT}
\label{subsubsec:forward}

It was shown in \cite{ACDIS08}, that the PPFT can be realized in $O(N^2\log N)$ flops with  $N\times N$ being the size of the input image.
We will now discuss how the extended pseudo-polar Fourier transform as defined in Definition \ref{defi:ppft} can be computed with similar
complexity.

For this, let $I$ be an image of size $N\times N$. Also, $m_0$ is set -- but not restricted -- to be $m_0 = \frac{2}{R}(RN+1)$; we will
elaborate on this choice at the end of this subsection. We now focus on $\Omega_R^1$, and mention that the PPFT on the other cone
can be computed similarly. Rewriting the pseudo-polar Fourier transform from Definition \ref{defi:ppft}, for $(\ox,\oy)
= (-\tfrac{2n}{R}\cdot\tfrac{2\ell}{N},\tfrac{2n}{R}) \in  \Omega_R^1$, we obtain
\begin{eqnarray} \nonumber
\hat{I}(\ox,\oy)
&=&\sum_{u,v=-N/2}^{N/2-1}I(u,v)e^{-\frac{2\pi i}{m_0}(u\ox+v\oy)}
=\sum_{u=-N/2}^{N/2-1}\sum_{v=-N/2}^{N/2-1}I(u,v)e^{-\frac{2\pi i}{m_0}(u\frac{-4n\ell}{RN}+v\frac{2n}{R})}\\ \label{eq:PPFTrewrite}
&=& \sum_{u=-N/2}^{N/2-1}\left(\sum_{v=-N/2}^{N/2-1}I(u,v)e^{-\frac{2\pi i vn}{RN+1}}\right)e^{-{2\pi i u\ell}\cdot \frac{-2n}{(RN+1) \cdot N}}.
\end{eqnarray}
This rewritten form, i.e., \eqref{eq:PPFTrewrite}, suggests that the pseudo-polar Fourier transform $\hat{I}$ of $I$ on $\Omega_R^1$ can be obtained by performing the 1D FFT
on the extension of $I$ along direction $v$ and then applying a fractional Fourier transform (frFT) along direction $u$.
To be more specific, we require the following operations:

{\em Fractional Fourier Transform.} For $c\in\bC^{N+1}$, the {\em (unaliased) discrete
fractional Fourier transform by $\alpha \in \bC$} is defined to be
\[
(F_{N+1}^\alpha c)(k) :=\sum_{j=-N/2}^{N/2} c(j)e^{-2\pi i \cdot j\cdot k \cdot\alpha}, \quad k = -\tfrac{N}{2}, \ldots, \tfrac{N}{2}.
\]
It was shown in \cite{Bailey:frFT}, that the fractional Fourier transform $F_{N+1}^{\alpha}c$ can be computed using
$O(N\log N)$ operations. For the special case of $\alpha = 1/(N+1)$, the fractional Fourier transform becomes the
(unaliased) 1D discrete Fourier Transform (1D FFT), which in the sequel will be denoted by $F_1$. Similarly,
the 2D discrete Fourier Transform (2D FFT) will be denoted by $F_2$, and the inverse of the $F_2$ by $F_2^{-1}$ (2D iFFT).

{\em Padding Operator}. For $N$ even, $m>N$ an odd integer, and $c\in\bC^{N}$, the {\em padding operator} $E_{m,n}$
gives a symmetrically zero padding version of $c$ in the sense that
\[
(E_{m,N}c)(k) =
\begin{cases}
c(k) & k=-\tfrac{N}{2}, \ldots, \tfrac{N}{2}-1,\\
0     & k\in\{-\tfrac{m}{2}, \ldots, \tfrac{m}{2}\}\setminus\{-\tfrac{N}{2}, \ldots, \tfrac{N}{2}-1\}.
\end{cases}
\]

Using these operators, \eqref{eq:PPFTrewrite} can be computed by
\begin{eqnarray*}
\hat{I}(\ox,\oy)
&=& \sum_{u=-N/2}^{N/2-1} F_1 \circ E_{RN+1,N} \circ I(u,n) e^{-{2\pi i u\ell}\cdot \frac{-n}{(RN+1) \cdot N/2}}\\
&=& \sum_{u=-N/2}^{N/2} E_{N+1,N} \circ F_1 \circ E_{RN+1,N} \circ I(u,n) e^{-{2\pi i u\ell}\cdot \frac{-2n}{(RN+1) \cdot N}}\\
&=& (F_{N+1}^{\alpha_n}\tilde I(\cdot,n))(\ell),
\end{eqnarray*}
where $\tilde{I} = E_{N+1,N} \circ F_1 \circ E_{RN+1,N} \circ I \in \bC^{(RN+1)\times (N+1)}$ and $\alpha_n=-\frac{n}{(RN+1)N/2}$.
Since the 1D FFT and 1D frFT require only $O(N\log N)$ operations for a vector of size $N$, the total complexity of this algorithm
for computing the pseudo-polar Fourier transform from Definition \ref{defi:ppft} is indeed $O(N^2\log N)$ for an image of size
$N\times N$.

We would like to also remark that for a different choice of constant $m_0$, one can  compute the pseudo-polar Fourier transform
also with complexity $O(N^2\log N)$ for an image of size $N\times N$. This however requires application of the fractional Fourier
transform in both directions $u$ and $v$ of the image, which results in a larger constant for the computational cost; see also
\cite{Bailey:frFT}.


\subsection{Density-Compensation Weights}
\label{subsec:weights}

Next we tackle Step (S2), which is more delicate than it might seem, since the weights will not be derivable from simple density
compensation arguments.

\subsubsection{A Plancherel Theorem for the PPFT}

For this, we now aim to choose weights $w : \Omega_R \to \bR^+$ so that the extended PPFT from Definition
\ref{defi:ppft} becomes an isometry, i.e.,
\beq \label{eq:ppPlancherel}
\sum_{u, v = -N/2}^{N/2-1} |I(u,v)|^2
= \sum_{(\ox, \oy) \in \Omega_R} \hspace*{-0.3cm} w(\ox,\oy) \cdot |\hat{I}(\ox,\oy)|^2.
\eeq
Observing the symmetry of the pseudo-polar grid, it seems natural to select weight functions $w$ which have full axis symmetry
properties, i.e., for all $(\ox,\oy)\in\Omega_R$, we require
\begin{equation}\label{w:full-axis-symmetry}
w(\ox, \oy) = w(\oy, \ox),\; w(\ox, \oy) = w(-\ox, \oy),\; w(\ox, \oy) = w(\ox, -\oy).
\end{equation}
Then the following `Plancherel theorem' for the pseudo-polar Fourier transform on $\Omega_R$ -- similar to the one for the
Fourier transform  on the cartesian grid -- can be proved.

\begin{theorem}[\cite{KSZ11}]\label{thm:weight}
Let $N$ be even, and let $w : \Omega_R \to \bR^+$ be a weight function satisfying \eqref{w:full-axis-symmetry}. Then
\eqref{eq:ppPlancherel} holds,  if and only if,  the weight function $w$ satisfies
\begin{eqnarray}\nonumber
\delta(u,v) & =  &w(0,0) \\\nonumber &+&4 \cdot \sum_{\ell = 0, N/2} \sum_{n = 1}^{RN/2}
w(\tfrac{2n}{R}, \tfrac{2n}{R}\cdot\tfrac{-2\ell}{N}) \cdot \cos(2\pi  u\cdot \tfrac{2n}{m_0R})\cdot \cos(2\pi v \cdot \tfrac{2n}{m_0R}\cdot\tfrac{2\ell}{N})\\
\label{cond:isometry}
&  +& 8 \cdot \sum_{\ell = 1}^{N/2-1} \sum_{n = 1}^{RN/2}
w(\tfrac{2n}{R}, \tfrac{2n}{R}\cdot\tfrac{-2\ell}{N})  \cdot \cos(2\pi  u\cdot \tfrac{2n}{m_0R})\cdot \cos(2\pi v \cdot \tfrac{2n}{m_0R}\cdot\tfrac{2\ell}{N})
\end{eqnarray}
for all $-N+1 \le u, v \le N-1$.
\end{theorem}

\begin{proof}
We start by computing the right hand side of \eqref{eq:ppPlancherel}:
\begin{eqnarray*}
\lefteqn{\sum_{(\ox, \oy) \in \Omega_R} w(\ox,\oy) \cdot |\hat{I}(\ox,\oy)|^2}\\
& = & \sum_{(\ox, \oy) \in \Omega_R} w(\ox,\oy)
\cdot \left|\sum_{u, v = -N/2}^{N/2-1}  I(u,v) e^{-\frac{2\pi i}{m_0}(u \ox + v \oy)}\right|^2\\
& = & \sum_{(\ox, \oy) \in \Omega_R} w(\ox,\oy)
\cdot \left[\sum_{u, v = -N/2}^{N/2-1} \sum_{u', v' = -N/2}^{N/2-1} \hspace*{-0.18cm} I(u,v) \overline{I(u',v')}
e^{-\frac{2\pi i}{m_0}((u-u') \ox + (v-v') \oy)}\right]\\
& = & \sum_{(\ox, \oy) \in \Omega_R} w(\ox,\oy) \cdot  \sum_{u, v = -N/2}^{N/2-1} |I(u,v)|^2\\
& & + \sum_{\stackrel{u, v, u', v' = -N/2}{(u,v) \neq (u',v')}}^{N/2-1} I(u,v) \overline{I(u',v')} \cdot
\left[ \sum_{(\ox, \oy) \in \Omega_R} w(\ox,\oy) \cdot e^{-\frac{2\pi i}{m_0}((u-u') \ox + (v-v') \oy)}
\right].
\end{eqnarray*}
Choosing $I=c_{u_1,v_1}\delta{(u-u_1,v-v_1)}+c_{u_2,v_2}\delta{(u-u_2,v-v_2)}$ for all $-N/2\le u_1,$ $v_1,$ $u_2,$ $v_2\le N/2-1$
and for all $c_{u_1,v_1},c_{u_2,v_2}\in\CC$, we can conclude that \eqref{eq:ppPlancherel} holds if and only if
\[
 \sum_{(\ox, \oy) \in \Omega_R} w(\ox,\oy) \cdot e^{-\frac{2\pi i}{m_0}(u\ox + v\oy)} =
  \delta(u,v),\quad -N+1 \le u, v \le N-1.
\]
By the symmetry of the weights \eqref{w:full-axis-symmetry}, this is equivalent to
\begin{equation}\label{eq:isometry1}
\sum_{(\ox, \oy) \in \Omega_R} w(\ox,\oy) \cdot
[\cos(\tfrac{2\pi}{m_0} u \ox)\cos(\tfrac{2\pi}{m_0} v \oy)] = \delta(u,v)
\end{equation}
for all $-N+1\le u,v\le N-1$. From this, we can deduce that \eqref{eq:isometry1} is equivalent to \eqref{cond:isometry},
which proves the theorem.
\qed
\end{proof}

Notice that \eqref{cond:isometry} is a linear system with $RN^2/4+RN/2+1$ unknowns and $(2N-1)^2$ equations, wherefore,
in general, one needs the oversampling factor $R$ to be at least $16$ to enforce solvability.

\subsubsection{Relaxed Form of Weight Functions}
\label{subsubsec:relaxed}

The computation of the weights satisfying Theorem~\ref{thm:weight} by solving the full linear system of equations \eqref{cond:isometry}
is much too complex. Hence, we relax the requirement for exact isometric weighting,
and represent the weights in terms of undercomplete basis functions on the pseudo-polar grid.

More precisely, we first choose a set of basis functions $w_1, \ldots, w_{n_0}:\Omega_R\rightarrow \bR^+$ such that
\[
\sum_{j=1}^{n_0} w_j(\ox,\oy)\neq 0 \quad \mbox{for all } (\ox,\oy)\in\Omega_R.
\]
We then represent weight functions $w:\Omega_R\rightarrow\bR^+$ by
\beq \label{eq:compute_w}
w:=\sum_{j=1}^{n_0} c_jw_j,
\eeq
with $c_1,\ldots,c_{n_0}$ being nonnegative constants. This approach now enables solving \eqref{cond:isometry} for the
constants $c_1,\ldots,c_{n_0}$ using the least squares method, thereby reducing the computational complexity significantly.
The `full' weight function $w$ is then given by \eqref{eq:compute_w}.

We next present two different choices of weights which were derived by this relaxed approach. Notice that $(\ox, \oy)$
and $(n,\ell)$ will be used interchangeably.

{\bf Choice 1}. The set of basis functions $w_1, \ldots, w_5$ is defined as follows:\\
{\em Center}:
\[
w_1 = 1_{(0,0)},
\]
{\em Boundary}:
\[
w_2=1_{\{(\ox, \oy) : |n|=NR/2,\, \ox=\oy\}} \mbox{ and } w_3=1_{\{(\ox, \oy) : |n|=NR/2,\, \ox\neq\oy\}},
\]
{\em Seam lines}:
\[
w_4=|n| \cdot 1_{\{(\ox, \oy) : 1\le|n|<NR/2,\, \ox=\oy\}},
\]
{\em Interior}:
\[
w_5=|n| \cdot 1_{\{(\ox, \oy) : 1\le|n|<NR/2,\,\ox\neq\oy\}}.
\]

{\bf Choice 2}. The set of basis functions $w_1, \ldots, w_{N/2+2}$ is defined as follows:\\
{\em Center}:
\[
w_1 = 1_{(0,0)},
\]
{\em Radial Lines}:
\[
w_{\ell+2}=1_{\{(\ox, \oy) :  1<|n|<NR/2,\, \oy=\frac{\ell}{N/2}\ox\}},\quad \ell=0,1,\ldots,N/2.
\]

The associated weight functions are displayed in Fig.~\ref{fig:001}. In general, suitable weight functions usually obey
the pattern of linearly increasing values along the radial direction. Thus, this is a natural requirement for the basis functions.
\begin{figure}[ht]
\begin{center}
\includegraphics[height=1.2in]{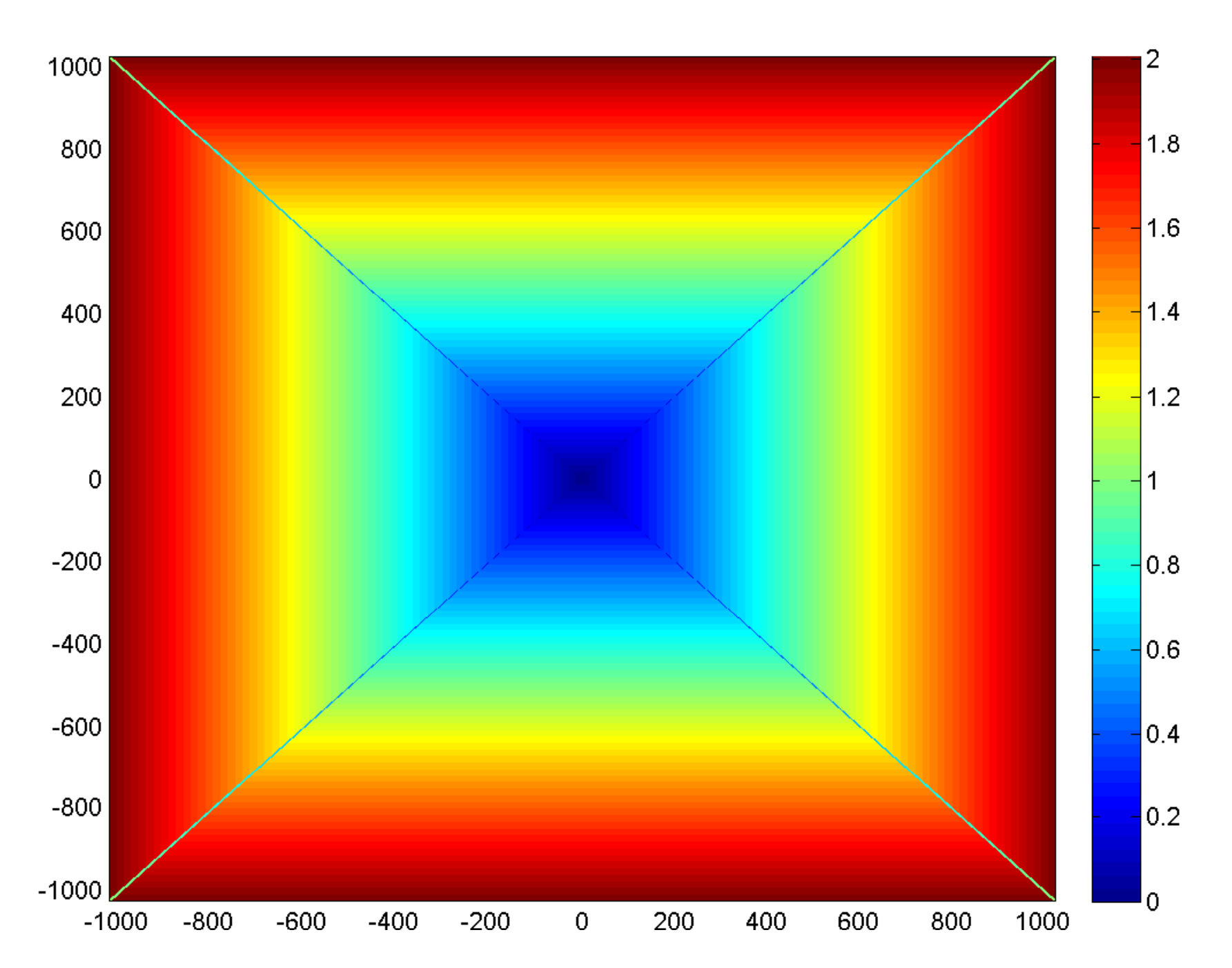}
\put(-70,-17){Choice 1}
\hspace*{1cm}
\includegraphics[height=1.2in]{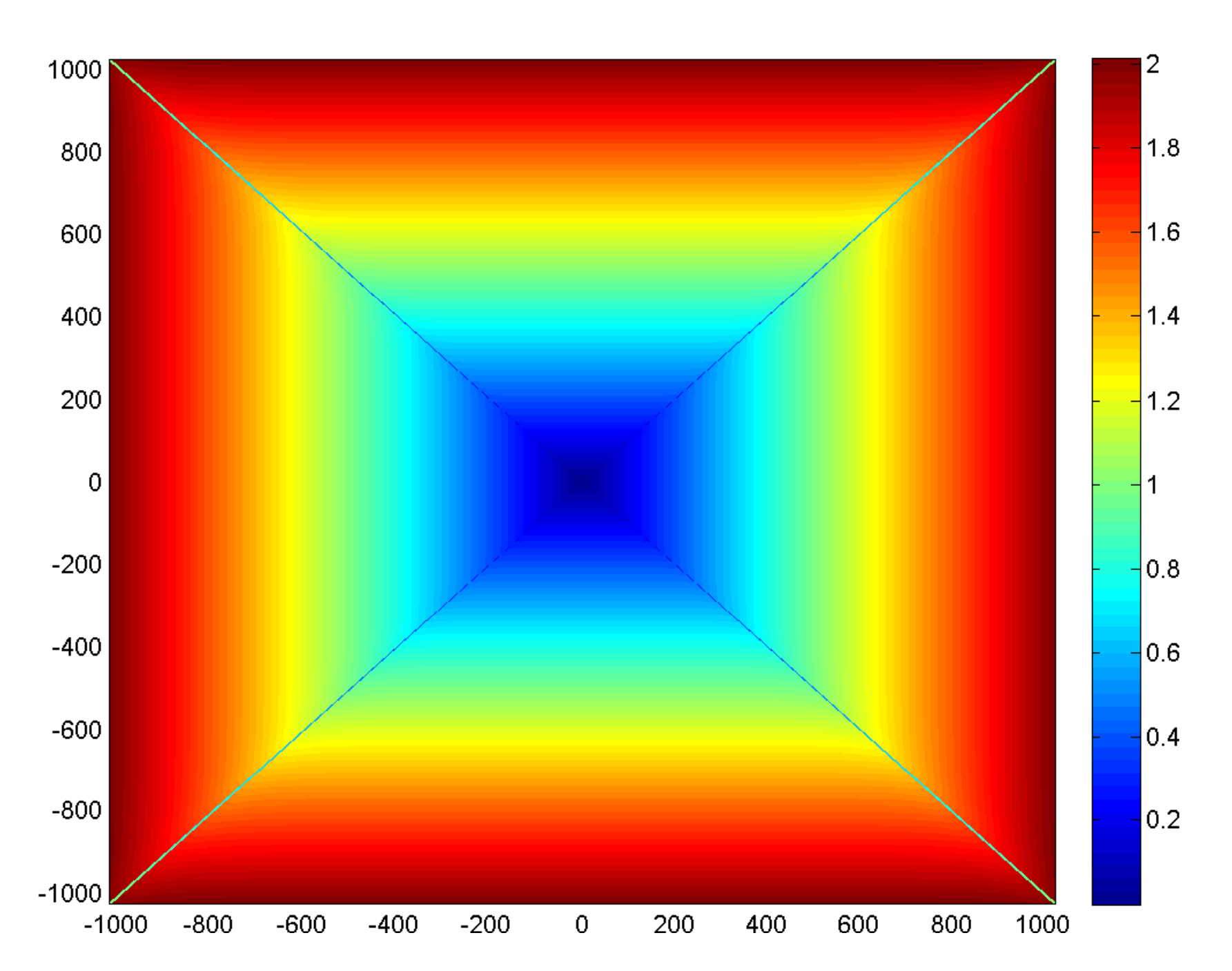}
\put(-70,-17){Choice 2}
\end{center}
\caption{Weight functions on the pseudo-polar grid for $N=256$ and  $R=8$.}
\label{fig:001}
\end{figure}

\subsubsection{Computation of the Weighting}
\label{subsubsec:weighting}

For the FDST --  as also in the implementation in \url{ShearLab} -- the coefficients in the expansion
\eqref{eq:compute_w} will be computed off-line, and then hardwired in the code. This enables the weighting of a function
on the pseudo-polar grid to simply be a point-wise multiplication in each sampling point. That is, letting $J:=\hat{I}:
\Omega_R\rightarrow\bC$ be the pseudo-polar Fourier transform of an $N \times N$ image $I$ and $w:\Omega_R\rightarrow \bR^+$ be
any suitable weight function on $\Omega_R$, the values
\[
J_w(\ox,\oy) = J(\ox,\oy)\cdot \sqrt{w(\ox,\oy)} \quad \mbox{for all } (\ox,\oy) \in \Omega_R
\]
need to be computed.

Let us comment on why the square root of the weight is utilized. If the weights $w$ satisfy the
condition in Theorem \ref{thm:weight}, we obtain $P^* w P = \Id$, which can be written in a symmetric form as
follows: $(\sqrt{w} P)^* \sqrt{w} P = \Id$. This form shows that the operator $\sqrt{w} P$ can be inverted by taking
the adjoint $(\sqrt{w} P)^*$. In other words, each image can be reconstructed from its weighted pseudo-polar Fourier
transform by applying the adjoint of the weighted pseudo-polar Fourier transform. This issue will be discussed in
further detail in Subsection \ref{subsubsec:adjoint}.


\subsection{Digital Shearlets on Pseudo-Polar Grid}
\label{subsec:DSHOnPPGrid}

We next aim at deriving a faithful digitization of the shearlet transform associated with a band-limited cone-adapted
discrete shearlet system to the pseudo-polar grid. This would settle Step (S3).

\subsubsection{Preparation for Faithful Digitization}

For this, let us recall the definition of the discrete shearlet transform associated with \eqref{def:csht}; taking
the particular form \eqref{eq:psidef} of the shearlet $\psi \in L^2(\R^2)$ into account. Restricting our attention
to the cone $\cC_{21}$, we obtain
\begin{eqnarray*}
f & \mapsto & \Big\langle \hat{f},2^{-j\tfrac32}\hat\psi(S_k^TA_{4^{-j}}\cdot)\chi_{\cC_{21}} e^{2\pi i \ip{A_{4^{-j}}S_km}{\cdot}} \Big\rangle\\
& & = \Big\langle \hat{f},2^{-j\tfrac32}\hat{\psi}_1(4^{-j}\xi_1) \hat{\psi}_2(k + 2^{j}\tfrac{\xi_2}{\xi_1}) \chi_{\cC_{21}}e^{2\pi i \ip{A_{4^{-j}}S_km}{\cdot}}
\Big\rangle,
\end{eqnarray*}
for scale $j$, orientation $k$, position $m$, and cone $\iota$.

To approach a faithful digitization, we first
have to partition $\Omega_R$ according to the partitioning of the plane into $\cC_{11}$, $\cC_{12}$, $\cC_{21}$, and $\cC_{22}$,
as well as a centered rectangle $\cR$. The center  $\cC$ as defined in  \eqref{eq:00} will play the role of $\cR$.
Thus it remains to partition the
set $\Omega_R$ beyond the already defined partitioning into $\Omega_R^1$ and $\Omega_R^2$ (cf. \eqref{eq:OmegaR1} and
\eqref{eq:OmegaR2}) by setting
\[
\Omega_R^1 = \Omega_R^{11} \cup \cC \cup \Omega_R^{12}
\qquad \mbox{and} \qquad
\Omega_R^2 = \Omega_R^{21} \cup \cC \cup \Omega_R^{22},
\]
where
\begin{eqnarray*}
\Omega_R^{11} & = &  \{(-\tfrac{2n}{R}\cdot\tfrac{2\ell}{N},\tfrac{2n}{R}) : -\tfrac{N}{2} \le \ell \le \tfrac{N}{2}, \, 1 \le n \le \tfrac{RN}{2}\},\\
\Omega_R^{12} & = &  \{(-\tfrac{2n}{R}\cdot\tfrac{2\ell}{N},\tfrac{2n}{R}) : -\tfrac{N}{2} \le \ell \le \tfrac{N}{2}, \, -\tfrac{RN}{2} \le n \le -1\},\\
\Omega_R^{21} & = &  \{(\tfrac{2n}{R},-\tfrac{2n}{R}\cdot\tfrac{2\ell}{N}) : -\tfrac{N}{2} \le \ell \le \tfrac{N}{2}, \, 1 \le n \le \tfrac{RN}{2}\}\\
\Omega_R^{22} & = &  \{(\tfrac{2n}{R},-\tfrac{2n}{R}\cdot\tfrac{2\ell}{N}) : -\tfrac{N}{2} \le \ell \le \tfrac{N}{2}, \, -\tfrac{RN}{2} \le n \le -1\}.
\end{eqnarray*}

When restricting to the cone $\Omega_R^{21}$, say, the exact digitization of the coefficients of the discrete shearlet system is
\begin{eqnarray}\nonumber
\lefteqn{\sum_{\omega:=(\ox,\oy) \in \Omega_R^{21}} J(\ox,\oy) 2^{-j\frac{3}{2}} \overline{\hat{\psi}(S_k^T A_{4^{-j}} \omega)} e^{-2\pi i\ip{A_{4^{-j}} S_k m}{\omega}}}\\
\nonumber
& = & \sum_{(\ox,\oy) \in \Omega_R^{21}} J(\ox,\oy)
2^{-j\frac{3}{2}} \overline{W(4^{-j}\omega_x)V(k+2^j\tfrac{\oy}{\ox})} e^{-2\pi i\ip{A_{4^{-j}} S_k m}{\omega}}\\ \label{eq:idealshearlet}
& = & \sum_{n=1}^{\frac{RN}{2}} \sum_{\ell = -\frac{N}{2}}^{\frac{N}{2}} J(\ox,\oy)
2^{-j\frac{3}{2}} \overline{W(4^{-j}\tfrac{2n}{R})}  \overline{V(k-2^{j+1}\tfrac{\ell}{N})} e^{-2\pi i \ip{m}{S_k^T A_{4^{-j}} \omega}},
\end{eqnarray}
where $V$ and $W$ as well as the ranges of $j$, $k$, and $m$ are to be carefully chosen.

Our main objective will be to achieve a digital shearlet transform, which is an isometry. This --  as in the continuum domain
situation -- is equivalent to requiring the associated shearlet system to form a tight frame for functions $J : \Omega_R \to \CC$.
For the convenience of the reader let us recall the notion of a Parseval frame in this particular situation. A sequence
$(\varphi_\lambda)_{\lambda \in \Lambda}$ -- $\Lambda$ being some indexing set -- is a {\em tight frame} for all functions
$J : \Omega_R \to \CC$, if
\[
\sum_{\lambda \in \Lambda} \Big| \sum_{(\ox, \oy) \in \Omega_R} J(\ox, \oy) \overline{\varphi_\lambda(\ox, \oy)} \, \Big|^2
= \sum_{(\ox, \oy) \in \Omega_R} |J(\ox,\oy)|^2.
\]

In the sequel we will define digital shearlets on $\Omega_R^{21}$ and extend the definition to the other cones by symmetry.

\subsubsection{Subband Windows on the Pseudo-Polar Grid}
\label{subsec:subbandwindows}

We start by defining the scaling function, which will depend on two functions $V_0$ and $W_0$, and the generating digital shearlet,
which will depend on again two functions $V$ and $W$. $W_0$ and $W$ will be chosen to be Fourier transforms of wavelets, and
$V_0$ and $V$ will be chosen to be `bump' functions, paralleling the construction of classical shearlets.

First, let $W_0$ be the Fourier transform of the Meyer scaling function such that
\beq \label{eq:defW0}
\Sp W_0 \subseteq [-1,1]\quad\mbox{and}\quad W_0(\pm 1)=0,
\eeq
and let $V_0$ be a `bump' function satisfying
\[
\Sp V_0 \subseteq [-3/2,3/2] \qquad \mbox{with} \qquad V_0(\xi)\equiv 1\mbox{ for } |\xi|\le 1, \xi\in\bR.
\]
Then we define the {\em scaling function} $\phi$ for the digital shearlet system to be
\[
\hat\phi(\xi_1,\xi_2) = W_0(4^{-j_L}\xi_1)V_0(4^{-j_L}\xi_2), \quad (\xi_1,\xi_2)\in\bR^2.
\]
For now, we define it in continuum domain, and will later restrict this function to the pseudo-polar grid.

Let next $W$ be the Fourier transform of the Meyer wavelet function satisfying the support constraints
\beq \label{eq:supportW}
\Sp W \subseteq [-4,1/4] \cup [1/4,4] \quad \mbox{and} \quad W(\pm 1/4)=W(\pm 4)=0,
\eeq
as well as, choosing the lowest scale $j_L$ to be $j_L:=-\lceil\log_4(R/2)\rceil$,
\beq \label{eq:summabilityW}
|W_0(4^{-j_L}\xi)|^2 + \sum_{j =j_L}^{\lceil\log_4 N\rceil}  |W(4^{-j} \xi)|^2 = 1 \qquad \mbox{for all } |\xi| \le N,\; \xi \in \bR.
\eeq
We further choose $V$ to be a `bump' function  satisfying
\beq \label{eq:supportV}
\Sp V \subseteq [-1,1] \quad \mbox{and} \quad V(\pm 1)=0,
\eeq
as well as
\beq \label{eq:summabilityV1}
|V(\xi-1)|^2 + |V(\xi)|^2 + |V(\xi+1)|^2 = 1 \qquad \mbox{for all } |\xi| \le 1, \; \xi \in \bR.
\eeq
Then the {\em generating shearlet} $\psi$ for the digital shearlet system on $\Omega_R^2$ is defined as
\beq \label{eq:digitalshearlet}
\hat\psi(\xi_1,\xi_2) = W(\xi_1)V(\tfrac{\xi_2}{\xi_1}), \quad (\xi_1,\xi_2)\in\bR^2.
\eeq
Notice that \eqref{eq:summabilityV1} implies
\beq \label{eq:summabilityV2}
\sum_{s=-2^j}^{2^j} |V(2^j\xi-s)|^2 = 1 \qquad \mbox{for all } |\xi| \le 1, \; \xi \in \bR; j\ge0,
\eeq
which will become important for the analysis of frame properties. For the particular choice of $V_0$, $W_0$, $V$, and $W$
in \url{ShearLab}, we refer to Subsection~\ref{subsubsec:windowing}.

\subsubsection{Range of Parameters}

We from now on assume that $R$ and $N$ are both positive even integers and that $N=2^{n_0}$ for some integer $n_0\in\bN$.
This poses no restrictions, since both parameters can be enlarged to satisfy this condition.

We start by analyzing the range of $j$. Recalling the definition of the shearlet $\psi$ in \eqref{eq:digitalshearlet}
and the support properties of $W$ and $V$ in \eqref{eq:supportW} and \eqref{eq:supportV}, respectively, we observe that
the digitized shearlet
\beq \label{eq:prototype}
2^{-j\frac{3}{2}} W(4^{-j}\tfrac{2n}{R}) V(k-2^{j+1}\tfrac{\ell}{N}) e^{2\pi i \ip{m}{S_k^T A_{4^{-j}} \omega}}
\eeq
from \eqref{eq:idealshearlet} has radial support
\beq \label{eq:radial_k}
n = 4^{j-1} \tfrac{R}{2} + t_1, \quad t_1 = 0 ,\ldots, 4^{j-1}\cdot \tfrac{15 R}{2}
\eeq
on the cone $\Omega_R^{21}$. To determine the appropriate range of $j$, we will analyze the precise support in radial direction.
If $j<-\lceil\log (R/2)\rceil$, then $n<1$, which corresponds to only one point -- the origin -- and is dealt with by the scaling
function. If $j > \lceil\log_4 N\rceil$, we have $n \ge \frac{RN}{2}$. Hence the value $W(1/4) = 0$ (cf. \eqref{eq:supportW}) is
placed on the boundary, and these scales can be omitted. Therefore, the range of the scaling parameter will be chosen to be
\[
j \in  \{j_L, \ldots, j_H\}, \quad\mbox{where } j_L :=-\lceil\log(R/2)\rceil \mbox{ and } j_H := \lceil\log_4 N\rceil.
\]

Next, we determine the appropriate range of $k$. Again recalling the definition of the shearlet $\psi$ in \eqref{eq:digitalshearlet},
the digitized shearlet \eqref{eq:prototype} has angular support
\beq \label{eq:angular_ell}
\ell = 2^{-j-1} N(k-1) + t_2, \quad t_2 = 0, \ldots, 2^{-j} N
\eeq
on the cone $\Omega_R^{21}$. To compute the range of $k$, we start by examining the case $j\ge 0$.  If $k > 2^j$, we have $\ell \ge N/2$.
Hence the value $V(-1)=0$ (cf. \eqref{eq:supportV}) is placed on the seam line, and these parameters can be omitted. By symmetry,
we also obtain $k \ge -2^j$. Thus the shearing parameter will be chosen to be
\[
k \in \{-2^j,\ldots,2^j\}.
\]

\subsubsection{Support Size of Shearlets}

We next compute the support as well as the support size of scaled and sheared version of digital shearlets. This will be used
for the normalization of digital shearlets.

As before, we first analyze the radial support. By \eqref{eq:radial_k}, the radial supports of the windows associated with
scales $j_L < j < j_H$ is
\beq \label{eq:defi_k}
n = 4^{j-1} \tfrac{R}{2} + t_1, \quad t_1 = 0 ,\ldots, 4^{j-1} \cdot \tfrac{15 R}{2},
\eeq
and the radial support of the windows associated with the scale $j_L=-\lceil \log_4(R/2)\rceil$ and $j_H=\lceil\log_4 N\rceil$ are
\beq \label{eq:defi_k_H_L}
\begin{aligned}
n &= t_1,  &t_1&=1,\ldots,4^{j_L+1}\tfrac{R}{2}, &\mbox{ for } &j=j_L,\\
n &= 4^{j_H-1} \tfrac{R}{2}+t_1, &t_1&=0, \ldots, \tfrac{RN}{2}-4^{j_H-1}\tfrac{R}{2},& \mbox{ for }
&j = j_H.
\end{aligned}
\eeq

Turning to the angular direction, by \eqref{eq:angular_ell}, the angular support of the windows at scale $j$ associated with
shears $-2^j < k < 2^j$ is
\beq \label{eq:defi_ell}
\ell = 2^{-j-1} N(k-1) + t_2, \quad t_2 = 0, \ldots, 2^{-j} N,
\eeq
the angular support at scale $j$ associated with the shear parameter $k =-2^j$ is
\[
\ell = 2^{-j-1} N(-2^j-1) + t_2, \quad t_2 = 2^{-j} \tfrac{N}{2}, \ldots, 2^{-j} N,
\]
and for $k =2^j$ it is
\beq \label{eq:defi_ell_U}
\ell = 2^{-j-1} N(2^j-1) + t_2, \quad t_2 = 0, \ldots, 2^{-j} \tfrac{N}{2}.
\eeq
For the case  $j<0$, we simply let $k=0$ and $\ell=-N/2+t_2$ with $t_2=0,\ldots,N$. Also, for this lower frequency
case, the window function $W(4^{-j}\ox)V(k+2^j\tfrac{\oy}{\ox})$ is slightly modified to be $W(4^{-j}\ox)V_0(k+2^j\tfrac{\oy}{\ox})$.

These computations now allow us to determine the support size of the function $W(4^{-j}\ox) V(k+2^j\frac{\oy}{\ox})$ in terms
of pairs $(n,\ell)$, which for scale $j$ and shear $k$, is
\beq \label{eq:L1}
\cL^1_{j} =\left\{ \begin{array}{cll}
4^{j+1}\tfrac{R}{2} &:& j= j_L,\\
4^{j-1} \cdot \frac{15 R}{2} + 1 & : & j_L<j<j_H,\\
\tfrac{RN}{2}-4^{j-1}\tfrac{R}{2}+1 & : & j=j_H,
 \end{array} \right.
\eeq
and
\beq \label{eq:L2}
\cL^2_{j,k} =\left\{ \begin{array}{cll}
2^{-j} N + 1 & : & -2^j < k < 2^j\;\mbox{ with }\;j\ge0,\\
2^{-j} \frac{N}{2}+1 & : & k \in \{-2^j,2^j\}\;\mbox{ with }\; j\ge0,\\
N+1 & : & j<0.
 \end{array} \right.
\eeq

\subsubsection{Digitization of the Exponential Term}
\label{subsubsec:exponential}

We next digitize the exponential term in \eqref{eq:prototype}, which can be rewritten as
\[
e^{-2\pi i \ip{m}{S_k^T A_{4^{-j}} \omega}}
= e^{-2\pi i \ip{m}{(4^{-j}\ox,4^{-j}k\ox + 2^{-j} \oy)}}
= e^{-2\pi i \ip{m}{(4^{-j}\frac{2n}{R},4^{-j}k\frac{2n}{R} - 2^{-j} \frac{4\ell n}{RN})}}.
\]
We observe two obstacles:
\bitem
\item The change of variables $\tau := S_k^T A_{4^{-j}} \omega$ possible in \eqref{eq:idealshearlet} can not be
performed similarly in this situation due to the fact that the pseudo-polar grid is {\em not} invariant under the action of
$S_k^T A_{4^{-j}}$. This is however the first step in the continuum domain reasoning for tightness; see Chapter
\cite{Introduction}.
\item The Fourier transform of a function defined on the pseudo-polar grid does {\em not} satisfy any
Plancherel theorem.
\eitem
These problems require a slight adjustment of the exponential term, which will be the only adaption we
allow us to make when digitizing. This will circumvent the two obstacles and enable us to construct
a Parseval frame as well as derive a direct application of the inverse Fast Fourier transform in FDST.

The adjustment will be made by using the mapping $\theta : \bR\setminus\{0\} \to \bR$ defined by $\theta(x,y) = (x,\tfrac{y}{x})$.
This yields the modified exponential term
\begin{equation}\label{eq:expTerm}
e^{-2\pi i \ip{m}{(\theta \circ (S_k^T)^{-1})(4^{-j}\frac{2n}{R},4^{-j}k\frac{2n}{R} - 2^{-j} \frac{4\ell n}{RN})}}
= e^{-2\pi i \ip{m}{(4^{-j}\frac{2n}{R},-2^{j+1}\frac{\ell}{N})}},
\end{equation}
which can be rewritten as
\[
e^{-2\pi i \ip{m}{(4^{-j}\frac{2n}{R},-2^{j+1}\frac{\ell}{N})}}
= e^{-2\pi i (\frac{m_1}{4}+(1-k)m_2)} e^{-2\pi  i \ip{m}{(4^{-j}\frac{2t_1}{R},-2^{j+1}\frac{t_2}{N})}},
\]
with $t_1$ and $t_2$ ranging over an appropriate set defined by \eqref{eq:defi_k}, \eqref{eq:defi_k_H_L},
and \eqref{eq:defi_ell}--\eqref{eq:defi_ell_U}. Fig.~\ref{fig:adjustmentexp} illustrates this adjustment.
\begin{figure}[ht]
\begin{center}
\includegraphics[height=1.0in]{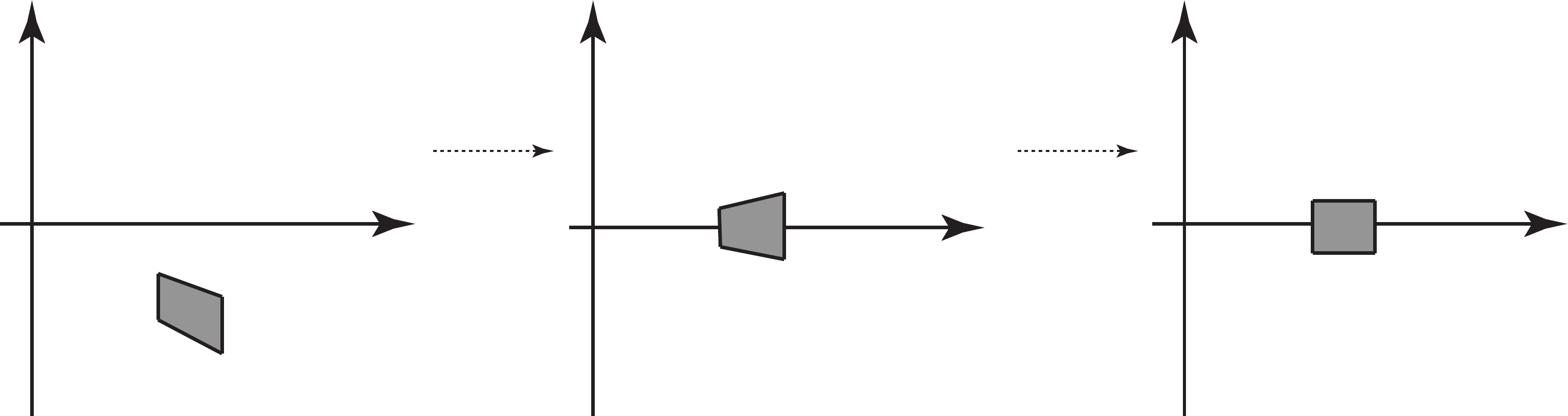}
\put(-238,60){$(S_k^T)^{-1}$}
\put(-105,60){$\theta$}
\end{center}
\caption{Adjustment of the exponential term through the map $\theta \circ (S_k^T)^{-1}$. }
\label{fig:adjustmentexp}
\end{figure}

Now, taking into account of the support size of each $W(4^{-j}\ox)V(k+2^j\frac{\oy}{\ox})$ as given in \eqref{eq:L1} and \eqref{eq:L2}, we
obtain the following reformulation of \eqref{eq:expTerm}:
\begin{equation}\label{eq:exp}
\exp\left\{-2\pi i \ip{m}{\left(\tfrac{\cL^1_j 4^{-j}(2/R)}{\cL^1_j}t_1,\tfrac{-\cL^2_{j,k}2^{j+1}(1/N)}{\cL^2_{j,k}}t_2\right)}\right\}, \quad t_1, \; t_2.
\end{equation}
This version shows that we might regard the exponential terms as characters of a suitable locally compact abelian group (see \cite{HR63})
with associated annihilator identified with the rectangle
\[
\cR_{j,k} = \left\{\left(\frac{4^{j}\tfrac{R}{2} \cdot r_1}{\cL^1_j}, -\frac{\tfrac{N}{2^{j+1}} \cdot r_2}{\cL^2_{j,k}}\right) : r_1 = 0 ,\ldots, \cL^1_j-1,\:
r_2 = 0, \ldots, \cL^2_{j,k} -1\right\},
\]
where $\cL^1_j$ and $\cL^2_{j,k}$ were defined in \eqref{eq:L1} and \eqref{eq:L2}, respectively. This viewpoint will be crucial to
guarantee that the digital shearlet system defined in Subsection \ref{subsubsec:digital} provides a Parseval frame on the pseudo-polar
grid $\Omega_R$. In practice, \eqref{eq:exp} also ensures that in Step (S3) on each windowed image on the pseudo-polar grid only a 2D-iFFT
-- in contrast to a fractional Fourier transform -- needs to be performed, thereby reducing the computational complexity.

For the low frequency square, we further require the set
\[
\cR = \{(r_1, r_2) : r_1 = -1 ,\ldots, 1,\: r_2 = -\tfrac{N}{2}, \ldots, \tfrac{N}{2}\},
\]
which will be shown to be sufficient for guaranteeing that digital shearlet system forms a Parseval frame.

\subsubsection{Digital Shearlets}
\label{subsubsec:digital}

We are now ready to define digital shearlets, which we define as functions on the pseudo-polar grid $\Omega_R$. The
spatial domain picture can thus be derived by the inverse pseudo-polar Fourier transform.

\begin{definition} \label{defi:digitalshearlets}
Retaining the definitions and notations from Subsection \ref{subsec:DSHOnPPGrid}, for all $(\ox, \oy) \in \Omega_R^{21}$,
we define {\em digital shearlets} at scale $j \in  \{j_L, \ldots, j_H\}$, shear $k=[-2^j, 2^j]\cap\bZ$, and
 spatial position $m \in \cR_{j,k}$ by
\begin{eqnarray*}
\sigma_{j,k,m}^{21}(\ox, \oy)
 =   \tfrac{C(\ox, \oy) }{\sqrt{|\cR_{j,k}|}} \, W(4^{-j} \ox) \, V^j(k+2^{j}\tfrac{\oy}{\ox})\chi_{\Omega_R^{21}}(\ox, \oy)\,
e^{2\pi i \ip{m}{(4^{-j}\ox,2^{j}\tfrac{\oy}{\ox})}},
\end{eqnarray*}
where $V^j = V$ for $j\ge0$ and $V^j=V_0$ for $j<0$, and
\[
C(\ox, \oy) = \left\{\begin{array}{cll}
1 & : &  (\ox, \oy) \not\in \cS_R^1 \cup \cS_R^2,\\
\frac{1}{\sqrt{2}} & : & (\ox, \oy) \in (\cS_R^1 \cup \cS_R^2)\setminus \cC,\\
\frac{1}{\sqrt{2(N+1)}} & : & (\ox, \oy) \in \cC.
\end{array}\right.
\]
The shearlets $\sigma_{j,k,m}^{11}, \sigma_{j,k,m}^{12}, \sigma_{j,k,m}^{22}$ on the
remaining cones are defined accordingly by symmetry with equal indexing sets for scale $j$, shear $k$, and spatial location $m$.
For $\iota_0=1,2$, $(\ox, \oy) \in \Omega_R^{\iota_0}$, and $n_0 \in \cR$, we define the {\em scaling function}
\[
\varphi_{n_0}^{\iota_0}(\ox, \oy)  =
\tfrac{C(\ox, \oy)}{\sqrt{|\cR|}} \hat{\phi}(\ox, \oy) \chi_{\Omega_R^{\iota_0}}(\ox, \oy)\, e^{2\pi i\ip{n_0}{(\frac{n}{3},\frac{\ell}{N+1})}}.
\]
Then the {\em digital shearlet system} $DSH$ is defined by
\begin{eqnarray*}
DSH&  = & \{\varphi_{n_0}^{\iota_0} : \iota_0=1,2, n_0 \in \cR\} \cup\{\sigma_{j,k,m}^{\iota}
: j \in \{j_L, \ldots, j_H\}, k\in \{-2^j, 2^j\}, \\& &\hspace*{5.5cm}  m \in \cR_{j,s},\iota=11,12,21,22\}.
\end{eqnarray*}
\end{definition}

As desired, the digital shearlet system $DSH$, which we derived as a faithful digitization of the continuum domain
band-limited cone-adapted discrete shearlet system, forms a Parseval frame for $J : \Omega_R \to \CC$.

\begin{theorem}[\cite{KSZ11}]\label{thm:DSHtight}
The digital shearlet system $DSH$ defined in Definition \ref{defi:digitalshearlets} forms a Parseval frame for functions $J : \Omega_R \to \CC$.
\end{theorem}

\begin{proof}
Letting $J : \Omega_R \to \CC$, we claim that
\begin{equation}\label{eq:RHSLHS}
\langle J,J\rangle_{\Omega_R}=\sum_{\iota_0,n_0}|\langle J,\varphi_n^{\iota_0} \rangle_{\Omega_R}|^2
+\sum_{\iota,j,k,m}|\langle J,\sigma_{j,k,m}^{\iota} \rangle_{\Omega_R}|^2
\end{equation}
which proves the result. Here $\langle J_1, J_2\rangle_{\Omega_R}:=\sum_{(\ox,\oy)\in\Omega_R}J_1(\ox,\oy)\overline{J_2(\ox,\oy)}$ for $J_1, J_2:\Omega_R\rightarrow \bC$.

We start by analyzing the first term on the RHS of \eqref{eq:RHSLHS}. Let $\iota_0 \in \{1, 2\}$ and
$J_C:\Omega_R\rightarrow\bC$ be defined by $J_C(\ox,\oy):=C(\ox,\oy)\cdot J(\ox,\oy)$ for $(\ox,\oy)\in\Omega_R$.
Using the support conditions of $\hat\phi$,
\begin{eqnarray*}
\hspace*{-0.5cm} &&\sum_n|\langle J,\varphi_{n_0}^{\iota_0}\rangle_{\Omega_R}|^2=\sum_{n_0}\Big|\sum_{(\ox,\oy) \in \Omega_R^{\iota_0}}
J(\ox,\oy) \overline{\varphi_{n_0}^{\iota_0}(\ox,\oy)}\Big|^2\\
& = & \frac{1}{|\cR|} \sum_{n_0} \Big|\sum_{(\ox,\oy) \in \Omega_R^{\iota_0}}   J_C(\ox,\oy)\cdot
\hat{\phi}(\ox, \oy) \cdot e^{-2\pi i\ip{n_0}{(\frac{n}{3},\frac{\ell}{N+1})}}\Big|^2\\
& = & \frac{1}{|\cR|} \sum_{n_0} \Big|\sum_{n=-1}^{1} \sum_{\ell=-N/2}^{N/2}
 J_C(\ox,\oy)\cdot \hat{\phi}(\ox, \oy) \cdot e^{-2\pi i\ip{n_0}{(\frac{n}{3},\frac{\ell}{N+1})}}\Big|^2.
\end{eqnarray*}
The choice of $\cR$ now allows us to use the Plancherel formula, see Subsection \ref{subsubsec:exponential}. Exploiting again
support properties (see Subsection~\ref{subsubsec:exponential}), we conclude that
\[
\sum_n|\langle J,\varphi_n^{\iota_0}\rangle_{\Omega_R}|^2=\sum_{(\ox,\oy) \in \Omega_R^{\iota_0}} |C(\ox,\oy) \cdot J(\ox,\oy)|^2 \cdot |\hat{\phi}(\ox, \oy)|^2.
\]
Combining $\iota_0=1,2$ and using \eqref{eq:defW0}, we proved
\beq \label{eq:firstterm}
\sum_{\iota_0}\sum_{n_0}|\langle J,\varphi_{n_0}^{\iota_0}\rangle_{\Omega_R}|^2=
\sum_{(\ox,\oy) \in \Omega_R} |J(\ox,\oy)|^2 \cdot |W_0(\ox)|^2.
\eeq

Next we study the second term on the RHS in \eqref{eq:RHSLHS}. By symmetry, it suffices to consider the case $\iota=21$.
By the support conditions on $W$ and $V$ (see \eqref{eq:supportW} and \eqref{eq:supportV}),
{\allowdisplaybreaks
\begin{eqnarray*}
&&\sum_{j,k,m}|\langle J,\sigma_{j,k,m}^{21}\rangle_{\Omega_R}|^2=
\sum_{j,k} \sum_{m \in \cR_{j,k}} \Big|\sum_{(\ox,\oy) \in \Omega_R^{21}} J(\ox,\oy)
\overline{\sigma_{j,k,m}^{21}(\ox,\oy)}\Big|^2\\
& = & \sum_{j,k} \frac{1}{|\cR_{j,k}|} \sum_{m \in \cR_{j,k}} \Big|\sum_{(\ox,\oy) \in \Omega_R^{21}} J_C(\ox,\oy) \cdot \overline{W(4^{-j} \ox)}
\\ & & \hspace*{3cm} \cdot
 \overline{V^j(k+2^{j}\tfrac{\oy}{\ox})} \cdot e^{-2\pi i\ip{m}{(4^{-j}\ox,2^{j}\frac{\oy}{\ox})}}\Big|^2\\  \nonumber
& = & \sum_{j,k} \frac{1}{|\cR_{j,k}|} \sum_{m \in \cR_{j,k}} \Big|\sum_{n=4^{j-1} (R/2)}^{4^{j+1} (R/2)} \sum_{\ell = 2^{-j-1}N(k-1)}^{2^{-j-1}N(k+1)}
 J_C(\ox,\oy)  \\
& &  \hspace*{3cm} \cdot  \overline{W(4^{-j} \ox)}\cdot \overline{V^j(k+2^{j}\tfrac{\oy}{\ox})} \cdot e^{-2\pi i\ip{m}{(4^{-j}\frac{2n}{R},-2^{j+1}\frac{\ell}{N})}}\Big|^2.
\end{eqnarray*}
}
Similarly as before, the choice of $\cR_{j,k}$ does allow us to use the Plancherel formula, see Subsection \ref{subsubsec:exponential}.
Hence,
\[
\sum_{j,k,m}|\langle J,\sigma_{j,k,m}^{21}\rangle_{\Omega_R}|^2=
\sum_{j,k} \sum_{(\ox,\oy) \in \Omega_R^{21}} \Big| J_C(\ox,\oy)\cdot\overline{W(4^{-j} \ox)V^j(k+2^{j}\tfrac{\oy}{\ox})}\Big|^2.
\]
Next we use \eqref{eq:summabilityV2} to obtain
\begin{eqnarray} \nonumber
\lefteqn{\sum_{j,k} \sum_{(\ox,\oy) \in \Omega_R^{21}} \Big| J_C(\ox,\oy) \cdot\overline{W(4^{-j} \ox)} \cdot\overline{V^j(k+2^{j}\tfrac{\oy}{\ox})}\Big|^2}\\ \nonumber
& = & \hspace*{-0.25cm} \sum_{(\ox,\oy) \in \Omega_R^{21}} | J_C(\ox,\oy)|^2 \sum_{j =j_L}^{j_H}
 |W(4^{-j} \ox)|^2 \cdot  \sum_{k=-2^j}^{2^j} |V^j(k+2^{j}\tfrac{\oy}{\ox})|^2 \\ \nonumber
& = & \hspace*{-0.25cm} \sum_{(\ox,\oy) \in \Omega_R^{21}} | J_C(\ox,\oy)|^2 \sum_{j = j_L}^{j_H}
 |W(4^{-j} \ox)|^2.
\end{eqnarray}
Thus the second term on the RHS in \eqref{eq:RHSLHS} equals
\beq \label{eq:secondterm}
\sum_{\iota}\sum_{j,k,m}|\langle J,\sigma_{j,k,m}^{\iota}\rangle_{\Omega_R}|^2=\sum_{(\ox,\oy) \in \Omega_R} |J(\ox,\oy)|^2\cdot \sum_{j =j_L}^{j_H}.
 |W(4^{-j} \ox)|^2.
\eeq
Finally, our claim \eqref{eq:RHSLHS} follows from combining \eqref{eq:firstterm}, \eqref{eq:secondterm}, and \eqref{eq:summabilityW}.
\qed
\end{proof}

\subsubsection{Digital Shearlet Windowing}
\label{subsubsec:windowing}

The final Step (S3) of the FDST then consists in decomposing the data on the points of the pseudo-polar grid given by the
previously -- in Steps (S1) and (S2) -- computed weighted pseudo-polar image $J_w : \Omega_R \to \CC$ into rectangular
subband windows according to the digital shearlet system DSH defined in Definition \ref{defi:digitalshearlets}, followed
by a 2D-iFFT. More precisely, given $J_w$, the set of digital shearlet coefficients
\[
c^{\iota_0}_{n_0} := \ip{J_w}{\varphi_{n_0}^{\iota_0}}_{\Omega_R}\quad \mbox{for all }  \iota_0, n_0
\]
and
\[
c_{j,k,m}^\iota := \ip{J_w}{\sigma_{j,k,m}^{\iota}}_{\Omega_R}\quad \mbox{for all } j,k,m,\iota
\]
is computed followed by application of the 2D-iFFT to each windowed image $J_w\varphi_{0}^{\iota_0}$ and $J_w\sigma_{j,k,0}^\iota$
restricted on the support of $\varphi_0^{\iota_0}$ and $\sigma_{j,k,0}^\iota$, respectively.

The definition of the digital shearlet system DSH in Definition \ref{defi:digitalshearlets} requires appropriate choices
of the functions $\phi$, $V_0$, $V$, $W_0$, and $W$, and the required conditions are stated throughout Subsection
\ref{subsec:subbandwindows}. We now discuss one particular choice, which is chosen in \url{ShearLab}. We start selecting the
 `wavelets' $W_0$ and $W$.  In Subsection \ref{subsec:subbandwindows}, these functions were defined to be
Fourier transforms of the Meyer scaling function and wavelet function, respectively, i.e.,
\[
W_0(\xi)=\left\{
\begin{array}{lcl}
1& : & |\xi|\leq\frac{1}{4},\\
\cos\left[\frac{\pi}{2}\nu(\frac{4}{3}|\xi|-\frac13)\right]&:& \frac{1}{4}\le|\xi|\le 1,\\
0 &:& \mbox{otherwise},
\end{array}\right.
\]
and
\[
W(\xi)=
\left\{\begin{array}{lcl}
\sin\left[\frac{\pi}{2}\nu(\frac{4}{3}|\xi|-\frac13)\right]&:& \frac{1}{4}\le|\xi|\leq 1,\\
\cos\left[\frac{\pi}{2}\nu(\frac{1}{3}|\xi|-\frac13)\right]&:& 1\leq|\xi|\le 4,\\
0 &:& \mbox{otherwise},
\end{array}\right.
\]
where $\nu\ge0$ is a $C^k$ function or $C^\infty$ function such that $\nu(x)+\nu(1-x)=1$ for $0\le x\le1$.
One possible choice for $\nu$ is the function $\nu(x) = x^4(35-84x+70x^2-20x^3)$, $0\le x\le 1$, which
then automatically fixes $W_0$ and $W$. Since $|W_0(\xi)|^2+|W(\xi)|^2=1$ for $|\xi|\le 1$, the required
condition \eqref{eq:summabilityW} is satisfied.
The function $\nu$  can be also used to design the `bump' function $V$ as well, which needs to satisfy
\eqref{eq:summabilityV1}. One possible choice for $V$ is to define it by
$V(\xi)=\sqrt{\nu(1+\xi)+\nu(1-\xi)}$, $-1\le \xi\le 1$. $V_0$ can then simply be chosen as $V_0\equiv 1$.
Let us finally mention that $\phi$ is defined depending on $V_0$ and $W_0$, wherefore fixing these two
functions determines $\phi$ uniquely.



\subsection{Algorithmic Realization of the FDST}
\label{subsec:FDST}

We have previously discussed all main ingredients of the fast digital shearlet transform (FDST) -- Fast PPFT, Weighting, and Digital Shearlet
Windowing --, and will now summarize those findings. Depending on the application at hand, a fast inverse transform is required, which
we will also detail in the sequel. In fact, we will present two possibilities: the Adjoint FDST and the Inverse FDST depending on whether
the weighting allows to use the adjoint for reconstruction or whether an iterative procedure is required for higher accuracy. Fig.~\ref{fig:flowcharts}
provides an overview of the main steps of of the FDST and its inverse. For
a more detailed description of FDST, Adjoint FDST, and Inverse FDST in form of pseudo-code, we refer to \cite{KSZ11}.
\begin{figure}[ht]
\begin{center}
\includegraphics[height=2.1in]{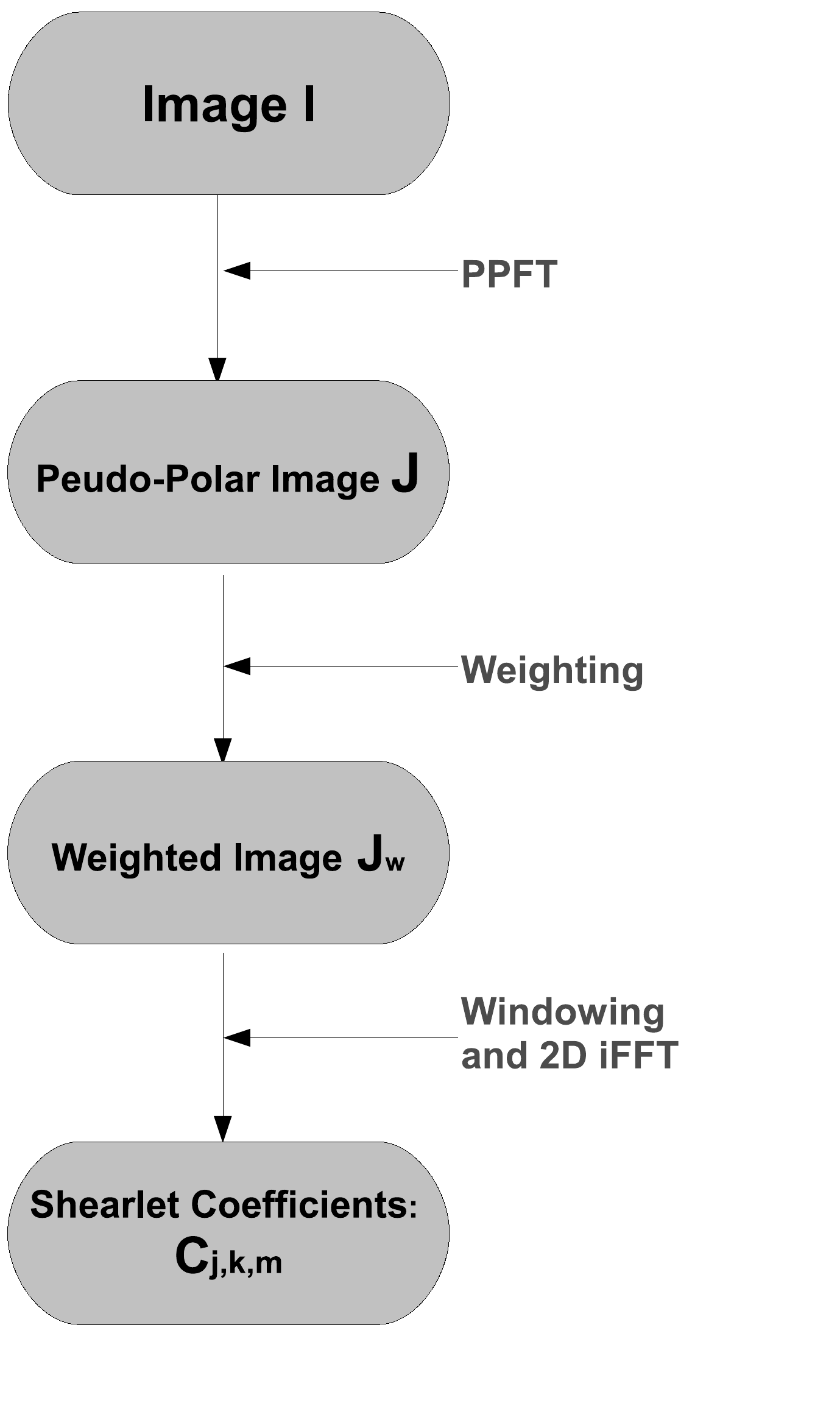}
\includegraphics[height=2.1in]{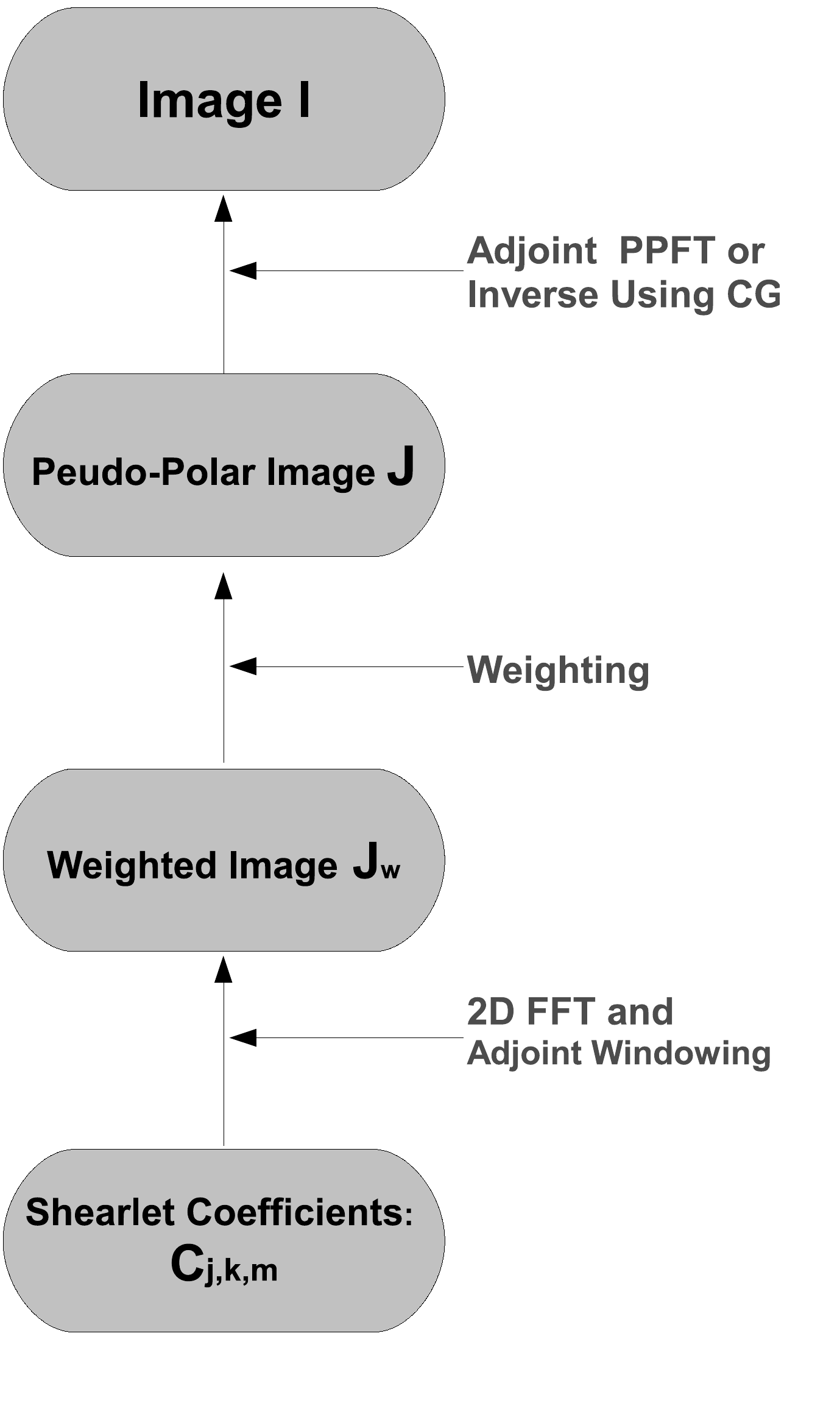}
\end{center}
\vspace{-0.5cm}
\caption{Flowcharts of the FDST (left) and its inverse (right).}
\label{fig:flowcharts}
\end{figure}

For the sake of brevity, we now let $P$, $w$, and $W$ denote the Fast PPFT from Subsection \ref{subsubsec:forward}, the weighting
on the pseudo-polar grid described in Subsection \ref{subsubsec:weighting}, and windowing operator consisting of the application of
the shearlet windows followed by 2D-iFFT to each array as detailed in Subsection \ref{subsubsec:windowing}, respectively.

\subsubsection{FDST}

We can summarize the steps of the algorithm FDST as follows:

\begin{itemize}
\item {\bf Step (S1):} For a given image $I$, apply the Fast PPFT as described in Subsection \ref{subsubsec:forward}
to obtain the function $P I : \Omega_R \to \C$.
\item {\bf Step (S2):} Apply the square root of an off-line computed weight function $w : \Omega_R \to \C$ to $P I$ as described
in Subsection \ref{subsubsec:weighting}, yielding $\sqrt{w} P I : \Omega_R \to \C$.
\item {\bf Step (S3):} Apply the shearlet windows to the function $w P I$, followed by a 2D-iFFT to each array to obtain the
shearlet coefficients $W \sqrt{w} P I$, which we denote by $c^{\iota_0}_{n_0}$, $\iota_0, n_0$ and $c_{j,k,m}^\iota$, $j,k,m,\iota$.
\end{itemize}

\subsubsection{Adjoint FDST}
\label{subsubsec:adjoint}

Assuming that the weight function $w$ used in Step (S2) satisfies the condition in Theorem \ref{thm:weight}, and using the Parseval
frame property of the digital shearlet system (Theorem \ref{thm:DSHtight}), we obtain
\[
(W\sqrt{w}P)^\star W\sqrt{w}P = P^\star\sqrt{w}(W^\star W)\sqrt{w}P = P^\star wP = Id.
\]
Hence in this case, the FDST, which is abbreviated by $W\sqrt{w}P$ can be inverted by applying the adjoint FDST, which
cascades the following steps:
\begin{itemize}
\item {\bf Step 1:} For given shearlet coefficients $C$, i.e., $c^{\iota_0}_{n_0}$, $\iota_0, n_0$ and $c_{j,k,m}^\iota$, $j,k,m,\iota$,
compute the linear combination of the shearlet windows with coefficients $c^{\iota_0}_{n_0}$, $\iota_0, n_0$ and $c_{j,k,m}^\iota$, $j,k,m,\iota$.
This gives the function $W^\star C : \Omega_R \to \C$.
\item {\bf Step 2:} Apply the square root of an off-line computed weight function $w : \Omega_R \to \C$ to $W^\star C$, yielding the
function $\sqrt{w} W^\star C : \Omega_R \to \C$.
\item {\bf Step 3:} Apply the Fast Adjoint PPFT by running the Fast PPFT `backwards'. For this, we just notice that the adjoint
fractional Fourier transform of a vector $c\in\bC^{N+1}$ with respect to a constant $\alpha\in\bC$ is  given by $F_{N+1}^{-\alpha}c$.
Also, for $m>N$, the adjoint padding operator $E_{m,N}^\star$ applied to a vector $c\in\bC^{m}$ is given by $(E^\star_{m,N}c)(k) = c(k)$,
$k=-N/2,\ldots, N/2-1$. The Adjoint PPFT gives an image $ P^\star \sqrt{w} W^\star C$.
\end{itemize}

\subsubsection{Inverse FDST}

Normally -- as also with the relaxed form of weights debated in Subsection \ref{subsubsec:relaxed} -- the weights will not
satisfy the conditions of Theorem \ref{thm:weight} precisely. A measure for whether application of the adjoint is still
feasible
{will be discussed in Subsection \ref{subsec:isometry}.}
If higher accuracy of the reconstruction is required, one might use iterative methods, such as conjugate gradient methods.
Since the digital shearlet system forms a Parseval frame, we always have
\[
W^\star W\sqrt{w}P = \sqrt{w}P.
\]
Hence, iterative methods need to be `only' applied to reconstruct an image $I$ from knowledge of $J := \sqrt{w}P I$, i.e., to solve
the equation
\[
P^\star w P I = P^\star w J
\]
for $I$. Since $J$ might not be in the range of $P$, $I$ is typically computed by solving the weighted least square problem
$\min_{I}\|\sqrt{w}P I-\sqrt{w}J\|_2$. Since the matrix corresponding to $P^\star P$ is symmetric positive definite, iterative
methods such as the conjugate gradient method are applicable. The conjugate gradient method is then applied to the equation
$A x = b$ with $A = P^\star w P$ and $b = P^\star w J$. Its performance can be measured by the condition number of the operator
$P^\star w P$: $cond(P^\star w P) =\lambda_{max}(P^\star w P)/\lambda_{min}(P^\star w P)$, and it turns out that
the weight function serves as a pre-conditioner. We remark that this measure is more closely studied in Subsection
\ref{subsec:isometry}.

To illustrate the behavior of the weights with respect to this measure, in Table \ref{tab:12} we compute $cond(P^\star w P)$ for
the weight functions arising from Choices 1 and 2 (cf. Subsection \ref{subsubsec:relaxed}) with oversampling rate $R=8$.

\begin{table}[ht]
\caption{Comparison of Choices 1 and 2 based on the performance measure $cond(P^\star w P)$.}
\label{tab:12}
\begin{center}
\begin{tabular}{p{2cm}p{1.8cm}p{1.8cm}p{1.8cm}p{1.8cm}p{1.8cm}p{2.0cm}}
\hline\noalign{\smallskip}
$N$ & 32 & 64 & 128 & 256 & 512 \\
\noalign{\smallskip}\hline\noalign{\smallskip}
Choice 1 & 1.379 & 1.503 & 1.621 & 1.731 & 1.833 \\
Choice 2 & 1.760 & 1.887 & 2.001 & 2.104 & N/A \\
\noalign{\smallskip}\hline\noalign{\smallskip}
\end{tabular}
\end{center}
\end{table}


\section{Digital Shearlet Transform using Compactly Supported Shearlets}
\label{sec:dst}

In this section, we will discuss two implementation strategies for computing shearlet coefficients associated with a cone-adapted
discrete shearlet system now based on {\em compactly supported} shearlets, as introduced in Chapter \cite{Introduction}. Again,
one main focus will be on deriving a digitization which is faithful to the continuum setting.

Recall that in the context of wavelet theory, faithful digitization is achieved by the concept of multiresolution analysis, where
scaling and translation are digitized by discrete operations: Downsampling, upsampling and convolution. In the case of directional
transforms however, {\em three} types of operators: Scaling, translation and direction, need to be digitized. In this section,
we will pay particular attention to deriving a framework in which each of the three operators is faithfully interpreted as a
digitized operation in digital domain. Both approaches will be based on the following digitization strategies:
\begin{itemize}
\item Scaling and translation: A multiresolution analysis associated with anisotropic scaling $A_{2^j}$ can be applied for each shear
parameter $k$.
\item Directionality:
A faithful digitization of shear operator $S_{2^{-j/2}k}$ has to be achieved with particular care.
\end{itemize}
After stating and discussing the two main obstacles we are facing when considering compactly supported shearlets in Subsection
\ref{subsec:problems}, we present the digital separable shearlet transform (DSST), which is associated with a shearlet system
generated by a separable function  alongside with discussions on its properties, e.g., its redundancy; see Subsection \ref{subsec:dst_dst}.
Subsection \ref{subsec:DNST} then presents the digital non-separable shearlet transform (DNST), whose shearlet elements are generated by non-
separable shearlet generator.

\subsection{Problems with Digitization of Compactly Supported Shearlets}
\label{subsec:problems}

Compactly supported shearlets have several advantages, and we exemplarily mention superior spatial localization and
simplified boundary adaptation. However, we have to face the following two problems:
\begin{enumerate}
\item[(P1)] Compactly supported shearlets do not form a tight frame, which prevents utilization of the adjoint as inverse transform.
\item[(P2)] There does not exist a natural hierarchical structure, mainly due to the application of a shear matrix, which -- unlike for
the wavelet transform -- does not allow a multiresolution analysis without destroying a faithful adaption of the continuum setting.
\end{enumerate}

Let us now comment on these two obstacles, before delving into the details of the implementation in Subsection \ref{subsec:dst_dst}.

\subsubsection{Tightness}
\label{subsubsec:tight}

Let us first comment on the problem of non-tightness. Letting $(\sigma_i)_{i \in I}$ denote a frame  for $L^2(\R^2)$ -- for example,
a shearlet frame --, each function $f \in L^2(\R^2)$ can be reconstructed from its frame coefficients $( \langle f,\sigma_i\rangle)_{i \in I}$
by
\[
f = \sum_{i \in I} \langle f,\sigma_i\rangle S^{-1} (\sigma_i),
\]
where $S = \sum_{i \in I} \langle \cdot ,\sigma_i\rangle \sigma_i$ is the associated frame operator on $L^2(\R^2)$, see
Chapter \cite{Introduction}. However, in case that $(\sigma_i)_{i \in I}$ does not form a {\em tight} frame, it is
in general difficult to explicitly compute the dual frame elements $S^{-1}(\sigma_i)$.

Nevertheless, it is well known that the inverse frame operator $S^{-1}$ can be effectively approximated using iterative schemes such
as the Conjugate Gradient method provided that the frame $(\sigma_i)_{i \in I}$ has 'good' frame bounds in the sense of their
ratio being `close'
to $1$, see also \cite{M099}. Therefore, now focussing on the situation of shearlet frames, we may argue that input data $f$ can
be efficiently reconstructed from its shearlet coefficients, if we use a compactly supported shearlet frame with 'good' frame bounds.
In fact, the theoretical frame bounds of compactly supported shearlet frames have been theoretically estimated as well as numerically
computed in \cite{KKL2010}. These results were derived for the class of 2D separable shearlet generators $\psi$ already described in Chapter
\cite{Introduction}, which we briefly recall for the convenience of the reader:

For positive integers $K$ and $L$ fixed, let the 1D lowpass filter $m_0$ be defined by
$$
|m_0(\xi_1)|^2 = (\cos(\pi\xi_1))^{2K}\sum_{n=0}^{L-1} {K-1+n\choose
  n}(\sin(\pi\xi_1))^{2n},
$$
for $\xi_1 \in \R$. Further, define the associated bandpass filter $m_1$ by
$$
|m_1(\xi_1)|^2 = |m_0(\xi_1+1/2)|^2, \quad \xi_1 \in \R,
$$
and the 1D scaling function $\phi_1$ by
$$
\hat{\phi_1}(\xi_1) = \prod_{j=0}^{\infty} m_0(2^{-j}\xi_1), \quad \xi_1 \in \R.
$$
Using the filter $m_1$ and scaling function $\phi_1$, we now define the 2D scaling function $\phi$ and separable shearlet generator $\psi$ by
\beq \label{eq:csGenerator}
\hat \phi(\xi_1,\xi_2) = \hat \phi_1(\xi_1)\hat \phi_1(\xi_2) \quad \text{and} \quad \hat \psi(\xi_1,\xi_2) = m_1(4\xi_1)\hat \phi_1(\xi_1)\hat \phi_1(2\xi_2).
\eeq
In \cite{KKL2010}, it was shown that compactly supported shearlets $\psi_{j,k,m}$ generated by the shearlet
generator $\psi$ form a frame for $L^2(C)^\vee$ with appropriately chosen parameters $K$ and $L$, where
$$
C = \{\xi \in \R^2 : |\xi_2/\xi_1| \leq 1, \,\, |\xi_1| \ge 1 \}.
$$
This construction is directly extended to construct a cone-adapted discrete shearlet frame for $L^2(\R^2)$  (cf. also Chapters \cite{Introduction}
and \cite{SparseApproximation}).

Table \ref{table:cs} provides some numerically estimated frame bounds in $L^2(C)$ for certain choice of $K$ and $L$. It shows that indeed the
ratio of the frame bounds of this class of compactly supported shearlet frames is sufficient small for utilizing an iterative scheme for efficient
reconstruction; in this sense the frame bounds are 'good'.
\begin{table}[h]
\caption{Numerically estimated frame bounds for various choices of the parameters $K$ and $L$.
$c_1$ and $c_2$ are the sampling constants in the sampling matrix $M_c$ for translation (see Chapter \cite{Introduction}).}
\label{table:cs}
\begin{center}
\begin{tabular}{p{2cm}p{2.25cm}p{2.25cm}p{2.25cm}p{2.25cm}}
\hline\noalign{\smallskip}
$K$ & $L$ & $c_1$ & $c_2$ & B/A \\
\noalign{\smallskip}\hline\noalign{\smallskip}
39 & 19 & 0.90 & 0.15  &  4.1084 \\
39 & 19 & 0.90 & 0.20  &  4.1085 \\
39 & 19 & 0.90 & 0.25 &   4.1104 \\
39 & 19 & 0.90 & 0.30 &   4.1328 \\
39 & 19 & 0.90 & 0.40  &  5.2495 \\
\noalign{\smallskip}\hline\noalign{\smallskip}
\end{tabular}
\end{center}
\end{table}

The frequency covering by compactly supported shearlets $\psi_{j,k,m}$,
\[
|\hat \phi(\xi)|^2+\sum_{j \ge 0}\sum_{k \in K_j} |\hat \psi(S^T_kA_{2^j}\xi)|^2 +|\hat{\tilde{\psi}}(\tilde{S}^T_k\tilde{A}_{2^j}\xi)|^2,
\]
is closely related to the ratio of frames bounds and, in particular, which areas in frequency domain cause a larger ratio.
This function is illustrated in Fig. \ref{fig:tile}, which shows that its upper and lower bounds are as expected well controlled.
\begin{figure}[h]
\begin{center}
\includegraphics[height=1.2in]{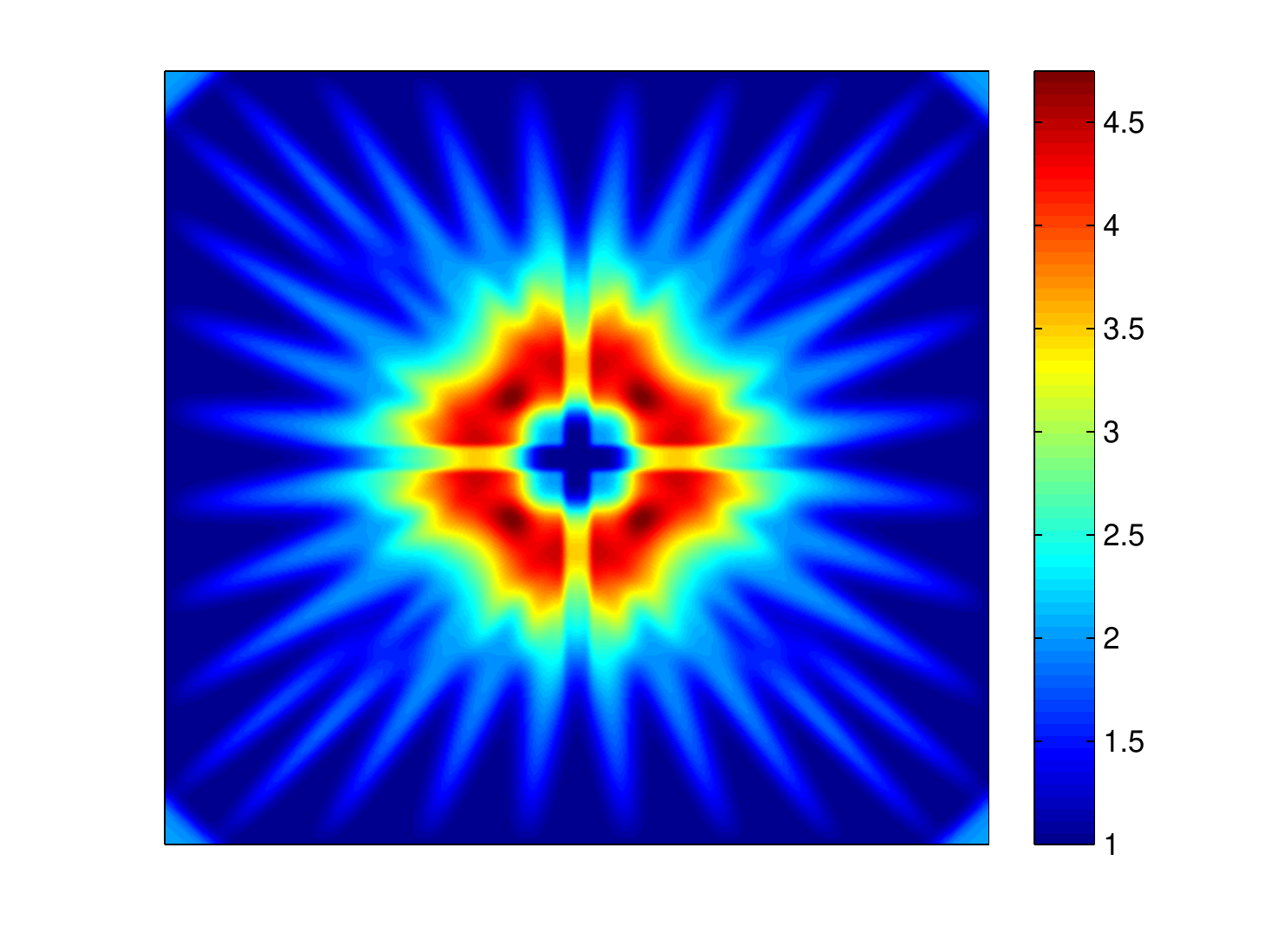}
\includegraphics[width=1.3in,height=1.2in]{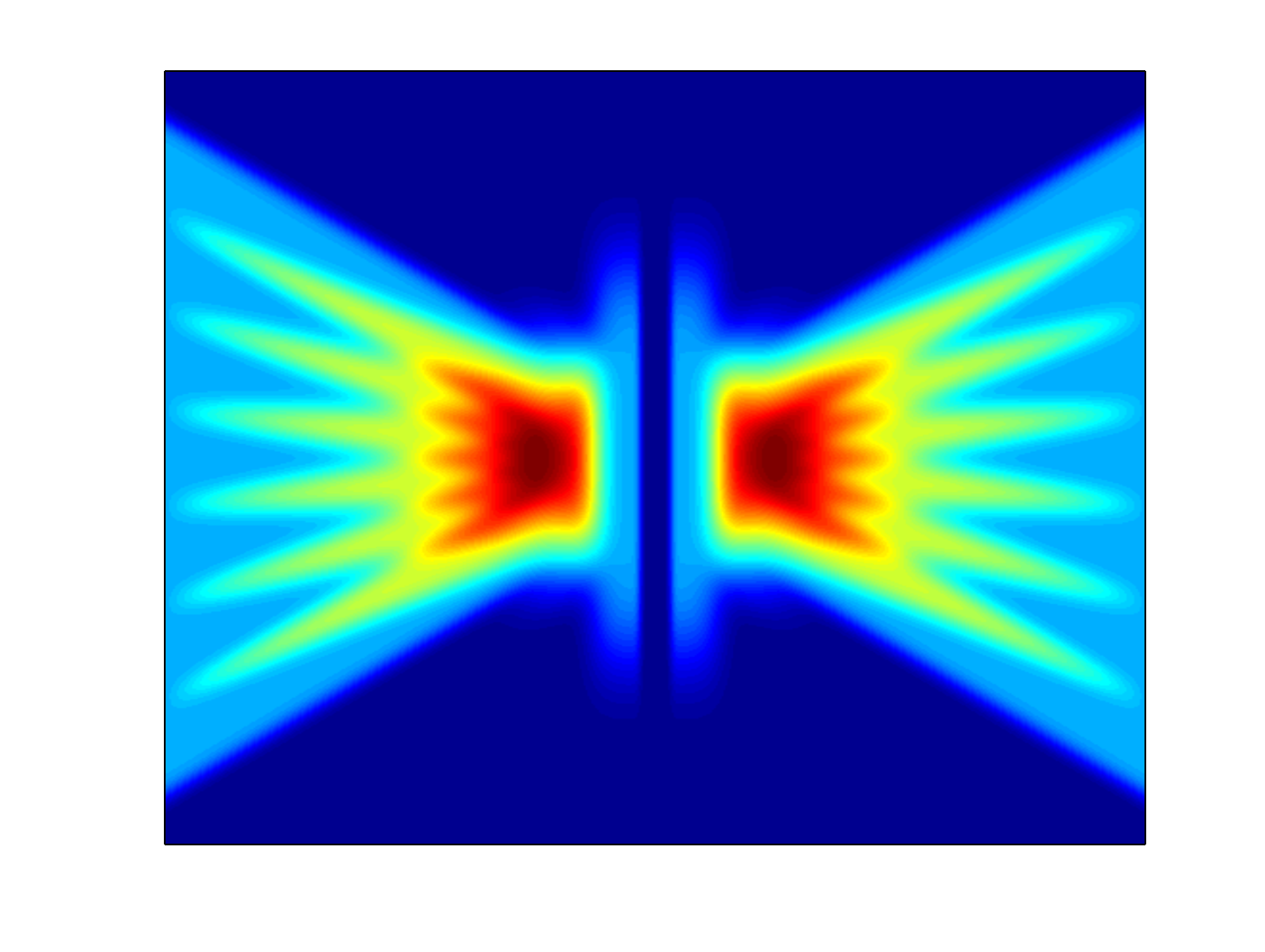}
\includegraphics[width=1.3in,height=1.2in]{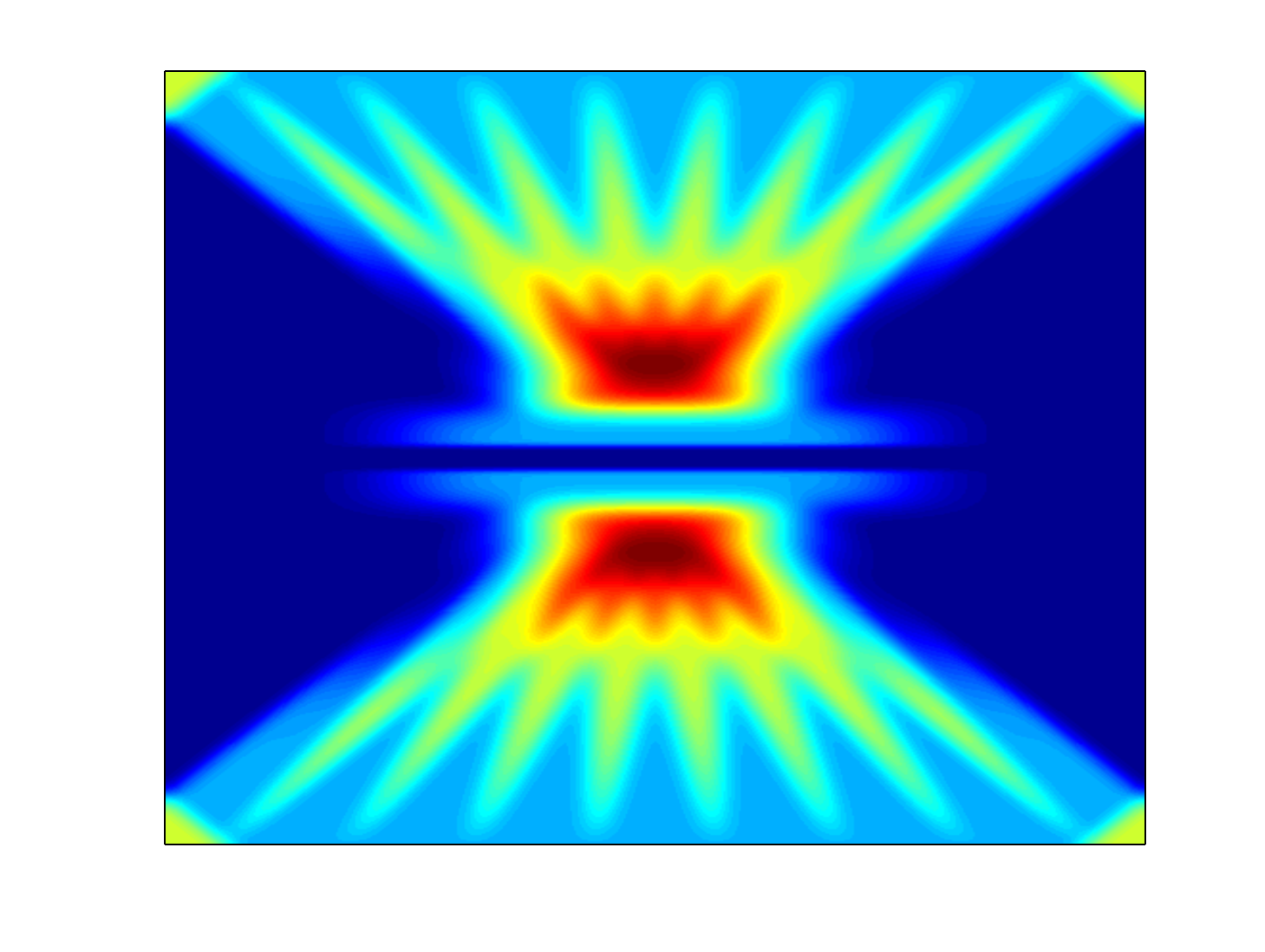}
\put(-300,0){(a) Whole frequency plane.}
\put(-175,0){(b) Horizontal cone. }
\put(-75,0){(c) Vertical cone.}
\end{center}
\caption{Frequency covering by shearlets $|\hat \psi_{j,k,m}|^2$: (a) Frequency covering of the entire frequency plane. (b) Frequency covering
of the horizontal cone. (c) Frequency covering of the vertical cone.}\label{fig:tile}
\end{figure}

\subsubsection{Hierarchical Structure}

Let us finally comment on the problem to achieve a hierarchical structuring. To allow fast implementations, the data structure of the transform is essential.
The hierarchical structure of the wavelet transform associated with a multiresolution analysis, for instance, enables a fast implementation based on filterbanks.
In addition, such a hierarchical ordering provides a full tree structure across scales, which is of particular importance for various applications such as
image compression and adaptive PDE schemes. It is in fact mainly due to this property -- and the unified treatment of the continuum and digital setting --
that the wavelet transform became an extremely successful methodology for many practical applications.

From a certain viewpoint, shearlets $\psi_{j,k,m}$ can essentially be regarded as wavelets associated with an anisotropic scale matrix $A_{2^j}$, when
the shear parameter $k$ is fixed. This observation allows to apply the wavelet transform to compute the shearlet coefficients, once the shear operation
is computed for each shear parameter $k$. This approach will be undertaken in the digital formulation of the compactly supported shearlet transform, and,
in fact, this approach implements a hierarchical structure into the shearlet transform. The reader should note that this approach does not lead to a completely
hierarchical structured shearlet transform -- also compare our discussion at the beginning of this section --, but it will be sufficient for deriving a
fast implementation while retaining a faithful digitization.

\subsection{Digital Separable Shearlet Transform (DSST)}
\label{subsec:dst_dst}

We now describe a faithful digitization of the continuum domain shearlet transform based on compactly supported shearlets as introduced in
 \cite{Lim2010}, which moreover is highly computationally efficient.

\subsubsection{Faithful Digitization of the Compactly Supported Shearlet Transform}
\label{subsubsec:digitizationOfCSST}

We start by discussing those theoretical aspects which allow a faithful digitization of the shearlet transform associated with the shearlet
system generated by \eqref{eq:csGenerator}. For this, we will only consider shearlets $\psi_{j,k,m}$ for the horizontal cone, i.e.,
belonging to $\Psi(\psi,c)$. Notice that the same procedure can be applied to compute the shearlet coefficients for the vertical cone,
i.e., those belonging to $\tilde\Psi(\tilde\psi,c)$, except for switching the order of variables.

To construct a separable shearlet generator $\psi \in L^2(\R^2)$ and an associated scaling function $\phi \in L^2(\R^2)$, let
$\phi \in L^2(\R)$ be a compactly supported 1D scaling function satisfying
\begin{equation}\label{eq:scale1}
\phi_1(x_1) = \sum_{n_1 \in \Z} h(n_1) \sqrt{2} \phi_1(2x_1-n_1)
\end{equation}
for some `appropriately chosen' filter $h$ --  we comment on the required condition below. An associated compactly supported 1D wavelet $\psi_1 \in L^2(\R)$
can then be defined by \begin{equation}\label{eq:scale2}
\psi_1(x_1) = \sum_{n_1 \in \Z} g(n_1) \sqrt{2} \phi_1(2x_1-n_1),
\end{equation}
where again $g$ is an `appropriately chosen' filter. The selected shearlet generator is then defined to be
\beq \label{eq:generatorHere}
\psi(x_1,x_2) = \psi_1(x_1)\phi_1(x_2),
\eeq
and the scaling function by
\[
\phi(x_1,x_2) = \phi_1(x_1)\phi_1(x_2).
\]

Let us comment on whether this is indeed a special case of the shearlet generators defined in \eqref{eq:csGenerator}.
The Fourier transform of $\psi$ defined in \eqref{eq:generatorHere} takes the form
\[
\hat \psi(\xi_1,\xi_2) = m_1(\xi_1/2)\hat \phi_1(\xi_1/2)\hat \phi_1(\xi_2/2),
\]
where $m_1$ is a trigonometric polynomial whose Fourier coefficients are $g(n_1)$. We need to compare this
expression with the Fourier transform of the shearlet generator $\psi$ given in \eqref{eq:csGenerator}, which is
\[
\hat \psi(\xi_1,\xi_2) = m_1(4\xi_1)\hat \phi_1(2\xi_1)\hat \phi_1(\xi_2),
\]
with 1D scaling function $\phi_1$ defined in \eqref{eq:scale1}. We remark that this later scaling function is
slightly different defined as in \eqref{eq:csGenerator}. This small adaption is for the sake of presenting a
simpler version of the implementation; essentially the same implementation strategy as the one we will describe
can be applied to the shearlet generator given in \eqref{eq:csGenerator}.

The filter coefficients $h$ and $g$ are required to be chosen so that $\psi$ satisfies a certain decay condition
(cf. \cite{KKL2010} of Chapter \cite{Introduction}) to guarantee a stable reconstruction from the shearlet coefficients.

For the signal $f \in L^2(\R^2)$ to be analyzed, we now assume that, for $J>0$ fixed, $f$ is of the form
\beq \label{eq:code2}
f(x) = \sum_{n \in \Z^2} f_J(n) 2^J \phi(2^Jx_1-n_1,2^Jx_2-n_2).
\eeq
Let us mention that this is a very natural assumption for a digital implementation in the sense that
the scaling coefficients can be viewed as sample values of $f$ -- in fact $f_J(n) = f(2^{-J}n)$ with appropriately
chosen $\phi$. Now aiming towards a faithful digitization of the shearlet coefficients $\langle f,\psi_{j,k,m}\rangle$
for $j = 0,\dots,J-1$, we first observe that
\begin{equation} \label{eq:code1}
\langle f,\psi_{j,k,m}\rangle = \langle f(S_{2^{-j/2}k}(\cdot)),\psi_{j,0,m}(\cdot)\rangle,
\end{equation}
and, WLOG we will from now on assume that $j/2$ is integer; otherwise either $\lceil j/2 \rceil$ or $\lfloor j/2 \rfloor$
would need to be taken. Our observation \eqref{eq:code1} shows us in fact precisely how to digitize the shearlet coefficients
$\langle f, \psi_{j,k,m} \rangle$: By applying the discrete separable wavelet transform associated with the anisotropic
sampling matrix $A_{2^{j}}$ to the sheared version of the data $f(S_{2^{-j/2}k}(\cdot))$.
This however requires -- compare the assumed form of $f$ given in \eqref{eq:code2} -- that
$f(S_{2^{-j/2}k}( \cdot ))$ is contained in the scaling space
$$
V_J = \{2^J\phi(2^J\cdot-n_1,2^J\cdot-n_2) : (n_1,n_2) \in \Z^2\}.
$$
It is easy to see that, for instance, if the shear parameter $2^{-j/2}k$ is non-integer, this is unfortunately not the case. The true
reason for this failure is that the shear matrix $S_{2^{-j/2}k}$ does {\em not} preserve the regular grid $2^{-J}\Z^2$ in $V_J$, i.e.,
\[
S_{2^{-j/2}k}(\Z^2) \neq \Z^2.
\]
In order to resolve this issue, we consider the new scaling space $V^k_{J+j/2,J}$ defined by
$$
V^k_{J+j/2,J} = \{2^{J+4/j}\phi(S_k(2^{J+j/2}\cdot-n_1,2^J\cdot-n_2)) : (n_1,n_2) \in \Z^2\}.
$$
We remark that the scaling space $V^k_{J+j/2,J}$ is obtained by refining the regular grid $2^{-J}\Z^2$ along the $x_1$-axis by a factor of
$2^{j/2}$. With this modification, the new grid $2^{-J-j/2}\Z \times 2^{-J}\Z$ is now invariant under the shear operator $S_{2^{-j/2}k}$,
since with $Q = \text{diag}(2,1)$,
\begin{eqnarray*}
2^{-J-j/2}\Z\times 2^{-J}\Z &=& 2^{-J}Q^{-j/2}(\Z^2) = 2^{-J}Q^{-j/2}(S_k(\Z^2))\\
&=& S_{2^{-j/2}k}(2^{-J-j/2}\Z\times 2^{-J}\Z).
\end{eqnarray*}
This allows us to rewrite $f(S_{2^{-j/2}k}( \cdot ))$ in \eqref{eq:code1} in the
following way.

\begin{lemma}
\label{lemm:rewriting}
Retaining the notations and definitions from this subsection, letting $\uparrow 2^{j/2}$ and $*_{1}$ denote the 1D upsampling operator by
a factor of $2^{j/2}$ and the 1D convolution operator along the $x_1$-axis, respectively, and setting $h_{j/2}(n_1)$ to be the Fourier
coefficients of the trigonometric polynomial
\beq \label{eq:poly1}
H_{j/2}(\xi_1) = \prod_{k=0}^{j/2-1}\sum_{n_1 \in \Z}h(n_1) e^{-2\pi i 2^k n_1 \xi_1},
\eeq
we obtain
\[
f(S_{2^{-j/2}k}(x)) = \sum_{n \in \Z^2} \tilde{f}_J(S_{k}{n})2^{J+j/4}\phi_{k}(2^{J+j/2}x_1-n_1,2^{J}x_2-n_2),
\]
where
\[
\tilde{f}_J(n) = ((f_J)_{\uparrow 2^{j/2}} *_{1} h_{j/2})(n).
\]
\end{lemma}

The proof of this lemma requires the following result, which follows from the cascade algorithm in the theory of wavelet.

\begin{proposition}[\cite{Lim2010}]\label{prop:cascade}
Assume that $\phi_1$ and $\psi_1 \in L^2(\R)$ satisfy equations \eqref{eq:scale1} and \eqref{eq:scale2} respectively.
For positive integers $j_1 \leq j_2$, we then have
\begin{equation}\label{eq:across1}
2^{\frac{j_1}{2}}\phi_1(2^{j_1}x_1-n_1) = \sum_{d_1 \in \Z}h_{j_2-j_1}(d_1-2^{j_2-j_1}n_1)2^{\frac{j_2}{2}}\phi_1(2^{j_2}x_1-d_1)
\end{equation}
and
\begin{equation}\label{eq:across2}
2^{\frac{j_1}{2}}\psi_1(2^{j_1}x_1-n_1) = \sum_{d_1 \in \Z}g_{j_2-j_1}(d_1-2^{j_2-j_1}n_1)2^{\frac{j_2}{2}}\phi_1(2^{j_2}x_1-d_1),
\end{equation}
where $h_j$ and $g_j$ are the Fourier coefficients of the trigonometric polynomials $H_j$ defined in \eqref{eq:poly1} and $G_j$ defined by
$$
G_{j}(\xi_1) = \Bigl(\prod_{k=0}^{j-2}\sum_{n_1 \in \Z}h(n_1) e^{-2\pi i 2^k n_1 \xi_1}\Bigr)\Bigl( \sum_{n_1 \in \Z}g(n_1) e^{-2\pi i 2^{j-1} n_1 \xi_1}\Bigr)
$$
for $j>0$ fixed.
\end{proposition}

\begin{proof}[Proof of Lemma \ref{lemm:rewriting}]
Equation \eqref{eq:across1} with $j_1 = J$ and $j_2 = J+j/2$ implies that
\begin{equation}\label{eq:refine}
2^{J/2}\phi_1(2^Jx_1-n_1) = \sum_{d_1 \in \Z} h_{J-j/2}(d_1-2^{j/2}n_1)2^{J/2+j/4}\phi_1(2^{J+j/2}x_1-d_1).
\end{equation}
Also, since $\phi$ is a 2D separable function of the form $\phi(x_1,x_2) = \phi_1(x_1)\phi_1(x_2)$, we have that
\[
f(x) = \sum_{n_2 \in \Z}\Bigl(\sum_{n_1 \in \Z} f_J(n_1,n_2) 2^{J/2}\phi_1(2^Jx_1-n_1)\Bigr) 2^{J/2}\phi_1(2^Jx_2-n_2).
\]
By \eqref{eq:refine}, we obtain
\[
f(x) = \sum_{n \in \Z^2}\tilde{f}_J(n)2^{J+j/4}\phi(2^JQ^{j/2}x-n),
\]
where $Q = \text{diag}(2,1)$. Using $Q^{j/2}S_{2^{-j/2}k} = S_{k}Q^{j/2}$, this finally implies
\begin{eqnarray*}
f(S_{2^{-j/2}k}(x)) &=& \sum_{n \in \Z^2} \tilde{f}_J(n)2^{J+j/4}\phi(2^JQ^{j/2}S_{2^{-j/2}k}(x)-n) \\
&=& \sum_{n \in \Z^2} \tilde{f}_J(n)2^{J+j/4}\phi(S_k(2^JQ^{j/2}x-S_{-k}n)) \\
&=& \sum_{n \in \Z^2} \tilde{f}_J(S_kn)2^{J+j/4}\phi(S_k(2^JQ^{j/2}x-n)).
\end{eqnarray*}
The lemma is proved. \qed
\end{proof}

The second term to be digitized in \eqref{eq:code1} is the shearlet $\psi_{j,k,m}$ itself. A direct
corollary from Proposition \ref{prop:cascade} is the following result.

\begin{lemma}
\label{lemm:rewriting2}
Retaining the notations and definitions from this subsection, we obtain
\[
\psi_{j,k,m}(x) = \sum_{d \in \Z^2} g_{J-j}(d_1-2^{J-j}m_1)h_{J-j/2}(d_2-2^{J-j/2}m_2)2^{J+j/4}{\phi(2^Jx-d).}
\]
\end{lemma}

As already indicated before, we will make use of the discrete separable wavelet transform associated with an anisotropic scaling matrix,
which, for $j_1$ and $j_2>0$ as well as $c \in \ell(\Z^2)$, we define by
\begin{equation}\label{eq:wavelet}
W_{j_1,j_2}(c)(n_1,n_2) = \sum_{m \in \Z^2}g_{j_1}(m_1-2^{j_1}n_1)h_{j_2}(m_2-2^{j_2}n_2)c(m_1,m_2), \quad  (n_1,n_2) \in \Z^2.
\end{equation}

Finally, Lemmata \ref{lemm:rewriting} and \ref{lemm:rewriting2} yield the following digitizable form of the shearlet coefficients $\langle f,\psi_{j,k,m}\rangle$.

\begin{theorem}[\cite{Lim2010}]\label{theo:Lim}
Retaining the notations and definitions from this subsection, and letting$\downarrow 2^{j/2}$
be 1D downsampling by a factor of $2^{j/2}$ along the horizontal axis, we obtain
\[
\langle f,\psi_{j,k,m} \rangle = W_{J-j,J-j/2}\Bigl( \Bigl((\tilde{f}_J(S_k\cdot)*\Phi_k) *_1 \overline{h}_{j/2} \Bigr)_{\downarrow 2^{j/2}}\Bigr)(m),
\]
where $\Phi_k(n) = \langle \phi(S_k(\cdot)), \phi(\cdot-n)\rangle$ for $n \in \Z^2$, and $\overline{h}_{j/2}(n_1) = h_{j/2}(-n_1)$.
\end{theorem}

\subsubsection{Algorithmic Realization}
\label{subsubsec:algorithm}

Computing the shearlet coefficients using Theorem \ref{theo:Lim} now restricts to applying the discrete separable wavelet transform \eqref{eq:wavelet}
associated with the sampling matrix $A_{2^j}$ to the scaling coefficients
\begin{equation}\label{eq:dshear}
S^d_{2^{-j/2}k}(f_J)(n) := \Bigl((\tilde{f}_J(S_k\cdot)*\Phi_k) *_1 \overline{h}_{j/2} \Bigr)_{\downarrow 2^{j/2}}(n) \quad \text{for} \quad f_J \in \ell^2(\Z^2).
\end{equation}
Before we state the explicit steps necessary to achieve this, let us take a closer look at the scaling coefficients $S^d_{2^{-j/2}k}(f_J)$, which
can be regarded as a new sampling of the data $f_J$ on the integer grid $\Z^2$ by the digital shear operator $S^d_{2^{-j/2}k}$. This procedure is
illustrated in Fig. \ref{fig:grid} in the case $2^{-j/2}k = -1/4$.
\begin{figure}[t]
\includegraphics[width=1.5in]{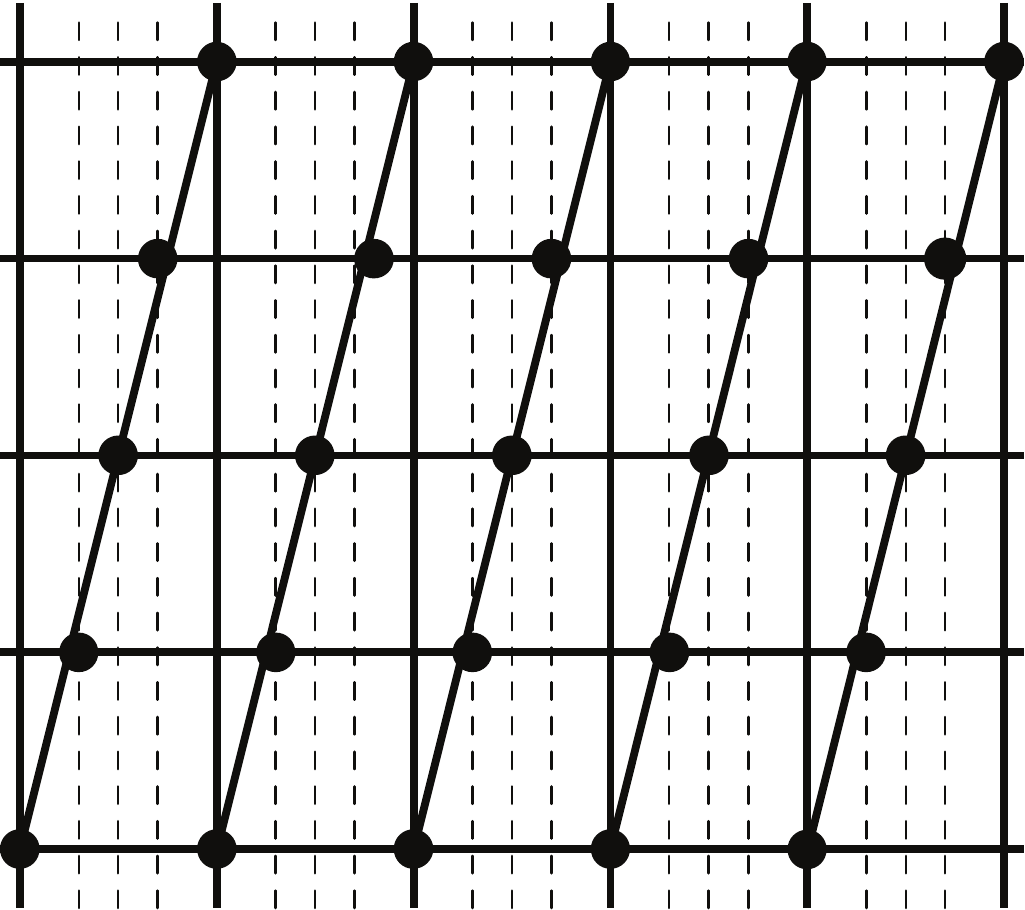}
\caption{Illustration of application of the digital shear operator $S^d_{-1/4}$: The dashed lines correspond to the refinement of the integer
grid. The new sample values lie on the intersections of the sheared lines associated with $S_{1/4}$ with this refined grid. }
\label{fig:grid}
\end{figure}

Let us also mention that the filter coefficients $\Phi_k(n)$ in \eqref{eq:dshear} can in fact be easily precomputed for each shear parameter $k$. For a
practical implementation, one may sometimes even skip this additional convolution step assuming that $\Phi_k = \chi_{(0,0)}$.

\bigskip

Concluding, the implementation strategy for the DSST cascades the following steps:
\begin{itemize}
\item {\bf Step 1:} For given input data $f_J$, apply the 1D upsampling operator by a factor of $2^{j/2}$ at the finest scale $j = J$.
\item {\bf Step 2:} Apply 1D convolution to the upsampled input data $f_J$ with 1D lowpass filter $h_{j/2}$ at the finest scale $j = J$.
This gives $\tilde{f}_J$.
\item {\bf Step 3:} Resample $\tilde{f}_J$ to obtain $\tilde{f}_J(S_k(n))$ according to the shear sampling matrix $S_k$ at the finest scale $j = J$.
Note that this resampling step is straightforward, since the integer grid is invariant under the shear matrix $S_k$.
\item {\bf Step 4:} Apply 1D convolution to $\tilde{f}_J(S_k(n))$ with $\overline{h}_{j/2}$ followed by 1D downsampling by a factor of $2^{j/2}$ at
the finest scale $j = J$.
\item {\bf Step 5:} Apply the separable wavelet transform $W_{J-j,J-j/2}$ across scales $j = 0,1,\dots,J-1$.
\end{itemize}

\subsubsection{Digital Realization of Directionality}
\label{subsec:direction}

Since the digital realization of a shear matrix $S_{2^{-j/2}k}$ by the digital shear operator $S^d_{2^{-j/2}k}$ is crucial for deriving a
faithful digitization of the continuum domain shearlet transform, we will devote this subsection to a closer analysis.

We start by remarking that in fact in the continuum domain, at least {\em two} operators exist which naturally provide directionality: Rotation and
shearing. Rotation is a very convenient tool to provide directionality in the sense that it preserves important geometric information such
as length, angles, and parallelism. However, this operator does not preserve the integer lattice, which causes severe problems for digitization.
In contrast to this, a shear matrix $S_k$ does not only provide directionality, but also preserves the integer lattice when the shear parameter $k$
is integer. Thus, it is conceivable to assume that directionality can be naturally discretized by using a shear matrix $S_k$.

To start our analysis of the relation between a shear matrix $S_{2^{-j/2}k}$ and the associated digital shear operator $S^d_{2^{-j/2}k}$,
let us consider the following simple example: Set $f_c = \chi_{\{x : x_1 = 0\}}$. Then digitize $f_c$ to obtain a function $f_d$ defined
on $\Z^2$ by setting $f_d(n) = f_c(n)$ for all $n \in \Z^2$. For fixed shear parameter $s \in \R$, apply the shear transform $S_s$ to $f_c$
yielding the sheared function $f_c(S_s( \cdot ))$. Next, digitize also this function by considering $f_c(S_s( \cdot ))|_{\Z^2}$.
The functions $f_d$ and $f_c(S_s( \cdot ))|_{\Z^2}$ are illustrated in Fig.~\ref{fig:lines} for $s = -1/4$.
\begin{figure}[h]
\begin{center}
\includegraphics[height=1.0in]{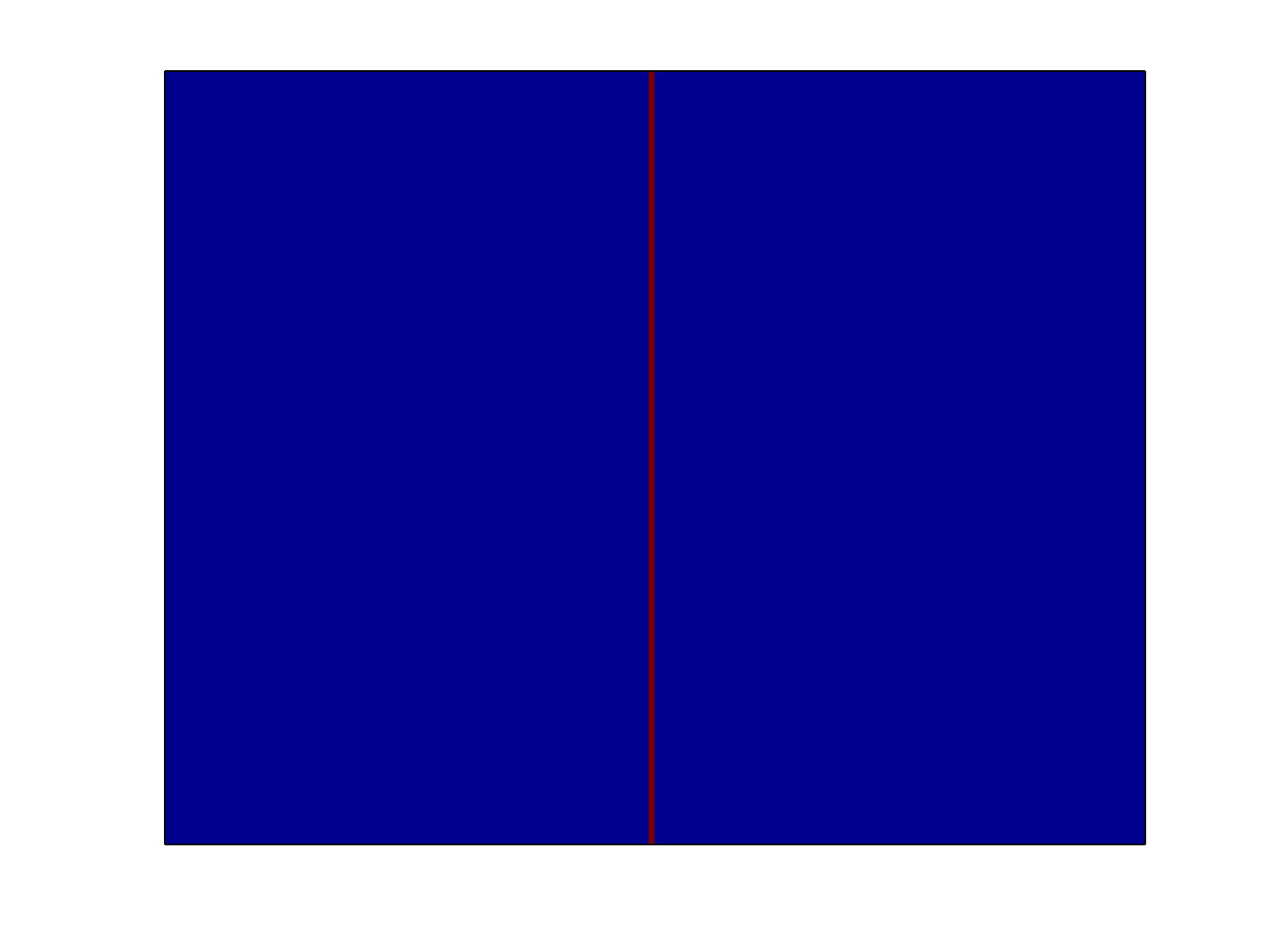}
\hspace*{1cm}
\includegraphics[height=1.0in]{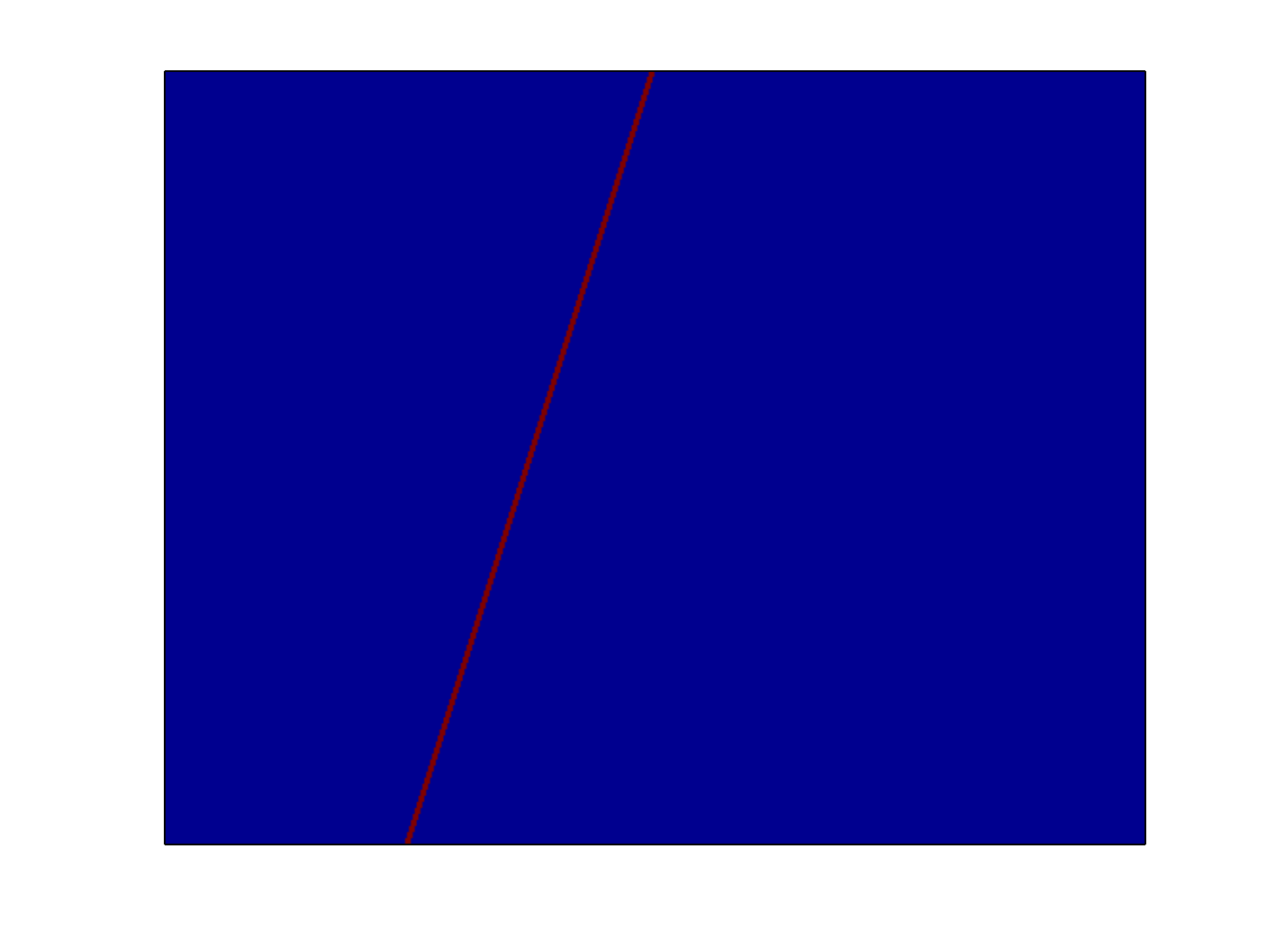}
\put(-180,0){(a)}
\put(-50,0){(b)}
\caption{(a) Original image $f_d(n)$. (b) Sheared image $f_c(S_{-1/4}n)$.}\label{fig:lines}
\end{center}
\end{figure}
We now focus on the problem that the integer lattice is not invariant under the shear matrix $S_{1/4}$. This prevents the sampling points $S_{1/4}(n)$,
$n \in \Z^2$ from lying  on the integer grid, which causes aliasing of the digitized image $f_c(S_{-1/4}( \cdot ))|_{\Z^2}$ as illustrated in
Fig.~\ref{fig:aliase}(a). In order to avoid this aliasing effect, the grid needs to be refined by a factor of 4 along the horizontal axis
followed by computing sample values on this refined grid.

More generally, when the shear parameter is given by $s = -2^{-j/2}k$, one can essentially avoid this directional aliasing effect by
refining a grid by a factor of $2^{j/2}$ along the horizontal axis followed by computing interpolated sample values on this refined grid.
This ensures that the resulting grid contains the sampling points $((2^{-j/2}k)n_2,n_2)$ for any $n_2 \in \Z$ and is
preserved by the shear matrix $S_{-2^{-j/2}k}$.
This procedure precisely coincides with the application of the digital shear operator $S^d_{2^{-j/2}k}$, i.e., we just described Steps 1 -- 4 from
Subsection \ref{subsubsec:algorithm} in which the new scaling coefficients $S^d_{2^{-j/2}k}({f}_J)(n)$ are computed.

Let us come back to the exemplary situation of $f_c = \chi_{\{x : x_1 = 0\}}$ and $S_{-1/4}$ we started our excursion with and compare
$f_c(S_{-1/4}( \cdot ))|_{\Z^2}$ with $S^d_{-1/4}(f_d)|_{\Z^2}$ obtained by applying the digital shear operator $S^d_{-1/4}$ to $f_d$.
And, in fact, the directional aliasing effect on the digitized image $f_c(S_{-1/4}(n))$ in frequency illustrated in Fig.~\ref{fig:aliase}(a)
is shown to be avoided in Fig.~\ref{fig:aliase} (b)-(c) by considering $S^d_{-1/4}(f_d)|_{\Z^2}$.
\begin{figure}[h]
\begin{center}
\includegraphics[height=1.0in]{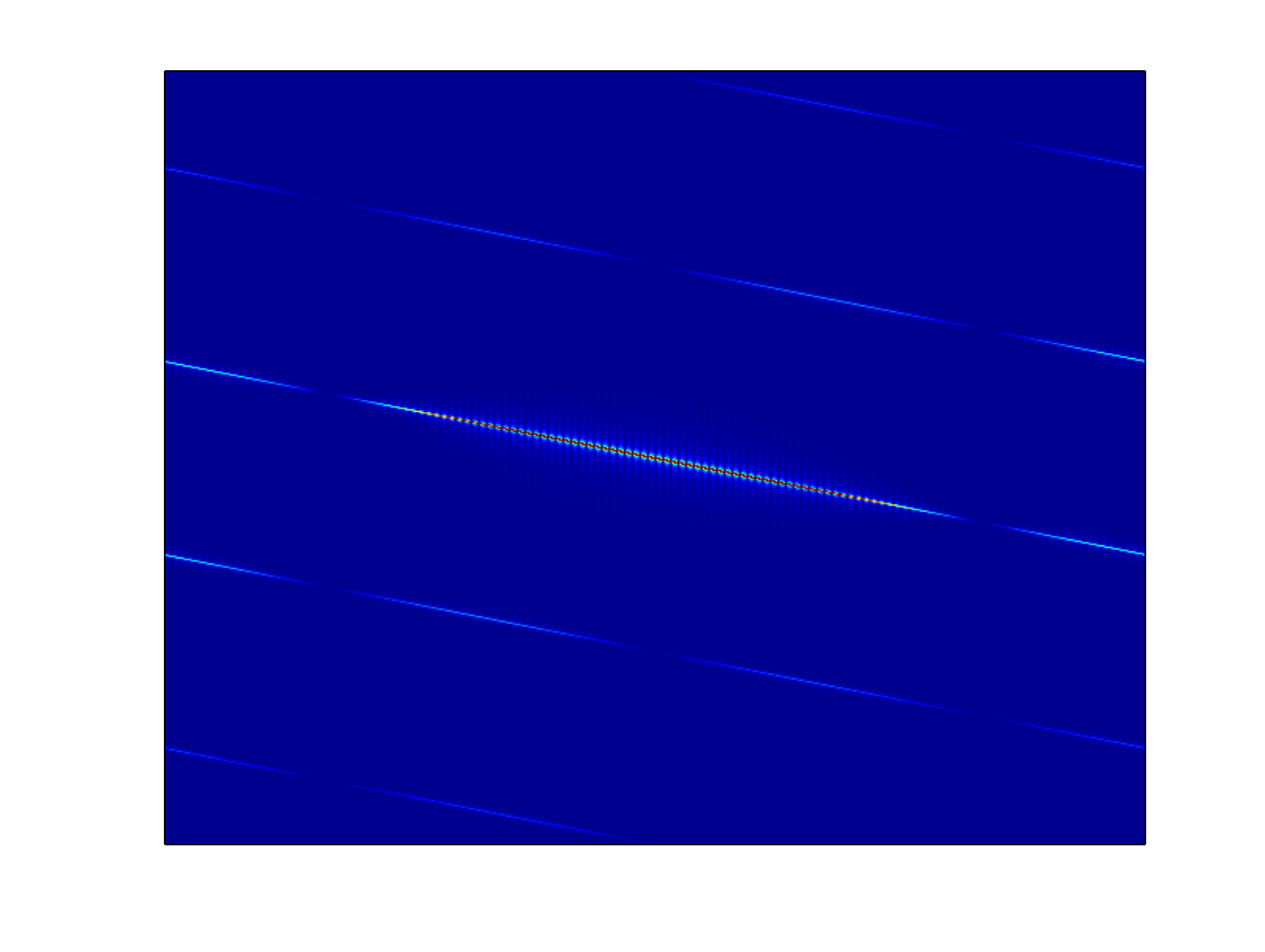}
\includegraphics[height=1.0in]{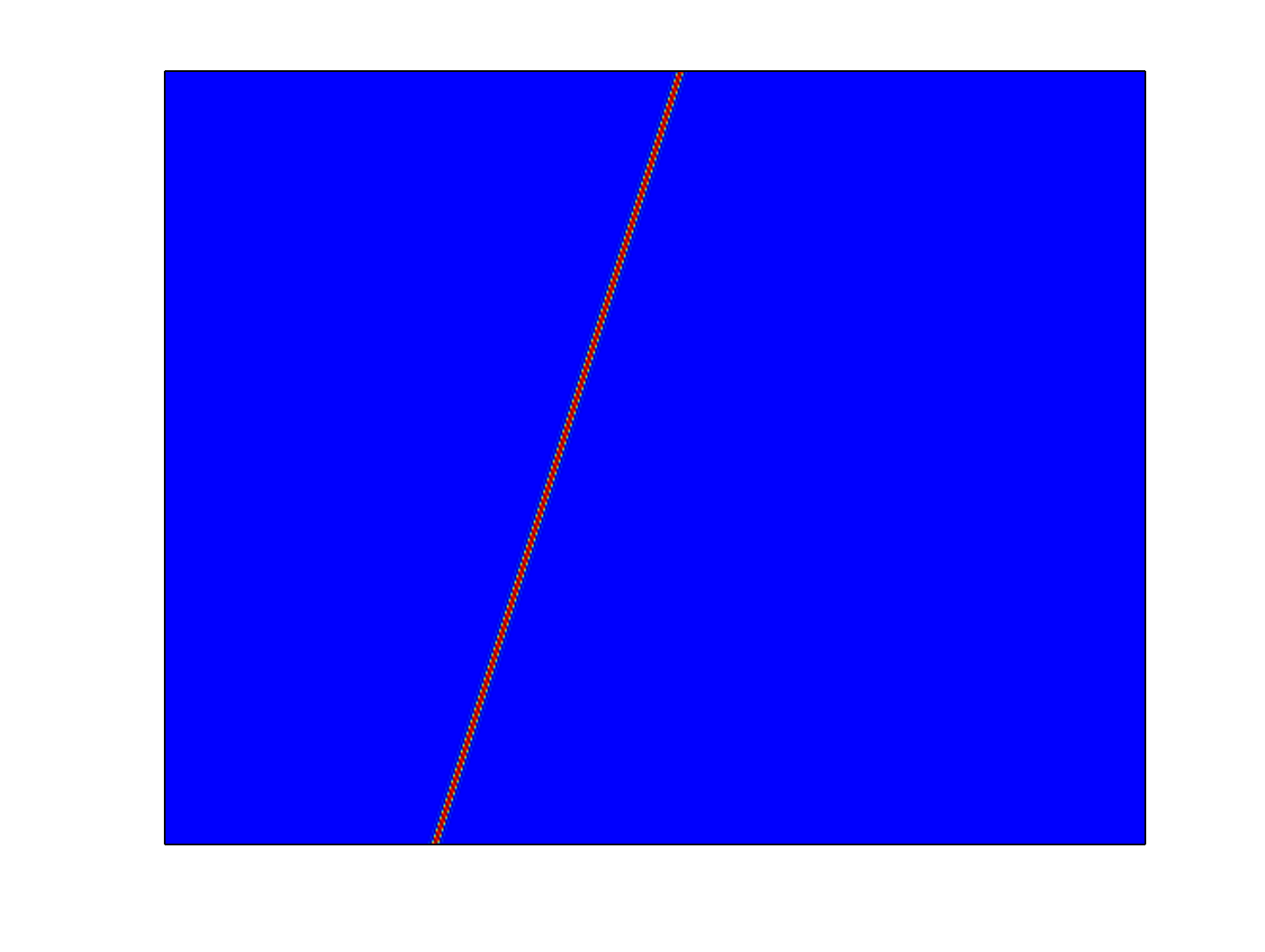}
\includegraphics[height=1.0in]{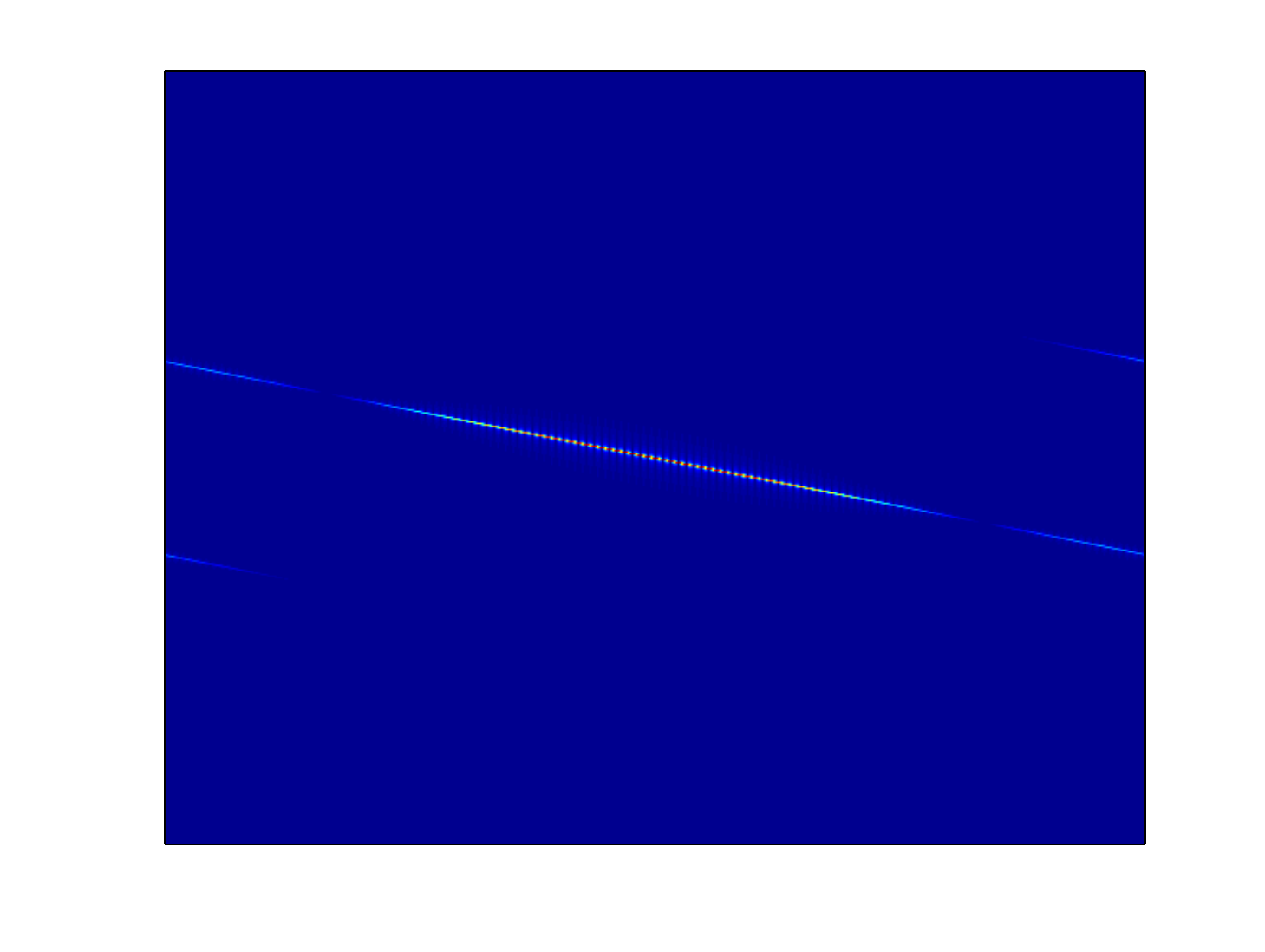}
\put(-250,0){(a)}
\put(-150,0){(b)}
\put(-50,0){(c)}
\end{center}
\caption{(a) Aliased image: DFT of ${f_c}(S_{-1/4}(n))$. (b) De-aliased image: $S^d_{-1/4}(f_d)(n)$. (c) De-aliased image: DFT of $S^d_{-1/4}(f_d)(n)$. }
\label{fig:aliase}
\end{figure}
Thus application of the digital shear operator $S^d_{2^{-j/2}k}$ allows a faithful digitization of the shearing operator associated with
the shear matrix $S_{2^{-j/2}k}$.

\subsubsection{Redundancy}
\label{subsec:redundancy}

One of the main issues which practical applicability requires is controllable redundancy.  To quantify the redundancy of the discrete
shearlet transform, we assume that the input data $f$ is a finite linear combination of translates of a 2D scaling function $\phi$ at
scale $J$ as follows:
\[
f(x) = \sum_{n_1=0}^{2^J-1}\sum_{n_2=0}^{2^J-1} d_n \phi(2^J x - n)
\]
as it was already the hypothesis in \eqref{eq:code2}. The redundancy -- as we view it in our analysis -- is then given by the number of
shearlet elements necessary to represent $f$. Furthermore, to state the result in more generality, we allow an arbitrary sampling matrix
$M_c = \text{diag}(c_1,c_2)$ for translation, i.e., consider shearlet elements of the form
\[
\psi_{j,k,m}(\cdot) = 2^{\frac{3}{4}j}\psi(S_kA_{2^j}\cdot - M_cm).
\]
We then have the following result.
\begin{proposition}[\cite{Lim2010}]
The redundancy of the DSST is
\[
\Bigl( \frac{4}{3}\Bigr)\Bigl( \frac{1}{c_1c_2}\Bigr).
\]
\end{proposition}

\begin{proof}
For this, we first consider shearlet elements for the horizontal cone for a fixed scale $j \in \{0, \dots, J-1\}$. We observe that
there exist $2^{j/2+1}$ shearing indices $k$ and  $2^j \cdot 2^{j/2} \cdot (c_1c_2)^{-1}$ translation indices associated with the
scaling matrix $A_{2^j}$ and the sampling matrix $M_c$, respectively. Thus, $2^{2j+1}(c_1c_2)^{-1}$ shearlet elements from the horizontal
cone are required for representing $f$. Due to symmetry reasons, we require the same number of shearlet elements from the vertical cone.
Finally, about $c_1^{-2}$ translates of the scaling function $\phi$ are necessary at the coarsest scale $j = 0$.

Summarizing, the total number of necessary shearlet elements across all scales is about
\[
\Bigl(\frac{4}{c_1c_2}\Bigr)\Bigl(\sum_{j=0}^{J-1}2^{2j}+1\Bigr) = \Bigl( \frac{4}{c_1c_2} \Bigr)\Bigl( \frac{2^{2J}+2}{3} \Bigr)
\]
The redundancy of each shearlet frame can now be computed as the ratio of the number of coefficients $d_n$ and this number.
Letting  $J \rightarrow \infty$ proves the claim. \qed
\end{proof}

As an example, choose a translation grid with parameters $c_1 = 1$ and $c_2 = 0.4$. Then the associated DSST has asymptotic redundancy $10/3$.

\subsubsection{Computational Complexity}

A further essential characteristics is the computational complexity (see also Subsection \ref{subsec:speed}), which we now
formally compute for the discrete shearlet transform.

\begin{proposition}[\cite{KKL2010}]
\label{prop:running}
The computational complexity of the DSST is
\[
O(2^{\log_2(1/2(L/2-1))}L\cdot N).
\]
\end{proposition}

\begin{proof}
Analyzing Steps 1 -- 5 from Subsection \ref{subsubsec:algorithm}, we observe that the most time consuming step is the computation
of the scaling coefficients in Steps 1 -- 4 for the finest scale $j = J$. This step requires 1D upsampling by a factor of $2^{j/2}$
followed by 1D convolution for each direction associated with the shear parameter $k$. Letting $L$ denote the total number of
directions at the finest scale $j = J$, and $N$ the size of 2D input data, the computational complexity for computing the scaling
coefficients in Steps 1 -- 4 is $O(2^{j/2}L\cdot N)$.
The complexity of the discrete separable wavelet transform associated with $A_{2^j}$ for Step 5 requires $O(N)$ operations,
wherefore it is negligible. The claim follows from the fact that $L = 2(2\cdot2^{j/2}+1)$. \qed
\end{proof}

It should be noted that the total computational cost depends on the number $L$ of shear parameters at the finest scale $j = J$, and
this total cost grows approximately by a factor of $L^2$ as $L$ is increased. It should though be emphasized that $L$ can be chosen
in such a way that this shearlet transform is favorably comparable to other redundant directional transforms with respect to running time as
well as performance. A reasonable number of directions at the finest scale is $6$, in which case the constant factor $2^{\log_2(1/2(L/2-1))}$
in Proposition \ref{prop:running} equals $1$. Hence in this case the running time of this shearlet transform is only about 6 times slower
than the discrete orthogonal wavelet transform, thereby remains in the range of the running time of other directional transforms.

\subsubsection{Inverse DSST}

In Subsection \ref{subsubsec:tight}, we already discussed that this transform is not an isometry, wherefore the adjoint cannot be used as an inverse transform.
However, the `good' ratio of the frame bounds in the sense as detailed in Subsection \ref{subsubsec:tight} leads to a fast convergence rate of iterative methods
such as the conjugate gradient method. Let us mention that using the conjugate gradient method basically requires computing the
forward DSST and its adjoint, and we refer to \cite{M099} and also Subsection \ref{subsubsec:adjoint} for more details.

\subsection{Digital Non-Separable Shearlet Transform (DNST)}
\label{subsec:DNST}

In this section, we describe an alternative approach to derive a faithful {digitization} of a discrete shearlet transform associated with
compactly supported shearlets. This algorithmic realization, which was developed in \cite{Lim2011}, resolves the following drawbacks of the DSST:
\begin{itemize}
\item Since this transform is not based on a tight frame, an additional computational effort is necessary to approximate the
inverse of the shearlet transform by iterative methods.
\item Computing the interpolated sampling values in \eqref{eq:dshear} requires additional computational costs.
\item This shearlet transform is not shift-variant, even when downsampling associated with $A_{2^j}$ is omitted.
\end{itemize}
We emphasize that although this alternative approach resolves these problems, the algorithm DSST provides a much
more faithful {digitization} in the sense that the shearlet coefficients can be exactly computed in this framework.

The main difference between DSST and DNST will be to exploit {\em non-separable} shearlet generators, which give more flexibility.

\subsubsection{Shearlet Generators}

We start by introducing the {\em non-separable} shearlet generators utilized in DNST.
First,  
define the shearlet generator $\psi^{\text{non}}$ by
\[
{\hat{\psi}^{\text{non}} (\xi) = P(\xi_1,2\xi_2)\hat \psi(\xi)},
\]
where 
the trigonometric polynomial $P$ is
a 2D fan filter (c.f. \cite{DV05}). For an illustration of $P$ we refer to Fig. \ref{fig:nonsupp}(a).
This in turn defines shearlets $\psi^{\text{non}}_{j,k,m}$ generated by non-separable generator functions $\psi_j^{\text{non}}$
for each scale index $j \ge 0$ by setting
$$
{\psi^{\text{non}}_{j,k,m}(x) = 2^{\frac{3}{4}j}\psi^{\text{non}}(S_kA_{2^j}x-M_{c_j}m)},
$$
where $M_{c_j}$ is a sampling matrix given by $M_{c_j} = \text{diag}(c_1^j,c_2^j)$ and
$c^j_1$ and $c^j_2$ are sampling constants for translation.

One major advantage of these shearlets $\psi^{\text{non}}_{j,k,m}$ is the fact that a fan filter enables refinement of the
directional selectivity in frequency domain at each scale. Fig. \ref{fig:nonsupp}(a)-(b) show the refined essential support of
$\hat{\psi}^{\text{non}}_{j,k,m}$ as compared to shearlets $\psi_{j,k,m}$ arising from a separable generator as in
Subsection \ref{subsubsec:digitizationOfCSST}.
\begin{figure}[h]
\begin{center}
\includegraphics[width=1.5in]{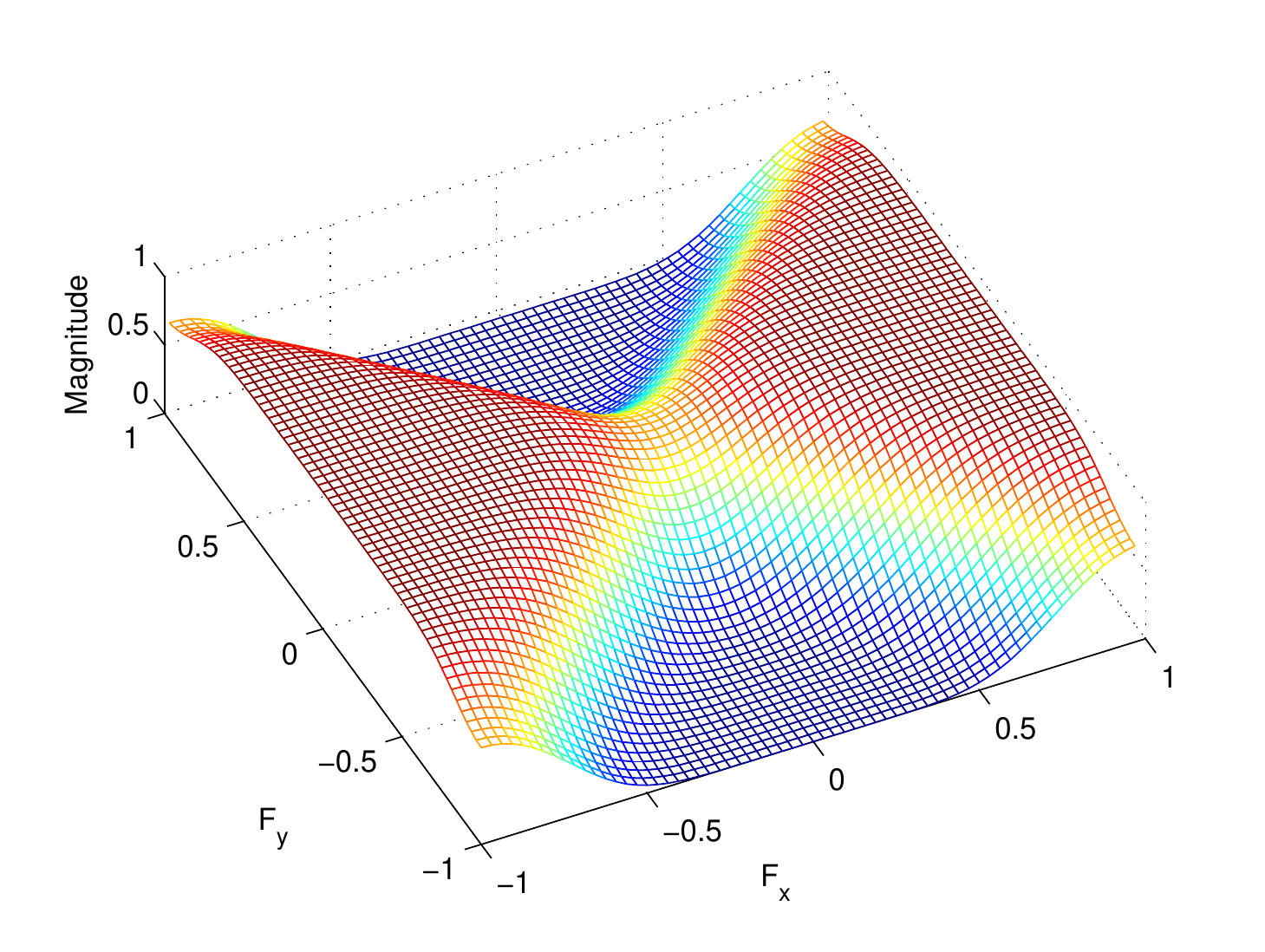}
\includegraphics[width=1.5in]{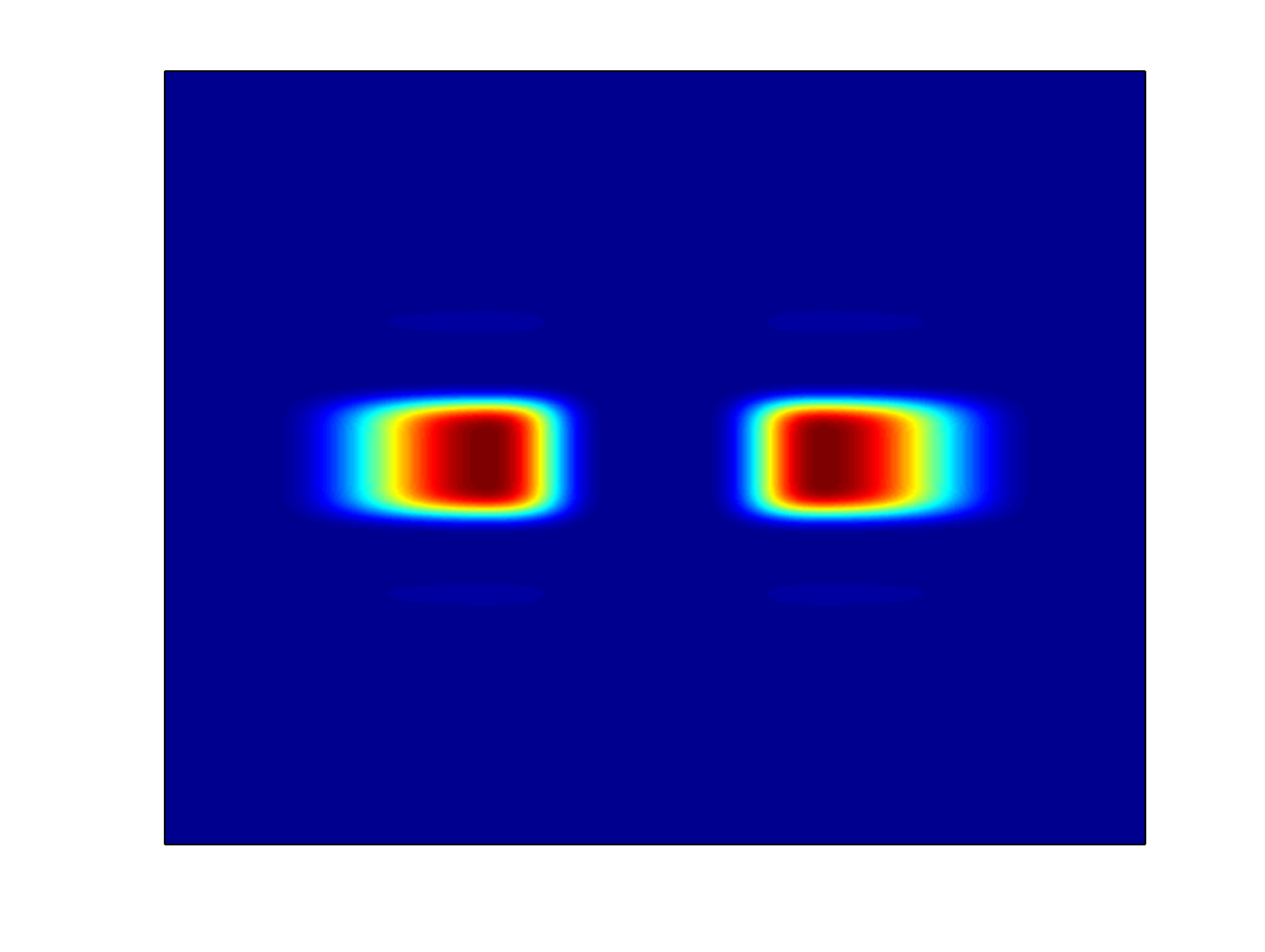}
\includegraphics[width=1.5in]{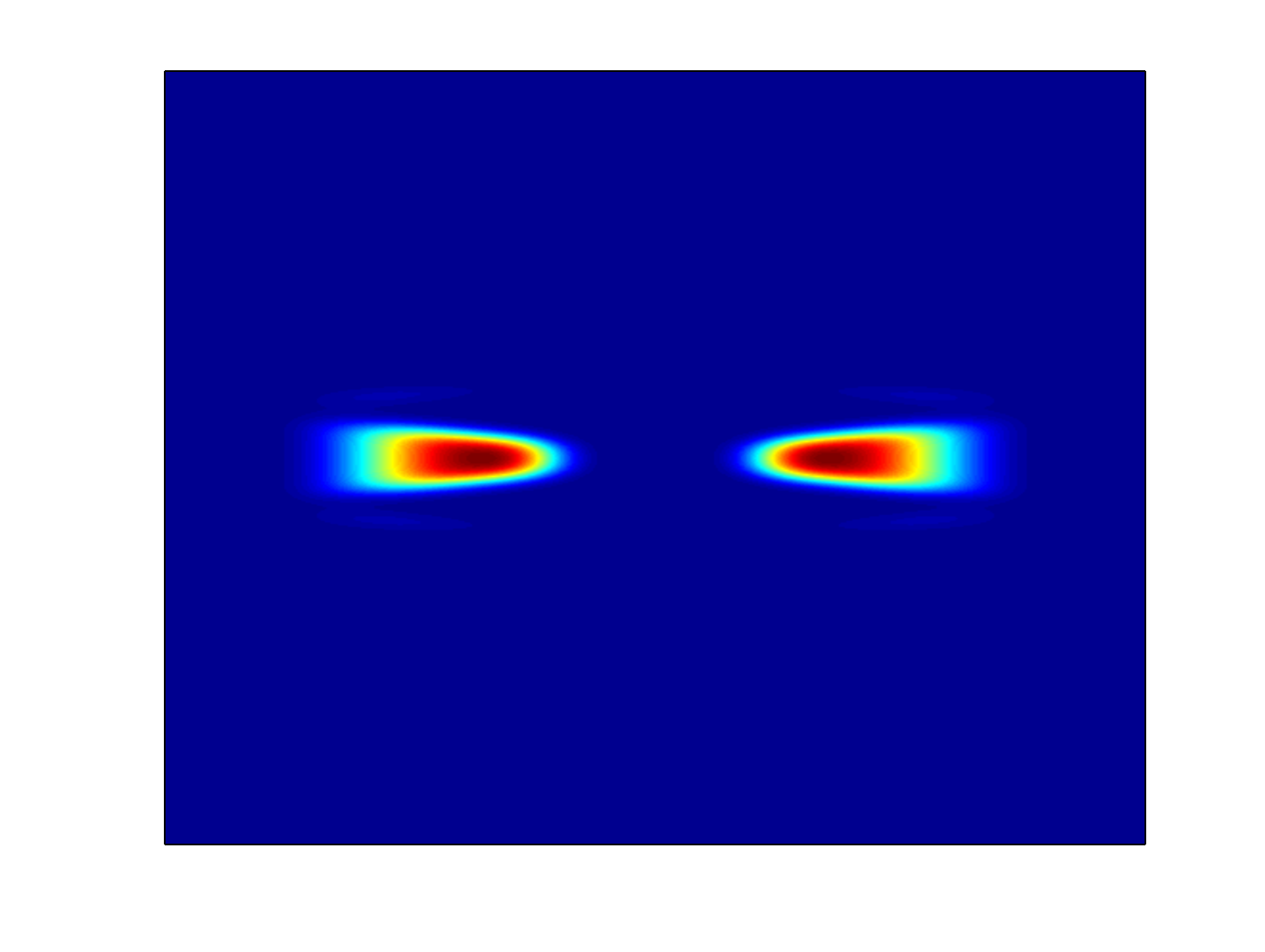}
\put(-280,0){(a)} \put(-170,0){(b)} \put(-50,0){(c)}
\end{center}
\caption{(a) Magnitude response of 2D fan filter. (b) Non-separable shearlet $\psi^{\text{non}}_{j,k,m}.$ (c) Separable shearlet $\psi_{j,k,m}$.}
\label{fig:nonsupp}
\end{figure}

\subsubsection{Algorithmic Realization}

Next, our aim is to derive a digital formulation of the shearlet coefficients $\langle f,\psi^{\text{non}}_{j,k,m}\rangle$
for a function $f$ as given in \eqref{eq:code2}. We will only discuss the case of shearlet coefficients associated with $A_{2^j}$
and $S_k$; the same procedure can be applied for $\tilde{A}_{2^j}$ and $\tilde{S}_k$ except for switching the order of variables
$x_1$ and $x_2$.

In Subsection \ref{subsec:direction}, we discretized a sheared function $f(S_{2^{-j/2}k}\cdot)$ using the digital shear operator
$S^d_{2^{-j/2}k}$ as defined in \eqref{eq:dshear}. In this implementation, we walk a different path. We digitalize the shearlets
$\psi^{\text{non}}_{j,k,m}(\cdot) = \psi^{\text{non}}_{j,0,m}(S_{2^{-j/2}k}\cdot)$ by combining multiresolution analysis and
digital shear operator $S^d_{2^{-j/2}k}$ to digitize the wavelet $\psi^{\text{non}}_{j,0,m}$ and the shear operator $S_{2^{-j/2}k}$,
respectively. This yields digitized shearlet filters of the form
\[
\psi^{\text{d}}_{j,k}(n) = S^d_{2^{-j/2}k}\Bigl(p_{J-j/2}*w_{j}\Bigr)(n),
\]
where $w_j$ is the 2D separable wavelet filter defined by $w_j(n_1,n_2) = g_{J-j}(n_1)\cdot$ $h_{J-j/2}(n_2)$ and $p_{J-j/2}(n)$ are
the Fourier coefficients of the 2D fan filter given by
{$P(2^{J-j}\xi_1,2^{J-j/2+1}\xi_2)$.}
The DNST associated with the non-separable shearlet generators $\psi^{\text{non}}_j$
is then given by
\[
DNST_{j,k}(f_J)(n) = (f_J*\overline{\psi}^{\text{d}}_{j,k})(2^{J-j}c^j_1n_1,2^{J-j/2}c^j_2n_2), \quad \text{for}\,\, f_J \in \ell^2(\Z^2).
\]
We remark that the discrete shearlet filters $\psi^{\text{d}}_{j,k}$ are computed by using a similar ideas as in Subsection
\ref{subsubsec:digitizationOfCSST}. As before, those filter coefficients can be precomputed to avoid additional computational effort.

Further notice that by setting $c^j_1 = 2^{j-J}$ and $c^j_2 = 2^{j/2-J}$, the DNST simply becomes a 2D convolution. Thus,
in this case, DNST is shear invariant.

\subsubsection{Inverse DNST}

In case that $c^j_1 = 2^{j-J}$ and $c^j_2 = 2^{j/2-J}$, the dual shearlet filters $\tilde{\psi}^{\text{d}}_{j,k}$ can be easily computed
by,
\wl{
\[
\hat{\tilde{\psi}}^{\text{d}}_{j,k}(\xi) = \frac{\hat{\psi}^{\text{d}}_{j,k}(\xi)}{\sum_{j,k} |\hat{\tilde{\psi}}^{\text{d}}_{j,k}(\xi)|^2}
\]
}
and we obtain the reconstruction formula
\[
f_J = \sum_{j,k} (f_J*\overline{\psi}^{\text{d}}_{j,k})*\tilde{\psi}^{\text{d}}_{j,k}.
\]
{This reconstruction is stable due to the existence of positive upper and lower bounds of $\sum_{j,k} |\hat{\tilde{\psi}}^{\text{d}}_{j,k}(\xi)|^2$,
which is implied by the frame property of nonseparable shearlets $\psi^{\text{non}}_{j,k,m}$, see \cite{Lim2011}.}
Thus, no iterative methods are required for the inverse DNST.
The frequency response of a discrete shearlet filter $\psi^{\text{d}}_{j,k}$ and its dual $\tilde{\psi}^{\text{d}}_{j,k}$ is illustrated
in Fig. \ref{fig:dual}. We observe that primal and dual shearlet filters behave similarly in the sense that both of filters are very well
localized in frequency.
\begin{figure}[h]
\begin{center}
\includegraphics[width=1.5in]{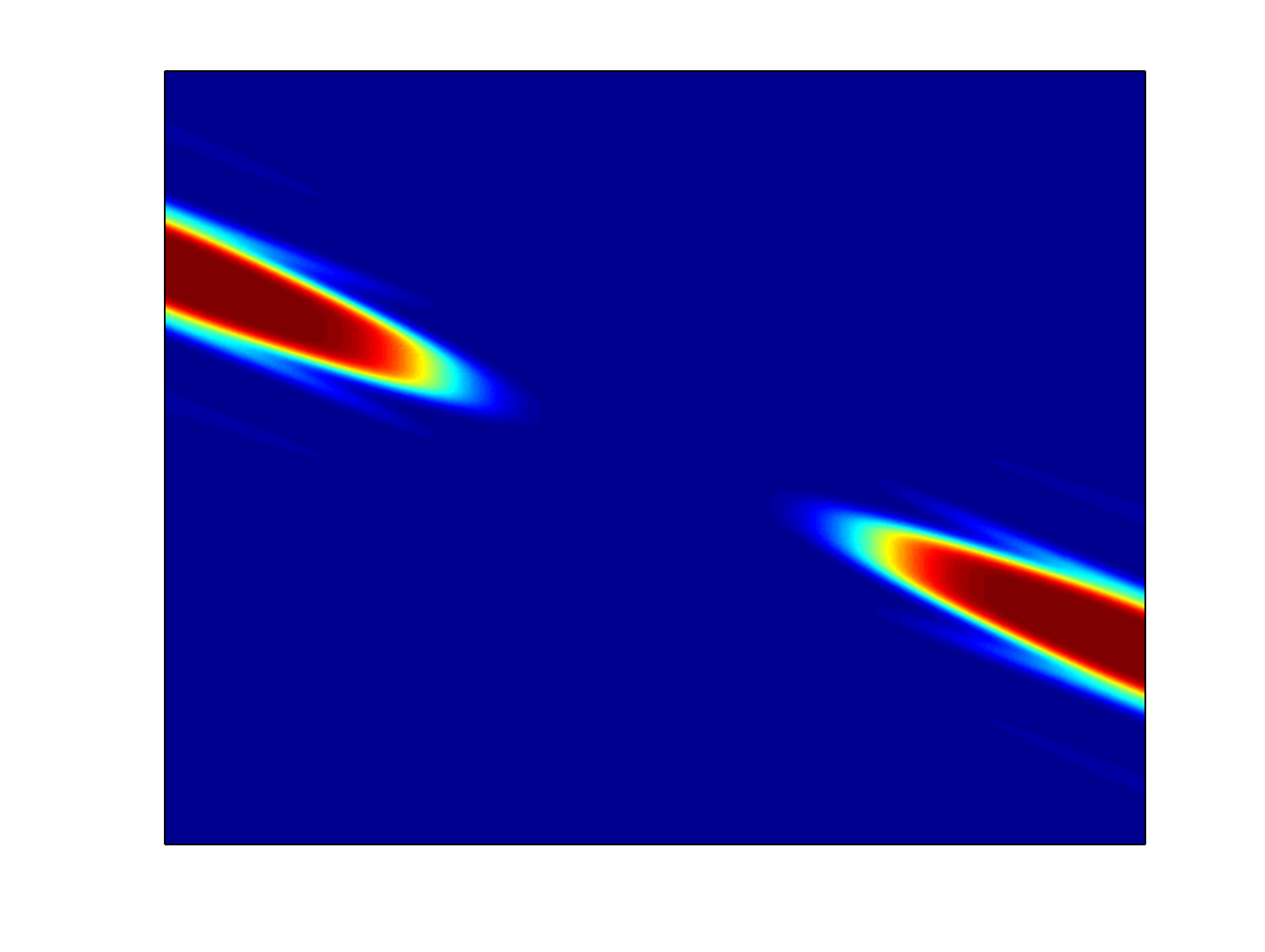}
\includegraphics[width=1.5in]{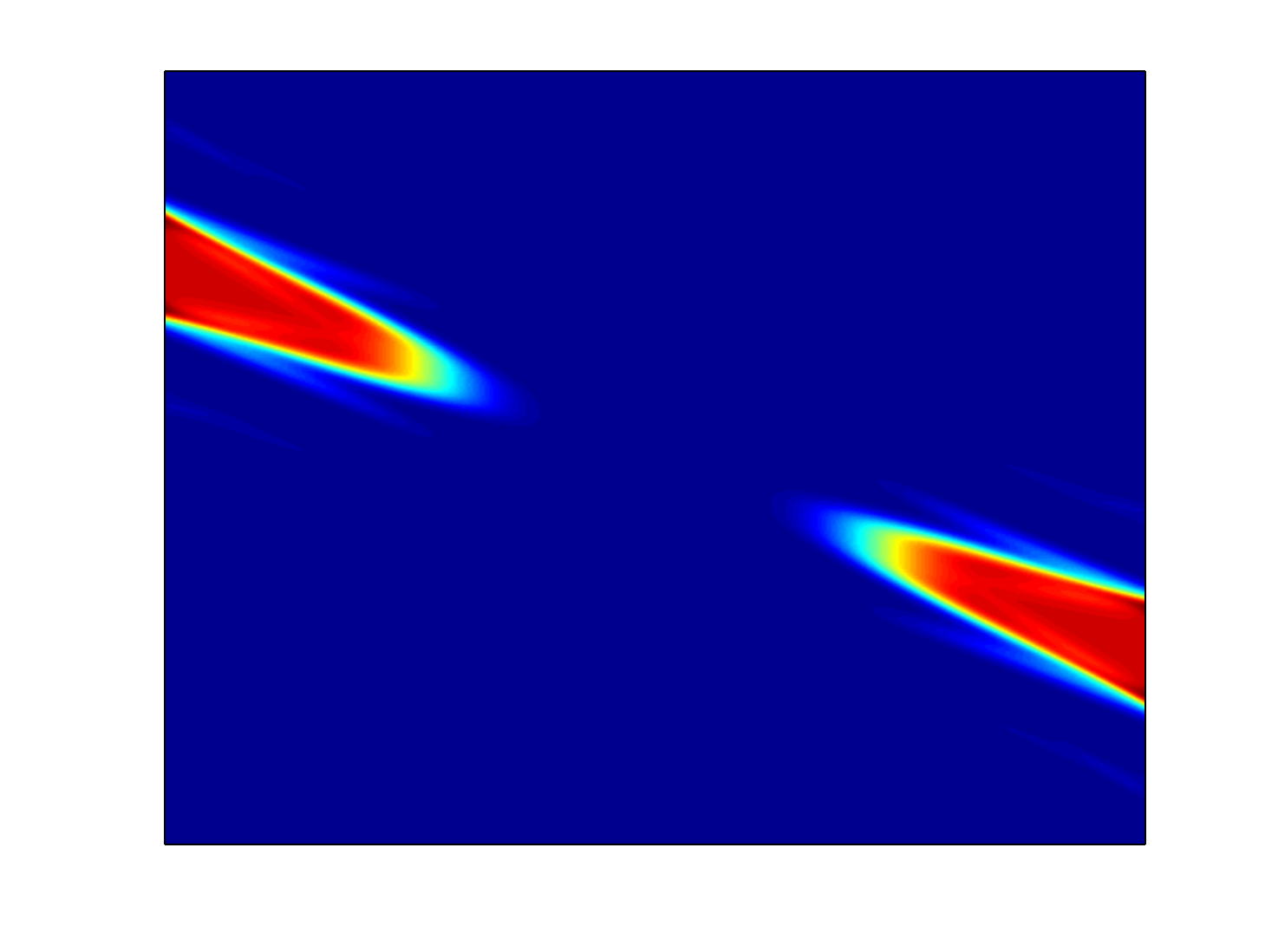}
\end{center}
\caption{Magnitude response of shearlet filter $\psi^{\text{d}}_{j,k}$ and its dual filter $\tilde{\psi}^{\text{d}}_{j,k}$.}
\label{fig:dual}
\end{figure}

\section{Framework for Quantifying Performance}
\label{sec:framework}

We next present the framework for quantifying performance of implementations of directional transforms, which
was originally introduced in \cite{KSZ11,DKSZ11}. This set of test measures was designed to analyze particular well-understood
properties of a given algorithm, which in this case are the desiderata proposed at the beginning of this chapter. This framework
was moreover introduced to serve as a tool for tuning the parameters of an algorithm in a rational way and as an
objective common ground for comparison of different algorithms. The performance of the three algorithms FDST, DSST, and DNST
will then be tested with respect to those measures. This will give us insight into their behavior with respect to the analyzed
characteristics, and also allow a comparison. However, the test values of these three algorithms will also show the delicateness
of designing such a testing framework in a fair manner, since due to the complexity of algorithmic realizations it is highly
difficile to do each aspect of an algorithm justice. It should though be emphasized that -- apart from being able to
rationally tune parameters -- such a framework of quantifying performance is essential for an objective judgement of
algorithms.
The codes of all measures are available in \url{ShearLab}.

In the following, $S$ shall denote the transform under consideration, $S^\star$ its adjoint, and, if iterative reconstruction is tested, $G_{A}J$
shall stand for the solution of the matrix problem $AI=J$ using the conjugate gradient method with residual error set to be
$10^{-6}$. Some measures apply specifically to transforms utilizing the pseudo-polar grid, for which purpose we introduce the
notation $P$ for the pseudo-polar Fourier transform, $w$ shall denote the weighting applied to the values on the pseudo-polar grid,
and $W$ shall be the windowing with additional 2D iFFT.

\subsection{Algebraic Exactness}

We require the transform to be the precise implementation of a theory for digital data on a pseudo-polar grid. In addition,
to ensure numerical accuracy, we provide the following measure. This measure is designed for transforms utilizing the pseudo-polar
grid.

\begin{measure}
Generate a sequence of $5$ (of course, one can choose any reasonable integer other than 5) random images $I_1,\ldots,I_{5}$ on a pseudo-polar
grid for $N=512$ and $R=8$ with standard normally distributed entries. Our quality measure will then
be the Monte Carlo estimate for the operator norm $\|W^\star W -  {\Id} \|_{op}$  given by
\[
M_{alg} = \max_{i=1,\ldots,5} \frac{\|W^\star W I_i - I_i \|_{2}}{\|I_i\|_2}.
\]
\end{measure}

This measure applies to the FDST -- not to the DSST or DNST -- , for which we obtain
\[
M_{alg} = 6.6E-16.
\]
This confirms that the windowing in the FDST is indeed up to machine precision a Parseval frame, which was
already theoretically confirmed by Theorem \ref{thm:DSHtight}.

\subsection{Isometry of Pseudo-Polar Transform}
\label{subsec:isometry}

We next test the pseudo-polar transform itself which might be used in the algorithm under consideration.
For this, we will provide three different measures, each being designed to test a different aspect.

\begin{measure}
\bitem
\item {\rm Closeness to isometry}. Generate a sequence of $5$ random images $I_1,\ldots,I_{5}$
of size $512 \times 512$ with standard uniformly distributed entries. Our quality measure will then
be the Monte Carlo estimate for the operator norm $\|P^\star w P -  {\Id} \|_{op}$  given by
\[
M_{isom_1} = \max_{i=1,\ldots,5} \frac{\|P^\star w P I_i - I_i \|_{2}}{\|I_i\|_2}.
\]
\item {\rm Quality of preconditioning}. Our quality measure will be
the spread of the eigenvalues of the Gram operator $P^\star w P$
given by
\[
M_{isom_2} = \frac{\lambda_{\max}(P^\star w P)}{\lambda_{\min}(P^\star w P)}.
\]
\item {\rm Invertibility}.  Our quality measure will be the Monte Carlo estimate for the
invertibility of the operator $\sqrt{w} P$ using conjugate gradient method $G_{\sqrt{w}P}$ (residual error
is set to be $10^{-6}$, here $G_{A}J$ means solving matrix problem $AI=J$ using conjugate gradient method) given by
\[
M_{isom_3} = \max_{i=1,\ldots,5} \frac{\|G_{\sqrt{w}P} \sqrt{w} P I_i - I_i \|_{2}}{\|I_i\|_2}.
\]
\eitem
\end{measure}

This measure applies to the FDST -- not to the DSST or DNST -- , for which we obtain the following numerical results,
see Table \ref{tab:1}.

\begin{table}[ht]
\caption{The numerical results for the test on isometry of the pseudo-polar transform.}
\label{tab:1}
\begin{center}
\begin{tabular}{p{2.7cm}p{2.7cm}p{2.7cm}p{3.1cm}}
\hline\noalign{\smallskip}
& $M_{isom_1}$  &  $M_{isom_2}$  & $M_{isom_3}$\\
\noalign{\smallskip}\hline\noalign{\smallskip}
FDST & 9.3E-4  &  1.834  & 3.3E-7 \\
\noalign{\smallskip}\hline\noalign{\smallskip}
\end{tabular}
\end{center}
\end{table}

The slight isometry deficiency of $M_{isom_1}\approx$ 9.9E-4 mainly results from the isometry deficiency of the weighting.
However, for practical purposes this transform can be still considered to be an isometry allowing the utilization of the
adjoint as inverse transform.

\subsection{Parseval Frame Property}

We now test the overall frame behavior of the system defined by the transform. These measures now apply to more
than pseudo-polar based transforms, in particular, to FDST, DSST, and DNST.

\begin{measure}
Generate a sequence of $5$ random images $I_1,\ldots,I_{5}$
of size $512 \times 512$ with standard uniformly distributed entries. Our quality measure will then
two-fold:
\bitem
\item {\rm Adjoint transform}. The measure will be the Monte Carlo estimate for the operator norm $\|S^\star S -  {\Id} \|_{op}$  given by
\[
M_{tight_1} = \max_{i=1,\ldots,5} \frac{\|S^\star S I_i - I_i \|_{2}}{\|I_i\|_2}.
\]
\item {\rm Iterative reconstruction}. Using conjugate gradient $G_{\sqrt{w}P}$, our measure will be given by
\[
M_{tight_2} = \max_{i=1,\ldots,5} \frac{\|G_{\sqrt{w}P} W^\star S I_i - I_i \|_{2}}{\|I_i\|_2}.
\]
\eitem
\end{measure}

The following table, Table \ref{tab:2}, presents the performance of FDST and DNST with respect to these
quantitative measures.

\begin{table}[ht]
\caption{The numerical results for the test on Parseval property.}
\label{tab:2}
\begin{center}
\begin{tabular}{p{3.5cm}p{3.5cm}p{4cm}}
\hline\noalign{\smallskip}
&  $M_{tight_1}$  &  $M_{tight_2}$ \\
\noalign{\smallskip}\hline\noalign{\smallskip}
FDST & 9.9E-4  &  3.8E-7\\
DSST & 1.9920  & 1.2E-7 \\
DNST & 0.1829  & 5.8E-16 (with dual filters)\\
\noalign{\smallskip}\hline\noalign{\smallskip}
\end{tabular}
\end{center}
\end{table}

The transform FDST is nearly tight as indicated by the measures $M_{\text{tight}_1} = 9.9$E-4, i.e., the chosen
weights force the PPFT to be sufficiently close to an isometry for most practical purposes. If a higher accurate
reconstruction is required, $M_{\text{tight}_2} = 3.8$E-7 indicates that this can be achieved by the CG method.
As expected, $M_{\text{tight}_1} = 1.9920$ shows that DSST is not tight. Nevertheless, the CG method provides
with $M_{\text{tight}_2} = 1.2$E-7 a highly accurate approximation of its inverse.
DNST is much closer to
being tight than DSST (see $M_{\text{tight}_1} = 0.1829$). We remark that this transform -- as discussed -- does
not require the CG method for reconstruction. The value $5.8$-16 was derived by using the dual shearlet filters,
which show superior behavior.

\subsection{Space-Frequency-Localization}

The next measure is designed to test the degree to which the analyzing elements, here phrased in terms of
shearlets but can be extended to other analyzing elements, are space-frequency localized.

\begin{measure}
Let $I$ be a shearlet in a $512 \times 512$ image centered at the origin
$(257,257)$ with slope $0$ of scale $4$, i.e., $\sigma_{4,0,0}^{11}+\sigma_{4,0,0}^{12}$. Our quality measure will be four-fold:
\bitem
\item {\rm Decay in spatial domain}. We compute the decay rates $d_1,\ldots,d_{512}$
along lines parallel to the $y$-axis starting from the line $[257,\;:\;]$ and the decay rates
$d_{512}$, $\ldots$, $d_{1024}$ with $x$ and $y$ interchanged. By decay rate, for instance, for the
line $[257:512,1]$, we first compute the smallest monotone majorant $M(x,1)$, $x=257,\ldots,512$
-- note that we could also choose an average amplitude here or a different `envelope' --
for the curve $|I(x,1)|$, $x=257,\ldots,512$. Then the decay rate is defined to be the average slope
of the line, which is a least square fit to the curve $\log(M(x,1))$, $x=257,\ldots,512$.
Based on these decay rates, we choose our measure to be the average of the decay rates
\[
M_{decay_1} = \frac{1}{1024}\sum_{i=1,\ldots,1024} d_i.
\]
\item {\rm Decay in frequency domain}. Here we intend to check whether the Fourier transform of $I$
is compactly supported and also the decay. For this, let $\hat{I}$ be the 2D-FFT of $I$ and compute the
decay rates $d_i$, $i=1,\ldots,1024$ as before. Then we define the following two measures:
\bitem
\item {\rm Compactly supportedness}.
\[
M_{supp} = \frac{\max_{|u|,|v|\le 3}|\hat I(u,v)|}{\max_{u,v}|\hat I(u,v)|}.
\]
\item {\rm Decay rate}.
\[
M_{decay_2} = \frac{1}{1024}\sum_{i=1,\ldots,512} d_i.
\]
\eitem
\item {\rm Smoothness in spatial domain}. We will measure smoothness by the average of local H\"older regularity.
For each $(u_0,v_0)$, we compute $M(u,v)= |I(u,v)-I(u_0,v_0)|$, $0<\max\{|u-u_0|,|v-v_0|\}\le 4$. Then the
local H\"older regularity $\alpha_{u_0,v_0}$ is the least square fit to the curve $\log(|M(u,v)|)$.
Then our smoothness measure is given by
\[
M_{smooth_1} = \frac{1}{512^2}\sum_{u,v}\alpha_{u,v}.
\]
\item {\rm Smoothness in frequency domain}. We compute the smoothness now for $\hat{I}$, the 2D-FFT
of $I$ to obtain the new $\alpha_{u,v}$ and define our measure to be
\[
M_{smooth_2} = \frac{1}{512^2}\sum_{u,v}\alpha_{u,v}.
\]
\eitem
\end{measure}

Let us now analyze the space-frequency localization of the shearlets utilized in FDST, DSST and DNST by these
measures. The numerical results are presented in Table \ref{tab:3}.

\begin{table}[ht]
\caption{The numerical results for the test on space-frequency localization.}
\label{tab:3}
\begin{center}
\begin{tabular}{p{1.8cm}p{1.8cm}p{1.8cm}p{1.8cm}p{1.8cm}p{2.0cm}}
\hline\noalign{\smallskip}
&  $M_{decay_1}$  & $M_{supp}$ &  $M_{decay_2}$ &  $M_{smooth_1}$ & $M_{smooth_2}$\\ \noalign{\smallskip}\hline\noalign{\smallskip}
FDST & -1.920 & 5.5E-5 & -3.257 & 1.319 &  0.734 \\
DSST & $-\infty$ & 8.6E-3 & -1.195 & 0.012 & 0.954 \\
DNST & $-\infty$  & 2.0E-3 & -0.716 & 0.188 & 0.949\\
\noalign{\smallskip}\hline\noalign{\smallskip}
\end{tabular}
\end{center}
\end{table}

The shearlet elements associated with FDST are band-limited and those associated with DSST and DNST are compactly supported,
which is clearly indicated by the values derived for $M_{decay_1}$, $M_{supp}$, and $M_{decay_2}$. It would be expected that
$M_{decay_2} = - \infty$ for FDST due to the band-limitedness of the associated shearlets. The shearlet elements are however
defined by their Fourier transform on a pseudo-polar grid, whereas the measure $M_{decay_2}$ is taken after applying the 2D-FFT to the shearlets
resulting in data on a cartesian grid, in particular, yielding a non-precisely compactly supported function.

The test values for
$M_{smooth_1}$ and $M_{smooth_2}$ show that the associated shearlets are more smooth in spatial domain for FDST than
for DSST and DNST, with the reversed situation in frequency domain.

\subsection{Shear Invariance}

Shearing naturally occurs in digital imaging, and it can -- in contrast to rotation -- be precisely realized in the digital domain.
Moreover, for the shearlet transform, shear invariance can be proven and the theory implies
\[
\ip{2^{3j/2} \psi(S_k^{-1} A_4^{j} \cdot -m)}{f(S_s \cdot)} = \ip{2^{3j/2} \psi(S_{k+2^{j} s}^{-1} A_4^{j} \cdot -m)}{f}.
\]
We therefore expect to see this or a to the specific directional transform adapted behavior. The degree to which this goal is
reached is tested by the following measure.

\begin{measure}
Let $I$ be an $256 \times 256$ image with an edge through the origin $(129,129)$ of slope $0$.  Given $-1\le s\le 1$, generates an
image $I_s:=I(S_s\cdot)$ and let $S_j$ be the set of all possible scales $j$ such that $2^js\in\bZ$. Our quality measure will then
be the curve
\[
M_{shear,j} = \max_{-2^j<k,k+2^js<2^j} \frac{\|C_{j,k}(S I_s)  - C_{j,k+2^js}( S I)\|_2}{\|I\|_2}, \qquad \text{scale } j\in S_j,
\]
where $C_{j,k}$ is the shearlet coefficients at scale $j$ and shear $k$.
\end{measure}


We present our results in Table \ref{tab:4}.

\begin{table}[ht]
\caption{The numerical results for the test on shear invariance.}
\label{tab:4}
\begin{center}
\begin{tabular}{p{2.2cm}p{2.2cm}p{2.2cm}p{2.2cm}p{2.2cm}}
\hline\noalign{\smallskip}
&  $M_{shear,1}$ & $M_{shear,2}$ &$M_{shear,3}$ &$M_{shear,4}$  \\
\noalign{\smallskip}\hline\noalign{\smallskip}
FDST & 1.6E-5 & 1.8E-4 & 0.002  & 0.003 \\
\noalign{\smallskip}\hline\noalign{\smallskip}
\end{tabular}
\end{center}
\end{table}

This table shows that the FDST is indeed almost shear invariant. A closer inspection shows that $M_{shear,1}$ and $M_{shear,2}$ are
relatively small compared to the measurements with respect to finer scales $M_{shear,3}$ and $M_{shear,4}$. The reason for this is
the aliasing effect which shifts some energy to the high frequency part near the boundary away from the edge in the frequency domain.

We did not test DSST and DNST with respect to this measure, since these transforms show a different -- not included in this Measure 5
-- type of shear invariance behavior.

\subsection{Speed}
\label{subsec:speed}

Speed is one of the most fundamental properties of each algorithm to analyze. Here, we test the speed up to a size of $N=512$
which regard as sufficient to computing the complexity.

\begin{measure}
Generate a sequence of $5$ random images $I_i$, $i=5,\ldots,9$ of size $2^i \times 2^i$
with standard normally distributed entries. Let $s_i$ be the speed of the shearlet transform $S$ applied to
$I_i$.
Our hypothesis is that the speed behaves like $s_i = c \cdot (2^{2i})^d$; $2^{2i}$ being the
size of the input. Let now $\tilde{d}_a$ be the average slope of the line,
which is a least square fit to the curve $i \mapsto \log(s_i)$.
Let also $f_i$ be the 2fft applied to $I_i$, $i=5,\ldots,9$. Our quality measure will then be three-fold:
\bitem
\item {\rm Complexity}.
\[
M_{speed_1} = \frac{\tilde{d}_a}{2\log 2}.
\]
\item {\rm The Constant}.
\[
M_{speed_2} = \frac15 \sum_{i=5}^{9} \frac{s_i}{(2^{2i})^{M_{speed,1}}}.
\]
\item {\rm Comparison with 2D-FFT}.
\[
M_{speed_3} = \frac15 \sum_{i=5}^{9} \frac{s_i}{f_i}.
\]
\eitem
\end{measure}

Table \ref{tab:6} presents the results of testing FDST, DSST and DNST with respect to these speed measures.

\begin{table}[ht]
\caption{The numerical results for the test on speed.}
\label{tab:6}
\begin{center}
\begin{tabular}{p{3cm}p{3cm}p{3cm}p{2.2cm}}
\hline\noalign{\smallskip}
&  $M_{speed_1}$ & $M_{speed_2}$ & $M_{speed_3}$ \\
\noalign{\smallskip}\hline\noalign{\smallskip}
FDST & 1.156 & 9.3E-6 & 280.560 \\
DSST & 0.821 & 4.5E-3 & 88.700 \\
DNST & 1.081 & 9.9E-8 & 40.519\\
\noalign{\smallskip}\hline\noalign{\smallskip}
\end{tabular}
\end{center}
\end{table}

To interpret these results correctly, we remark that the DNST was tested only with test images $I_i$ for $i=7,\dots,9$, since it
can not be implemented for small size images. Interestingly, the results also show that the 2D FFT based convolution makes DNST
comparable to DSST with respect to these speed measures, although it is much more redundant than DSST. Finally, the results
show that FDST is comparable with both DSST and DNST with respect to complexity measure $M_{\text{speed}_1}$. From this, it is
conceivable to assume that FDST is highly comparable with respect to speed for large scale computations. The larger
value $M_{speed_3} = 280.560$ appears due to the fact that the FDST employs fractional Fourier transforms
on an oversampled pseudo-polar grid of size.

\subsection{Geometric Exactness}

One major advantage of directional transforms is their sensitivity with respect to geometric features alongside with
their ability to sparsely approximate those (cf. Chapter \cite{SparseApproximation}). This measure is designed to
analyze this property.

\begin{measure}
Let $I_1,\ldots, I_8$ be $256 \times 256$ images of an edge through the origin $(129,129)$ and of slope $[-1,-0.5,0,0.5,1]$
and the transpose of the middle three, and let $c_{i,j}$ be the associated shearlet coefficients
for image $I_i$ and scale $j$. Our quality measure will two-fold:
\bitem
\item {\rm Decay of significant coefficients}.
Consider the curve
\[
\frac18 \sum_{i=1}^8 \max{|c_{i,j} \text{(of analyzing elements aligned with the line)}|}, \qquad \text{scale } j,
\]
let $d$ be the average slope of the line, which is a least square fit to $\log$ of this curve, and define
\[
M_{geo_1} = d.
\]
\item {\rm Decay of insignificant coefficients}.
Consider the curve
\[
\frac18 \sum_{i=1}^8 \max{|c_{i,j} \text{(of all other analyzing elements)}|}, \qquad \text{scale } j,
\]
let $d$ be the average slope of the line, which is a least square fit to $\log$ of this curve, and define
\[
M_{geo_2} = d.
\]
\eitem
\end{measure}

Table \ref{tab:5} shows the numerical test results for FDST, DSST, and  DNST.

\begin{table}[ht]
\caption{The numerical results for the test on geometric exactness.}
\label{tab:5}
\begin{center}
\begin{tabular}{p{3.5cm}p{3.5cm}p{4.0cm}}
\hline\noalign{\smallskip}
&  $M_{geo_1}$ & $M_{geo_2}$ \\
\noalign{\smallskip}\hline\noalign{\smallskip}
FDST & -1.358 & -2.032 \\
DSST & -0.002 & -0.030\\
DNST & -0.019 & -0.342\\
\noalign{\smallskip}\hline\noalign{\smallskip}
\end{tabular}
\end{center}
\end{table}

As expected, the decay rate of the insignificant shearlet coefficients of FDST, i.e.,
the ones not aligned with the line singularity, measured by $M_{geo_2}\approx$ -2.032 is much larger than the decay rate of the significant
shearlet coefficients measured by $M_{geo_1}\approx$ -1.358. Notice that this difference is even more significant in the case of
the DSST and DNST.

\subsection{{Stability}}

To analyze {stability} of an algorithm, we choose thresholding as the most common impact on a sequence of transform coefficients.

\begin{measure}
Let $I$ be the regular sampling of a Gaussian function with mean 0 and variance $256$
on $\{-128,127\}^2$ generating an $256 \times 256$-image.
\bitem
\item {\rm Thresholding 1}. Our first quality measure will be the curve
\[
M_{thres_{1,p_1}} = \frac{\|G_{\sqrt{w}P} W^\star \; {\rm thres}_{1,p_1} \, S I  - I\|_2}{\|I\|_2},
\]
where ${\rm thres}_{1,p_1}$ discards $100 \cdot (1-2^{-p_1})$ percent of the coefficients ($p_1 = [2:2:10]$).
\item {\rm Thresholding 2}. Our second quality measure will be the curve
\[
M_{thres_{2,p_2}} = \frac{\|G_{\sqrt{w}P} W^\star \; {\rm thres}_{2,p_2} \, S I  - I\|_2}{\|I\|_2},
\]
where ${\rm thres}_{2,p_2}$ sets all those coefficients to zero with absolute values below
the threshold $m(1-2^{-p_2})$ with $m$ being the maximal absolute value of all coefficients. ($p_2 = [0.001:0.01:0.041]$)
\eitem
\end{measure}

Table \ref{tab:7} shows that even if we discard $100(1-2^{-10}) \sim 99.9\%$ of the FDST coefficients, the original image
is still well approximated by the reconstructed image. Thus the number of the significant coefficients is relatively small
compared to the total number of shearlet coefficients. From Table \ref{tab:8}, we note that knowledge of the shearlet
coefficients with absolute value greater than $m(1-1/2^{0.001}) (\sim 0.1\%$ of coefficients) is sufficient for precise
reconstruction.

DNST shows a similar behavior with worse values for relatively large $p_1$. It should be however emphasized that firstly,
the redundancy of DNST used in this test is 25 and this is lower than the redundancy of FDST, which is about 71. This effect
can be more strongly seen by the test results of DSST whose redundancy with 4 even much smaller. Secondly, a significant
part of the low frequency coefficients in both DSST and DNST will be removed by a relatively large threshold, since the
ratio between the number of the low frequency coefficients and the total number of coefficients is much higher than FDST.
This prohibits a similarly good reconstruction of a Gaussian function.


This test in particular shows the delicateness of comparing different algorithms by merely looking at the test values without
a rational interpretation; in this case, without considering the redundancy and the ratio between the number of the low frequency
coefficients and the total number of coefficients.

\begin{table}[ht]
\caption{The numerical results for $M_{thres_{1,p_1}}$.}
\label{tab:7}
\begin{center}
\begin{tabular}{p{1.8cm}p{1.8cm}p{1.8cm}p{1.8cm}p{1.8cm}p{2.0cm}}
\hline\noalign{\smallskip}
$p_1$ &  2 & 4 & 6 & 8 & 10 \\
\noalign{\smallskip}\hline\noalign{\smallskip}
FDST &1.5E-08&7.2E-08&2.5E-05 & 0.001  & 0.007 \\
DSST & 0.02961 & 0.02961 & 0.02961 & 0.0296 & 0.0331\\
DNST & 5.2E-10 & 1.2E-04 & 0.00391 & 0.0124 & 0.0396\\
DNST+DWT &  1.2E-09 & 3.2E-06 & 3.4E-05 & 5.5E-05 & 5.5E-05\\
\noalign{\smallskip}\hline\noalign{\smallskip}
\end{tabular}
\end{center}
\end{table}

\begin{table}[ht]
\caption{The numerical results for $M_{thres_{2,p_2}}$.}
\label{tab:8}
\begin{center}
\begin{tabular}{{p{1.8cm}p{1.8cm}p{1.8cm}p{1.8cm}p{1.8cm}p{2.0cm}}}
\hline\noalign{\smallskip}
$p_2$ &  0.001 & 0.011 & 0.021 & 0.031 & 0.041 \\ \noalign{\smallskip}\hline\noalign{\smallskip}
FDST & 0.005 & 0.039 & 0.078 & 0.113  & 0.154\\
DSST & 0.030 & 0.036 & 0.046 & 0.056 & 0.072\\
DNST & 0.002 & 0.018 & 0.035 & 0.055 & 0.076\\
DNST+DWT & 0.001 & 0.013 & 0.020 & 0.024 & 0.031\\
\noalign{\smallskip}\hline\noalign{\smallskip}
\end{tabular}
\end{center}
\end{table}




\begin{thebibliography}{99}

\bibitem{Introduction}
Introduction

\bibitem{Applications}
Chapter on Applications.

\bibitem{ShearletMRA}
Chapter on ShearletMRA.

\bibitem{SparseApproximation}
Chapter on Sparse Approximation.

\bibitem{ACDIS08}
A. Averbuch, R. R. Coifman, D. L. Donoho, M. Israeli, and Y. Shkolnisky,
{\em A framework for discrete integral transformations I -- the pseudo-polar Fourier transform},
SIAM J. Sci. Comput. {\bf 30} (2008), 764--784.

\bibitem{Bailey:frFT}
D. H. Bailey and P. N. Swarztrauber, {\em The fractional Fourier transform and applications}, SIAM Review, {\bf 33} (1991), 389--404.

\bibitem{CDDY06}
E. J. Cand\`{e}s, L. Demanet, D. L. Donoho and L. Ying,
{\em  Fast discrete curvelet transforms,}
Multiscale Model. Simul. {\bf 5} (2006), 861--899.

\bibitem{CD99}
E. J. Cand\`{e}s and D. L. Donoho,
{\em Ridgelets: a key to higher-dimensional intermittency?,}
Phil. Trans. R. Soc. Lond. A. {\bf 357} (1999), 2495--2509.

\bibitem{CD04}
E. J.~Cand\`es and D. L.~Donoho,
{\em New tight frames of curvelets and optimal representations of objects with $C^2$ singularities},
Comm. Pure Appl. Math. {\bf 56} (2004), 219--266.

\bibitem{CD05a}
E. J.~Cand\`es and D. L.~Donoho,
{\em Continuous curvelet transform: I. Resolution of the wavefront set},
Appl. Comput. Harmon. Anal. {\bf 19} (2005), 162--197.

\bibitem{CD05b}
E. J.~Cand\`es and D. L.~Donoho,
{\em Continuous curvelet transform: II. Discretization of frames},
Appl. Comput. Harmon. Anal. {\bf 19} (2005), 198--222.

\bibitem{DV05}
M. N. Do and M. Vetterli,
{\em The contourlet transform: an efficient directional multiresolution image representation},
IEEE Trans. Image Process. {\bf 14} (2005), 2091--2106.

%

\bibitem{Don99}
D. L. Donoho,
{\em Wedgelets: nearly minimax estimation of edges},
Ann. Statist.  {\bf 27}  (1999), 859--897.

\bibitem{DKSZ11}
D. L. Donoho, G. Kutyniok, M. Shahram, and X. Zhuang,
{\em A rational design of a digital shearlet transform},
Proceeding of the 9th International Conference on Sampling Theory and Applications, Singapore, 2011.


\bibitem{DMSSU08}
D. L. Donoho, A. Maleki, M. Shahram, V. Stodden, and I. Ur-Rahman,
{\em Fifteen years of Reproducible Research in Computational Harmonic Analysis,}
preprint.


\bibitem{ELL08}
G. Easley, D. Labate, and W.-Q Lim, {\em Sparse directional image representations using the discrete shearlet transform}, Appl. Comput. Harmon. Anal.
{\bf 25} (2008), 25--46.


%
%
%
%
%
%
%
%
%
%
\bibitem{HKZ10}
B. Han, G. Kutyniok, and Z. Shen,
{\em A unitary extension principle for Shearlet Systems},
preprint.

\bibitem{HR63} E. Hewitt and K.A. Ross, {\em Abstract Harmonic Analysis I, II},
Springer-Verlag, Berlin/ Heidelberg/New York, 1963.

%

\bibitem{KKL2010}
P. Kittipoom, G. Kutyniok, and W.-Q Lim,{\em  Construction of Compactly Supported Shearlet Frames}. J. Fourier Anal. Appl., 2010, to appear.



\bibitem{KSZ11}
G. Kutyniok, M. Shahram, and X. Zhuang,
{\em ShearLab: A Rational Design of a Digital Parabolic Scaling Algorithm},
preprint.

\bibitem{KS08}
G. Kutyniok and T. Sauer,
{\em Adaptive directional subdivision schemes and Shearlet Multiresolution Analysis},
preprint.

\bibitem{Lim2010}
W.-Q Lim, {\em The Discrete Shearlet Transform: A new directional transform and compactly supported shearlet frames}, IEEE Trans. Imag. Proc. {\bf 19} (2010), 1166--1180.

\bibitem{Lim2011} {W.-Q Lim, {\em Nonseparable Shearlet Transforms}, preprint.}


\bibitem{M099} S. Mallat, {\em A Wavelet Tour of Signal Processing}, 2nd ed. New
York: Academic, 1999.




\end{thebibliography}
\end{document}